\documentclass[10pt]{amsart}

\usepackage{amssymb,verbatim,color}
\usepackage{esint}
\usepackage[colorlinks=true,urlcolor=blue, citecolor=red,linkcolor=blue,linktocpage,pdfpagelabels, bookmarksnumbered,bookmarksopen]{hyperref}
\usepackage[hyperpageref]{backref}

\newtheorem{lm}{Lemma}[section]
\newtheorem{prop}[lm]{Proposition}
\newtheorem{coro}[lm]{Corollary}
\newtheorem{teo}[lm]{Theorem}
\theoremstyle{definition}
\newtheorem{oss}[lm]{Remark}
\newtheorem{defi}[lm]{Definition}
\newtheorem{exa}[lm]{Example}
\newtheorem*{ack}{Acknowledgements}

\numberwithin{equation}{section}

\title[Higher H\"older regularity]{Higher H\"older regularity\\ for the fractional $p-$Laplacian\\ in the superquadratic case}

\author[Brasco]{Lorenzo Brasco}
\address[L.\ Brasco]{Dipartimento di Matematica e Informatica
\newline\indent
Universit\`a degli Studi di Ferrara
\newline\indent
Via Machiavelli 35, 44121 Ferrara, Italy}
\email{lorenzo.brasco@unife.it}

\author[Lindgren]{Erik Lindgren}
\address[E. Lindgren]{Department of Mathematics, Stockholm University
\newline\indent
106 91 Stockholm, Sweden}
\email{eriklin@kth.se, erik.lindgren@math.su.se}

\author[Schikorra]{Armin Schikorra}
\address[A. Schikorra]{Department of Mathematics, 301 Thackeray Hall, University of Pittsburgh
\newline\indent
PA 15260, USA}
\email{armin@pitt.edu}


\begin{document}
\begin{abstract}
We prove higher H\"older regularity for solutions of equations involving the fractional $p-$Laplacian of order $s$,
when $p\ge  2$ and $0<s<1$. In particular, we provide an explicit H\"older exponent for solutions of the non-homogeneous equation with data in $L^q$ and $q>N/(s\,p)$,
which is almost sharp whenever $s\,p\leq (p-1)+N/q$. The result is new already for the homogeneous equation.
\end{abstract}
\maketitle
\begin{center}
\begin{minipage}{10cm}
\small
\tableofcontents
\end{minipage}
\end{center}

\section{Introduction}
\subsection{Overview}
In this paper, we study the H\"older regularity of local weak solutions of the nonlinear and nonlocal elliptic equation
\begin{equation}\label{MainPDE}
(-\Delta_p)^s u=f, \qquad f\in L^q_{\rm loc}.
\end{equation}
Here $2\le p<\infty$, $0<s<1$ and $(-\Delta_p)^s$ is the {\it fractional $p-$Laplacian of order $s$}, formally defined by
\[
(-\Delta_p)^s u\, (x):=2\, \lim_{\varepsilon\to 0}\int_{\mathbb{R}^N\setminus B_\varepsilon(x)}\frac{|u(x)-u(y)|^{p-2}\,(u(x)-u(y))}{|x-y|^{N+s\,p}}\, dy.
\]
The operator $(-\Delta_p)^s$ is a nonlocal (or fractional) version of the well studied $p-$Laplacian operator, 
\[
\Delta_pu=\operatorname{div}(|\nabla u|^{p-2}\nabla u),
\]
and has in recent years attracted extensive attention. In weak form, this operator naturally arises as the first variation of {the Sobolev-Slobodecki\u{\i} seminorm for $W^{s,p}(\mathbb{R}^N)$, namely} of the nonlocal functional
\[
u\mapsto \iint_{\mathbb{R}^N\times \mathbb{R}^N} \frac{|u(x)-u(x)|^p}{|x-y|^{N+s\,p}}\,dx\,dy.
\]
We are concerned with the \emph{higher H\"older regularity} of solutions of \eqref{MainPDE}. More precisely, we prove that  if the right-hand side $f$ is in $L^q_{\rm loc}$ for 
\[
\left\{\begin{array}{lr}
q>\dfrac{N}{s\,p},& \mbox{ if } s\,p\le N,\\
&\\
q\ge 1,& \mbox{ if }s\,p>N,
\end{array}
\right.
\]
then local weak solutions (see Definition \ref{defi:localweak}) are locally $\delta-$H\"older continuous  
\[
\mbox{ for any }\delta <\min\left\{\frac{1}{p-1}\left(s\,p-\frac{N}{q}\right),\,1\right\}=:\Theta(N,s,p,q).
\] 
\begin{oss}[Homogeneous equation]
We observe that when $q=\infty$, the latter reduces to
\[
\min\left\{\frac{s\,p}{p-1},\,1\right\}.
\]
Thus this is the limit exponent we get for {\it $(s,p)-$harmonic functions}, i.e. local weak solutions of the homogeneous equation $(-\Delta_p)^s u=0$. %Actually, we will prove that these functions enjoy stronger regularity properties on the scale of Sobolev spaces, see Corollary \ref{coro:higherdiff} below. 
\end{oss}
To the best of our knowledge, this is the first result with an explicit H\"older exponent, even for the homogeneous equation. We point out that in general the exponent 
\[
\frac{1}{p-1}\left(s\,p-\frac{N}{q}\right),
\]
{\it is the sharp one}, see Examples \ref{exa:sharp} and \ref{exa:sharpinfty} below. 
This means that {\it our result is almost sharp} whenever 
\[
s\leq \frac{p-1}{p}+\frac{1}{p}\,\frac{N}{q}.
\]
and \emph{sharp} when $s =1/2$, $p=2$ and $q = \infty$.
For a more detailed description of our results we refer the reader to Section~\ref{sec:main}.
\subsection{State of the art} Equations of the type \eqref{MainPDE} were first considered in \cite{IN}, where viscosity solutions are studied. Existence, uniqueness, and the convergence to the $p-$Laplace equation as $s$ goes to $1$, are proved. The first pointwise regularity result for equations of this type is, to the best of our knowledge, \cite{DKP}, where the local H\"older regularity was proved, generalizing the celebrated ideas of De Giorgi to this nonlocal and nonlinear setting. The same authors pursued their investigation in the field, by proving Harnack's inequality for solutions of the homogeneous equation in \cite{DKP2}.
\par
There are also many other recent regularity results. In \cite{Li}, the second author studied the local H\"older regularity, using viscosity methods. In \cite{Co}, the results of \cite{DKP, DKP2} are generalized to functions belonging to a fractional analogue of the so-called {\it De Giorgi classes} (we refer to \cite[Chapter 7]{Gi} for the classical definition). This can be seen as the nonlocal counterpart of the celebrated results by Giaquinta and Giusti contained in \cite{GG}, for minima of local functionals with $p-$growth. It is worth pointing out that a pioneering use of fractional De Giorgi classes has been done in \cite{Mi}, in a local context.
\par
Finally, in \cite{IMS} H\"older regularity up to the boundary was obtained, using barrier arguments. Such a result has been recently enhanced by the same authors in \cite{IMS2}, for $p\ge 2$.
\par
Here we seize the opportunity to mention the paper \cite{KMS}, in which regularity for equations of the type \eqref{MainPDE} with a measure datum $f$ is studied. In particular, in \cite[Corollary 1.2]{KMS} the authors obtained the remarkable (sharp) result
\[
u\in C^0\qquad \mbox{ whenever } f\in L^{\frac{N}{s\,p},\frac{1}{p-1}}_{\rm loc},
\]
where the latter is a Lorentz space.
There has also been some recent progress in terms of {\it higher differentiability} of the solutions: in \cite{KMS2} and \cite{Coz15} the case $p=2$ is treated (also for a more general kernel $K(x,y)$, not necessarily given by $|x-y|^{-N-2\,s}$), while for a general $p\ge 2$ results in this sense have been obtained by the third author in \cite[Theorem 1.3]{Sc} and by the first and second author in \cite[Theorem 1.5]{Brolin}. See also the more recent contributions \cite{Auscher1,Auscher2}.
\par
However, in none of the above mentioned papers an explicit H\"older exponent has been obtained. This is typical when for instance using techniques in the vein of  De Giorgi and Moser: one can just prove the existence of {\it some} H\"older exponent. 
\begin{oss}
It is worth pointing out that in the linear case (i.e. when $p=2$), the regularity of solutions is well-understood. For instance, it follows from the integral representation in terms of Poisson kernels that solutions of $(-\Delta)^s u =0$ are $C^\infty$, see \cite[page 125, formula (1.6.19)]{La}. In addition, solutions of $(-\Delta)^s u =f$ for $f\in L^\infty$ are $C^{2\,s}$ whenever $2\,s\neq 1$ (see for example \cite[Theorem 6.4]{ALP}).
\end{oss}
\begin{oss}
We also mention that there are other ways of defining fractional versions of the $p-$Laplace operator. This has been done in \cite{BCF, BCF2, CJ}, \cite{SSS} and \cite{SV}. In \cite{BCF} and \cite{BCF2} an interesting connection to a nonlocal tug-of-war is found, while the operator introduced in \cite{CJ} is connected to L\'evy processes. We point out that the operators considered in these papers differ substantially from the one considered in the present paper.
\end{oss}

\subsection{Main results}\label{sec:main} Here we describe our main result. This is valid for \emph{local weak solutions} of \eqref{MainPDE}. However, a proper definition of weak solutions is quite technical and we therefore postpone this definition and other technicalities to Section \ref{sec:ex}. Below is our main theorem:
\begin{teo}\label{mainthm}
Let $\Omega\subset\mathbb{R}^N$ be a bounded and open set. Assume $2\le p<\infty$, $0<s<1$ and 
\[
\left\{\begin{array}{lr}
q>\dfrac{N}{s\,p},& \mbox{ if } s\,p\le N,\\
&\\
q\ge 1,& \mbox{ if }s\,p>N,
\end{array}
\right.
\]
We define the exponent
\begin{equation}
\label{theta}
\Theta(N,s,p,q):=\min\left\{\frac{1}{p-1}\left(s\,p-\frac{N}{q}\right),\,1\right\}.
\end{equation}
Let $u\in W^{s,p}_{\rm loc}(\Omega)\cap L^\infty_{\rm loc}(\Omega)\cap L^{p-1}_{s\,p}(\mathbb{R}^N)$ be a local weak solution of 
$$
(-\Delta_p)^s u = f, \qquad \mbox{ in }\Omega,
$$
where $f\in L_{\rm loc}^{q}(\Omega)$. Then $u\in C^\delta_{\rm loc}(\Omega)$ for every $0<\delta<\Theta(N,s,p,q)$. 
\par
More precisely, for every $0<\delta<\Theta(N,s,p,q)$ and every ball $B_{4R}(x_0)\Subset\Omega$, there exists a constant $C=C(N,s,p,\delta)>0$ such that 
\[
\begin{split}
[u]_{C^{\delta}(B_{R/8}(x_0))}\leq \frac{C}{R^{\delta}}&\,\Bigg [\|u\|_{L^\infty(B_{R}(x_0))}+\left(R^{s\,p}\,\int_{\mathbb{R}^N\setminus B_{R}(x_0)}\frac{|u(y)|^{p-1}}{|x_0-y|^{N+s\,p}}\,  dy\right)^\frac{1}{p-1}\\
&+\left(R^{s\,p-\frac{N}{q}}\|f\|_{L^{q}(B_{R}(x_0))}\right)^\frac{1}{p-1}\Bigg ].
\end{split}
\]
\end{teo}
We first observe that we ask the solution to be locally bounded. This may seem to be a restriction.
However, this kind of mild regularity for a weak solution comes for free under the standing assumptions on $f$, see Theorem~\ref{teo:loc_bound} and Remark \ref{oss:stupid} below. 
\par
Also, the reader may observe that in the theorem above, there is the assumption $u\in L^{p-1}_{s\,p}(\mathbb{R}^N)$. This guarantees that the term
\[
\int_{\mathbb{R}^N\setminus B_{R}(x_0)}\frac{|u(y)|^{p-1}}{|x_0-y|^{N+s\,p}}\,  dy,
\]
is finite (see Section \ref{sec:not} for the precise definition of this space). It is noteworthy to point out that $u\in L^{p-1}_{s\,p}(\mathbb{R}^N)$ whenever $u$ is globally bounded or grows slower that $|x|^{s\,p/(p-1)}$ at infinity.
\vskip.2cm
Before going further and giving the ideas and the details of the proof of Theorem~\ref{mainthm}, we give an example showing that for every finite $q$ satisfying the assumptions of Theorem~\ref{mainthm} and every $0<\varepsilon\ll 1$, there exists $u_\varepsilon\in W^{s,p}_{\rm loc}\cap L^\infty_{\rm loc}\cap L_{s\,p}^{p-1}$ such that 
\[
(-\Delta_p)^s u_\varepsilon\in L^q_{\rm loc},\qquad u_\varepsilon\in C^{\Theta+\varepsilon} \qquad \mbox{ but }\qquad u_\varepsilon\not\in \bigcup_{\beta>\Theta+\varepsilon} C^\beta.
\]
This shows that in general the exponent $\Theta$ found in Theorem~\ref{mainthm} can not be improved.
\begin{exa}[Sharpness of the H\"older exponent, $q<\infty$]
\label{exa:sharp}
For $N\ge 2$, we take 
\[
2<p\le N+1\qquad \mbox{ and }\qquad 0<s\le \frac{p-1}{p}.
\] 
Observe that with these choices, we automatically get $s\,p\le N$. We then fix an exponent $N/(s\,p)<q<\infty$ and observe that 
\[
\Theta(N,s,p,q)=\min\left\{\frac{1}{p-1}\,\left(s\,p-\frac{N}{q}\right),\,1\right\}=\frac{1}{p-1}\,\left(s\,p-\frac{N}{q}\right).
\] 
again thanks to our choices. Finally, we consider the function $u_\varepsilon(x)=|x|^{\Theta+\varepsilon}$, where 
\[
0<\varepsilon<\frac{1}{p-1}\,\frac{N}{q}.
\]
By construction, we have that 
\[
\Theta+\varepsilon>s \qquad \mbox{ and thus }\qquad u_\varepsilon\in C^{\Theta+\varepsilon}_{\rm loc}(\mathbb{R}^N)\subset W^{s,p}_{\rm loc}(\mathbb{R}^N),
\]
and
\[
(\Theta+\varepsilon)\,(p-1)<s\,p\qquad \mbox{ and thus }\qquad u_\varepsilon\in L^{p-1}_{s\,p}(\mathbb{R}^N).
\]
By using the homogeneity and radial symmetry of $u_\varepsilon$ and the properties of the operator $(-\Delta_p)^s$, it is not difficult to see that $u_\varepsilon\in W^{s,p}_{\rm loc}(\mathbb{R}^N)\cap L^\infty_{\rm loc}(\mathbb{R}^N)\cap L^{p-1}_{s\,p}(\mathbb{R}^N)$ is a local weak solution of 
\[
(-\Delta_p)^s u=f,\qquad \mbox{ with } f(x)=C_{s,p,\alpha}\,|x|^{(\Theta+\varepsilon-s)\,(p-1)-s}.
\] 
We observe that 
\[
f\in L^q_{\rm loc}(\mathbb{R}^N) \quad \Longleftrightarrow\quad \frac{N}{q}>s\,p-(\Theta+\varepsilon)\,(p-1)\quad \Longleftrightarrow \quad \Theta+\varepsilon>\frac{1}{p-1}\,\left(s\,p-\frac{N}{q}\right),
\]
and the latter is true, thanks to the definition of $\Theta$. We also observe that 
\[
u_\varepsilon\not\in C^\beta_{\rm loc}(\mathbb{R}^N),\qquad \mbox{ for any }\beta>\Theta+\varepsilon,
\]
since in this case
\[
[u_\varepsilon]_{C^{0,\beta}}(B_R)=\sup_{x,y\in B_R} \frac{|u_\varepsilon(x)-u_\varepsilon(y)|}{|x-y|^\beta}\ge \sup_{x\in B_R} \frac{|u_\varepsilon(x)-u_\varepsilon(0)|}{|x|^\beta}=\sup_{x\in B_R} |x|^{\Theta+\varepsilon-\beta}=+\infty.
\]
\end{exa}
\vskip.2cm
As for the case $q=\infty$, we observe that for $s= (p-1)/p$ the exponent $\Theta$ reduces to $1$.
We now give an example showing that in general this exponent cannot be reached, already in the linear case $p=2$.  
Namely, for $N \geq 2$
one can find $u$ such that
\[
(-\Delta)^{\frac{1}{2}} u \in L^\infty_{\rm loc}\qquad \not\Rightarrow\qquad  u\not\in C^{0,1}_{\rm loc}.
\]
This the very same reason\footnote{Essentially, both statements are a consequence of the fact that the Riesz transforms do not map $L^\infty$ into $L^\infty$.} for the failure of the Calder\'on-Zygmund estimates for the Laplacian in the borderline case of $L^\infty$, 
\[
\Delta u \in L^\infty_{\rm loc}\qquad \not \Rightarrow\qquad u\not\in C^{1,1}_{\rm loc}.
\]

\begin{exa}[Sharpness of the H\"older exponent, $s=1/2$, $p=2$ and $q=\infty$]\label{exa:sharpinfty}
For $s=1/2$, $p = 2$ and $q = \infty$, we have
\[
\Theta(N,s,p,q) =1,
\]
as already noticed.
We denote by $B \subset \mathbb{R}^N$ the unit ball and set for $N \geq 3$ the following {\it Riesz potential}
\[
 u(x) := \int_{B} |x-y|^{1-N}\, dy,\qquad \mbox{ for }x\in\mathbb{R}^N,
\]
i.e. the convolution between the singular kernel $I_1(x)=|x|^{1-N}$ and the characteristic function $1_B$. We first observe that by \cite[Theorem 1, Chapter V]{St} and using the fact that $1_B\in L^1(\mathbb{R}^N)\cap L^\infty(\mathbb{R}^N)$, we have
\[
u \in L^m(\mathbb{R}^N),\qquad\mbox{ for any } \frac{N}{N-1}<m\le \infty.
\]
In particular, $u\in L^2(\mathbb{R}^N)$ for $N\ge 3$.
% 
% Then by \cite[Lemma 2, Chapter V]{St} we have
% \begin{equation}
% \label{half}
% |\xi|\,\mathcal{F}[u](\xi)=C\,\mathcal{F}[1_B](\xi),\qquad \mbox{ for }\xi\in\mathbb{R}^N,
% \end{equation}
% where $\mathcal{F}$ is the Fourier transform and $C$ is a dimensional constant. This in particular implies that
% \[
% (1+|\xi|)\,\mathcal{F}[u]\in L^2(\mathbb{R}^N),
% \]
% which yields that $u\in W^{1/2,2}(\mathbb{R}^N)$, see \cite[Lemma 16.3]{Ta}. From the relation \eqref{half} we also obtain
% \[
% \mathcal{F}^{-1}\left(|\xi|\,\mathcal{F}[u]\right)=C\,1_B,
% \]
% which implies 
By \cite[Theorem 3.22]{Samko} $(-\Delta)^{1/2} u$ coincides (up to a multiplicative constant) with $1_B$.
However, in view of \cite[Lemma 2.15]{Oh}, we find  
\[
 \|\nabla u\|_{L^\infty(\Omega)} = \infty,
\]
for any bounded domain $\Omega$ compactly containing $B$. In particular $u$ is not a Lipschitz function.
\end{exa}

\subsection{A glimpse of the proof}
The starting point is to prove Theorem~\ref{mainthm} for the homogeneous equation
\begin{equation}
\label{MainPDE0}
(-\Delta_p)^s u=0,
\end{equation}
and then use a ``harmonic replacement'' argument to transfer the regularity to solutions of \eqref{MainPDE}, under the standing assumptions on the right-hand side $f$.
\par
The main idea to prove Theorem~\ref{mainthm} for $f\equiv 0$ is quite simple: we differentiate equation \eqref{MainPDE0} in a discrete sense and then test the differentiated equation against monotone power functions of fractional derivatives of the solution, i.e. quantities like
$$
\left|\frac{\delta_h u(x)}{|h|^\vartheta}\right|^{\beta-1}\,\frac{\delta_h u(x)}{|h|^\vartheta},\qquad \mbox{ where } \delta_h u(x):=u(x+h)-u(x).
$$
By choosing $\vartheta>0$ and $\beta\ge 1$ in a suitable way, this establishes a recursive gain in integrability (see Proposition \ref{prop:improve}) of the type
$$
\left\|\frac{\delta_h u(x)}{|h|^s}\right\|_{L^{q+1}(B_{1/2})}\lesssim \left\|\frac{\delta_h u(x)}{|h|^s}\right\|_{L^{q}(B_1)}.
$$
This can be iterated finitely many times in order to obtain
\[
\frac{\delta_h u(x)}{|h|^s}\in L^q_{\rm loc},\qquad \mbox{ for every } q<\infty, \mbox{ uniformly in } |h|\ll 1,
\]
and thus $u\in C_\text{loc}^\delta$ for any $0<\delta<s$. This part of the proof can be considered as a nonlocal counterpart of the method based on the {\it Moser's iteration} that can be used to prove Lipschitz regularity for $p-$harmonic functions (see for example \cite[Proposition 3.3]{DiB}): differentiate the equation in discrete sense (of order one); test with powers of first order differential quotients; use Sobolev inequality to get reverse H\"older inequalities and iterate infinitely many times to get
\[
\frac{\delta_h u(x)}{|h|}\in L^\infty_{\rm loc},\qquad \mbox{ uniformly in } |h|\ll 1,
\]
i.e. local Lipschitz regularity of the solutions 
\par
Once ``almost'' $s-$H\"older regularity is established, we are able to refine the estimates used to prove Proposition \ref{prop:improve}, and obtain a more powerful (yet more complicated) iterative self-improving scheme. This is Proposition \ref{prop:improve2}, giving an estimate of the type
\begin{equation}
\label{imperik}
\sup_{0<|h|< h_0}\left\|\frac{\delta^2_h u}{|h|^{\frac{1+s\,p+\vartheta\,\beta}{\beta-1+p}}}\right\|_{L^{\beta+p-1}(B_{1/2})}^{\beta+p-1}\lesssim \left\|\frac{\delta^2_h u }{|h|^\frac{1+\vartheta\, \beta}{\beta}}\right\|_{L^\beta(B_{1})}^\beta,\qquad \mbox{ whenever }\ \frac{1+\vartheta\, \beta}{\beta}<1.
\end{equation}
This part of the proof is quite peculiar of the nonlocal setting and it seems not to have a local counterpart. This new scheme allows for an iteration both on the differentiability $\vartheta$ and the integrability $\beta$ giving in the end 
\[
u\in C_\text{loc}^{\alpha},\qquad \mbox{ for every } \alpha <\min\left\{\frac{s\,p}{p-1},\,1\right\}.
\] 
We point out that the method can never give better than Lipschitz regularity. This is due to the simple fact that we can only test with first order differential quotients. This can also be seen in the requirement that the order of differentiability $(1+\vartheta\, \beta)/\beta$ on the right-hand side in \eqref{imperik} must be less than $1$.
\par
As announced at the beginning, once we proved this regularity for the homogeneous equation, we can transfer the regularity to the inhomogeneous equation by quite a standard perturbation argument. However, in order not to lose the H\"older exponent in this transfer, we need to employ a blow-up argument similar to that used by Caffarelli and Silvestre in \cite{CS}, see Proposition \ref{prop:caffsilv} below.
\par
Noteworthy is that the methods in this paper are quite similar to those in \cite{Brolin} by the first two authors, where we only test with 
\[
\frac{\delta_h u(x)}{|h|^\vartheta},
\]
and iterate to obtain {\it higher differentiability} results.
\begin{oss}[Dependence of the constant]
If one would carefully keep track of all the constants, one could perhaps iterate the previous scheme infinitely many times and arrive at exactly $C^\Theta$ regularity. Also, if great care was taken in estimating the constants, the results should be stable as $s\nearrow 1$. However, we believe that the proofs are complicated and long enough without explicitly estimating the $s-$dependence at each step.
\par
About the case $s \nearrow 1$, we only point out that the integrability hypothesis on $f$ of Theorem \ref{mainthm} becomes in the limit
\[
f\in L^q_{\rm loc}\qquad \mbox{ with }\qquad \left\{\begin{array}{lr}
q>\dfrac{N}{p},& \mbox{ if } p\le N,\\
&\\
q\ge 1,& \mbox{ if }p>N,
\end{array}
\right.
\]
which is exactly the sharp assumption on $f$ on the scale of Lebesgue spaces, in order to get H\"older regularity of solutions in the local case, i.e. local weak solutions of
\[
-\Delta_p u=f.
\]
We refer to \cite[Corollary 1, page 26]{KM} for this last result. This coincidence and a careful inspection of the results in \cite{KM}, seem to suggest that one could still prove Theorem \ref{mainthm} by slightly weakening the assumptions on $f$ and taking it to belong to some suitable Lorentz or Marcinkiewicz space. However, we prefer not to insist on this point.
\end{oss}
\subsection{Plan of the paper}
 In Section \ref{sec:prel}, we introduce the expedient definitions and notations needed in the paper. We also state and prove some preliminary results such as existence of solutions and functional space embeddings. In Section \ref{sec:basic}, we recall some known results and prove a first H\"older regularity result for the non-homogeneous equation.
 \par 
 The hard work is carried out in Section \ref{sec:almost}, where the almost $C^s$ regularity is proved for the homogeneous equation. It is based on Proposition \ref{prop:improve}, which expresses a gain of integrability of second order differential quotients. This self-improving estimate is then iterated to obtain Theorem~\ref{teo:1}.
 \par
 In Section \ref{sec:higher}, we use Theorem~\ref{teo:1} to obtain an enhanced version of Proposition \ref{prop:improve}, namely Proposition \ref{prop:improve2}, which expresses an interlinked gain between integrability and differentiability of second order differential quotients.
 \par
The iteration of this then results in Theorem~\ref{teo:1higher}, which is nothing but Theorem~\ref{mainthm} for the homogeneous equation. Finally, in Section \ref{sec:last}, we use perturbation arguments in order transfer the regularity to the inhomogeneous equation and prove the main theorem in full generality. The Appendix \ref{sec:app} contains the proof of several crucial pointwise inequalities.

\begin{ack}
We thank Shigehiro Sakata (University of Miyazaki) for having drawn our attention on \cite[Lemma 2.15]{Oh}. 
Part of this work has been carried out during a visit of L. B. \& E. L. to Freiburg in January 2017, a visit of L. B. to Stockholm in February 2017 and of E. L. to Ferrara in May 2017. Hosting institutions are kindly acknowledged. 
\par
E. L.  is supported by the Swedish Research Council, grant no. 2012-3124 and 2017-03736.
A. S. is supported by the German Research Foundation, grant no.~SCHI-1257-3-1, and Daimler-Benz Foundation, grant no. 32-11/16. A. S. was Heisenberg fellow.
\end{ack}

\section{Preliminaries}\label{sec:prel}

\subsection{Notation}\label{sec:not}
Let $1<p<\infty$. We define the monotone function $J_p:\mathbb{R}\to\mathbb{R}$ by
\[
J_p(t)=|t|^{p-2}\, t.
\]
We denote by $\omega_N$ the measure of the $N-$dimensional open ball of radius $1$.
\par
For a measurable function $\psi:\mathbb{R}^N\to\mathbb{R}$ and a vector $h\in\mathbb{R}^N$, we define
\[
\psi_h(x)=\psi(x+h),\qquad \delta_h \psi(x)=\psi_h(x)-\psi(x),\qquad \delta^2_h \psi(x)=\delta_h(\delta_h \psi(x))=\psi_{2\,h}(x)+\psi(x)-2\,\psi_h(x).
\]
We recall that we have the discrete Leibniz rule
\begin{equation}
\label{leibniz}
\delta_h(\varphi\,\psi)=\psi_h\,\delta_h \varphi+\varphi\,\delta_h \psi.
\end{equation}
Let $1\le q<\infty$ and let $\psi\in L^q(\mathbb{R}^N)$, for $0<\beta\le 1$ we set
\[
[\psi]_{\mathcal{N}^{\beta,q}_\infty(\mathbb{R}^N)}:=\sup_{|h|>0} \left\|\frac{\delta_h \psi}{|h|^{\beta}}\right\|_{L^q(\mathbb{R}^N)},
\]
and for $0<\beta<2$
\[
[\psi]_{\mathcal{B}^{\beta,q}_\infty(\mathbb{R}^N)}:=\sup_{|h|>0} \left\|\frac{\delta_h^2 \psi}{|h|^{\beta}}\right\|_{L^q(\mathbb{R}^N)}.
\]
We then introduce the two Besov-type spaces
\[
\mathcal{N}^{\beta,q}_\infty(\mathbb{R}^N)=\left\{\psi\in L^q(\mathbb{R}^N)\, :\, [\psi]_{\mathcal{N}^{\beta,q}_\infty(\mathbb{R}^N)}<+\infty\right\},\qquad 0<\beta\le 1,
\]
and
\[
\mathcal{B}^{\beta,q}_\infty(\mathbb{R}^N)=\left\{\psi\in L^q(\mathbb{R}^N)\, :\, [\psi]_{\mathcal{B}^{\beta,q}_\infty(\mathbb{R}^N)}<+\infty\right\},\qquad 0<\beta<2.
\]
We also need the {\it Sobolev-Slobodecki\u{\i} space}
\[
W^{\beta,q}(\mathbb{R}^N)=\left\{\psi\in L^q(\mathbb{R}^N)\, :\, [\psi]_{W^{\beta,q}(\mathbb{R}^N)}<+\infty\right\},\qquad 0<\beta<1,
\]
where the seminorm $[\,\cdot\,]_{W^{\beta,q}(\mathbb{R}^N)}$ is defined by
\[
[\psi]_{W^{\beta,q}(\mathbb{R}^N)}=\left(\iint_{\mathbb{R}^N\times \mathbb{R}^N} \frac{|\psi(x)-\psi(y)|^q}{|x-y|^{N+\beta\,q}}\,dx\,dy\right)^\frac{1}{q}.
\]
We endow these spaces with the norms
\[
\|\psi\|_{\mathcal{N}^{\beta,q}_\infty(\mathbb{R}^N)}=\|\psi\|_{L^q(\mathbb{R}^N)}+[\psi]_{\mathcal{N}^{\beta,q}_\infty(\mathbb{R}^N)},
\]
\[
\|\psi\|_{\mathcal{B}^{\beta,q}_\infty(\mathbb{R}^N)}=\|\psi\|_{L^q(\mathbb{R}^N)}+[\psi]_{\mathcal{B}^{\beta,q}_\infty(\mathbb{R}^N)},
\]
and
\[
\|\psi\|_{W^{\beta,q}(\mathbb{R}^N)}=\|\psi\|_{L^q(\mathbb{R}^N)}+[\psi]_{W^{\beta,q}(\mathbb{R}^N)}.
\]
A few times we will also work with the space $W^{\beta,q}(\Omega)$ for a subset $\Omega\subset \mathbb{R}^N$,
\[
W^{\beta,q}(\Omega)=\left\{\psi\in L^q(\Omega)\, :\, [\psi]_{W^{\beta,q}(\Omega)}<+\infty\right\},\qquad 0<\beta<1,
\]
where we define
\[
 [\psi]_{W^{\beta,q}(\Omega)}=\left(\iint_{\Omega\times \Omega} \frac{|\psi(x)-\psi(y)|^q}{|x-y|^{N+\beta\,q}}\,dx\,dy\right)^\frac{1}{q}.
\]
We will also denote the average of a function $\psi$ over the ball $B_r(x_0)$ by
\[
\overline{\psi}_{x_0,r}=\fint_{B_r(x_0)} \psi\,dx=\frac{1}{|B_{r}(x_0)|}\,\int_{B_r(x_0)} \psi\,dx.
\]
\subsection{Tail spaces} We introduce the {\it tail space}
\[
L^{q}_{\alpha}(\mathbb{R}^N)=\left\{u\in L^{q}_{\rm loc}(\mathbb{R}^N)\, :\, \int_{\mathbb{R}^N} \frac{|u|^q}{1+|x|^{N+\alpha}}\,dx<+\infty\right\},\qquad q>0 \mbox{ and } \alpha>0,
\]
and define for every $x_0\in\mathbb{R}^N$, $R>0$ and $u\in L^q_{\alpha}(\mathbb{R}^N)$
\[
\mathrm{Tail}_{q,\alpha}(u;x_0,R)=\left[R^{\alpha}\,\int_{\mathbb{R}^N\setminus B_R(x_0)} \frac{|u|^q}{|x-x_0|^{N+\alpha}}\,dx\right]^\frac{1}{q}.
\]
It is not difficult to see that the quantity above is always finite, for a function $u\in L^q_{\alpha}(\mathbb{R}^N)$.
\begin{lm}
Let $\alpha>0$ and $0<q<m<\infty$. Then
\[
L^{m}_{\alpha}(\mathbb{R}^N)\subset L^{q}_{\alpha}(\mathbb{R}^N).
\]
\end{lm}
\begin{proof}
This is an easy consequence of Jensen's inequality and the fact that the measure
\[
(1+|x|^{N+\alpha})^{-1}\,dx,
\]
is finite on $\mathbb{R}^N$. We leave the details to the reader.
\end{proof}
The following easy technical results contain computations that will be used many times. We state them as separate results for ease of readability.
\begin{lm}
\label{lm:allarga}
Let $\alpha>0$ and $0< q<\infty$. For every $0<r<R$ and $x_0\in\mathbb{R}^N$ we have
\[
R^\alpha\,\sup_{x\in B_r(x_0)}\int_{\mathbb{R}^N\setminus B_R(x_0)} \frac{|u(y)|^q}{|x-y|^{N+\alpha}}\,dy\le \left(\frac{R}{R-r}\right)^{N+\alpha}\,\mathrm{Tail}_{q,\alpha}(u;x_0,R)^q.
\]
\end{lm}
\begin{proof}
It is sufficient to observe that for $x\in B_r(x_0)$ and $y\in \mathbb{R}^N\setminus B_R(x_0)$, we have
\[
|x-y|\ge |y-x_0|-|x-x_0|\ge |y-x_0|-r\ge |y-x_0|-\frac{r}{R}\,|y-x_0|=\frac{R-r}{R}\,|y-x_0|.
\]
This gives the desired conclusion.
\end{proof}
\begin{lm}
\label{lm:tails}
Let $\alpha>0$ and $0< q<\infty$. Suppose that $B_r(x_0)\subset B_R(x_1)$. Then for every $u\in L^q_\alpha(\mathbb{R}^N)$ we have
\[
\mathrm{Tail}_{q,\alpha}(u;x_0,r)^q\le \left(\frac{r}{R}\right)^\alpha\,\left(\frac{R}{R-|x_1-x_0|}\right)^{N+\alpha}\,\mathrm{Tail}_{q,\alpha}(u;x_1,R)^q+r^{-N}\,\|u\|^q_{L^q(B_R(x_1))}.
\]
If in addition $u\in L^m_{\rm loc}(\mathbb{R}^N)$ for some $q<m\le \infty$, then
\begin{equation}
\label{paura}
\begin{split}
\mathrm{Tail}_{q,\alpha}(u;x_0,r)^q&\le \left(\frac{r}{R}\right)^\alpha\,\left(\frac{R}{R-|x_1-x_0|}\right)^{N+\alpha}\,\mathrm{Tail}_{q,\alpha}(u;x_1,R)^q\\
&+\left(N\,\omega_N\,\frac{m-q}{\alpha\,m+N\,q}\right)^\frac{m-q}{m}\,r^{-\frac{q\,N}{m}}\,\|u\|^q_{L^m(B_R(x_1))}.
\end{split}
\end{equation}
\end{lm}
\begin{proof}
We first observe that
\[
|x-x_0|\ge |x-x_1|-|x_1-x_0|\ge \frac{R-|x_1-x_0|}{R}\,|x-x_1|,\qquad \mbox{ for every }x\not \in B_R(x_1),
\]
and
\[
|x-x_0|\ge r,\qquad \mbox{ for every }x\not\in B_r(x_0).
\]
Then we decompose the tail as follows
\[
\begin{split}
\mathrm{Tail}_{q,\alpha}(u;x_0,r)^{q}&=r^{\alpha}\,\int_{\mathbb{R}^N\setminus B_R(x_1)} \frac{|u|^q}{|x-x_0|^{N+\alpha}}\,dx+r^{\alpha}\,\int_{B_R(x_1)\setminus B_r(x_0)} \frac{|u|^q}{|x-x_0|^{N+\alpha}}\,dx\\
&\le \left(\frac{r}{R}\right)^\alpha\,\left(\frac{R}{R-|x_1-x_0|}\right)^{N+\alpha}\,\mathrm{Tail}_{q,\alpha}(u;x_1,R)^q+r^{-N}\,\int_{B_R(x_1)} |u|^q\,dx.
\end{split}
\]
This gives the first estimate.
\par
In order to get \eqref{paura}, we proceed in a slightly different way to estimate the second integral: by using H\"older inequality
\[
\begin{split}
\int_{B_R(x_1)\setminus B_r(x_0)} \frac{|u|^q}{|x-x_0|^{N+\alpha}}\,dx&\le \|u\|^{q}_{L^m(B_R(x_1))}\,\left(\int_{\mathbb{R}^N\setminus B_r(x_0)} \frac{1}{|x-x_0|^{(N+\alpha)\,\frac{m}{m-q}}}\,dx\right)^\frac{m-q}{m}\\
&= \left(N\,\omega_N\,\frac{m-q}{\alpha\,m+N\,q}\right)^\frac{m-q}{m}\,r^{-\frac{\alpha\,m+N\,q}{m}}\,\|u\|^{q}_{L^m(B_R(x_1))}.
\end{split}
\]
Thus we obtain \eqref{paura} as well.
\end{proof}
\subsection{Embedding inequalities}

The following result can be found for example in \cite[Lemma 2.3]{BS}.
\begin{lm}
\label{lm:fromBtoN}
Let $0<\beta<1$ and $1\le q<\infty$. Then we have the continuous embedding 
\[
\mathcal{B}^{\beta,q}_\infty(\mathbb{R}^N)\hookrightarrow \mathcal{N}^{\beta,q}_\infty(\mathbb{R}^N).
\] 
More precisely, for every $\psi\in \mathcal{B}^{\beta,q}_\infty(\mathbb{R}^N)$ we have
\[
[\psi]_{\mathcal{N}^{\beta,q}_\infty(\mathbb{R}^N)}\le \frac{C}{1-\beta}\,[\psi]_{\mathcal{B}^{\beta,q}_\infty(\mathbb{R}^N)},
\]
for some constant $C=C(N,q)>0$. 
\end{lm}
The next one is contained in\footnote{This embedding should be seen as a Sobolev-type embedding. For the embedding properties of Triebel-Lizorkin and Besov spaces the interested reader is referred to, e.g., \cite[Section 2.2]{RS}} \cite[Proposition 2.4]{BS}.
\begin{prop}
\label{prop:lostinpassing}
Let $1\le q<\infty$ and $0<\alpha<\beta< 1$. Then we have the continuous embedding 
\[
\mathcal{N}^{\beta,q}_\infty(\mathbb{R}^N)\hookrightarrow W^{\alpha,q}(\mathbb{R}^N).
\]
More precisely, for every $\psi\in \mathcal{N}^{\beta,q}_\infty(\mathbb{R}^N)$ we have
\[
[\psi]^q_{W^{\alpha,q}(\mathbb{R}^N)}\le C\,\frac{\beta}{(\beta-\alpha)\,\alpha}\, \left([\psi]_{\mathcal{N}^{\beta,q}_\infty(\mathbb{R}^N)}^q\right)^\frac{\alpha}{\beta}\,\left(\|\psi\|^q_{L^q(\mathbb{R}^N)}\right)^\frac{\beta-\alpha}{\beta},
\]
for some constant $C=C(N,q)>0$.
\end{prop}
The following result is a sort of localized version of the previous estimate.
\begin{lm}
\label{lm:erik2}
Suppose $1\le q<\infty$ and $0<\alpha<\beta<1$. Let $u\in L^q_{\rm loc}(\mathbb{R}^N)$ be such that for some $h_0>0$ and some ball $B_r\subset\mathbb{R}^N$ with $r>h_0$ we have
\[
\sup_{0<|h|<h_0} \left\|\frac{\delta^2_h u}{|h|^\beta}\right\|_{L^q(B_r)}<+\infty.
\]
Then for every $\varrho>0$ such that $\varrho+h_0\le r$ we have
\[
[u]^q_{W^{\alpha,q}(B_\varrho)}\le C_1\,\left(\sup_{0<|h|< h_0} \left\|\frac{\delta^2_h u}{|h|^\beta}\right\|_{L^q(B_{r})}^q+\|u\|^q_{L^q(B_r)}\right),
\]
and
\[
\sup_{0<|h|<h_0} \left\|\frac{\delta_h u}{|h|^\beta}\right\|^q_{L^q(B_\varrho)}\le C_2\,\left(\sup_{0<|h|< h_0} \left\|\frac{\delta^2_h u}{|h|^\beta}\right\|_{L^q(B_{r})}^q+\|u\|^q_{L^q(B_r)}\right)
\]
where $B_\varrho$ is concentric with $B_r$. Here $C_1=C_1(N,q,\beta,\alpha,h_0)>0$ and $C_2=C_2(N,q,\beta,h_0)$ are constants that blow up as $\beta\nearrow 1$, $\alpha\nearrow\beta$ or $h_0\searrow 0$.
\end{lm}
\begin{proof}
We take $\chi\in C_0^\infty(B_{\varrho+h_0/3})$  such that 
\[
0\leq \chi\leq 1, \qquad \chi =1 \text{ in $B_{\varrho}$},\qquad |\nabla \chi|\leq \frac{C}{h_0}\quad \mbox{ and }\quad |D^2 \chi|\leq \frac{C}{h_0^2}.
\]
Note that we have
$$
|\delta_h\chi|\leq \frac{C}{h_0}\,|h|\quad \mbox{ and }\quad|\delta^2_h\chi|\leq \frac{C}{h_0^2}\,|h|^2.
$$
We start by observing that  
\begin{equation}
\label{clearly}
[u]_{W^{\alpha,q}(B_\varrho)}^q\leq [u\,\chi]_{W^{\alpha,q}(\mathbb{R}^N)}^q.
\end{equation}
By using \cite[Proposition 2.7]{Brolin}, we obtain
\[
[u\,\chi]_{W^{\alpha,q}(\mathbb{R}^N)}^q\leq C\,\left(\frac{h_0^{(\beta-\alpha)\,q}}{\beta-\alpha}\,\sup_{0<|h|< h_0}\left\|\frac{\delta_h (u\,\chi)}{|h|^\beta}\right\|_{L^q(\mathbb{R}^N)}^q+\frac{h_0^{-\alpha\,q}}{\alpha}\,\|u\,\chi\|_{L^q(\mathbb{R}^N)}\right)\\
\]
where $C=C(N,q)>0$. We now use \cite[Lemma 2.3]{Brolin} to find
\begin{equation}
\label{asprilla}
[u\,\chi]_{W^{\alpha,q}(\mathbb{R}^N)}^q\leq C\,\left(\sup_{0<|h|<h_0}\left\|\frac{\delta^2_{h} (u\,\chi)}{|h|^{\beta}}\right\|_{L^q(\mathbb{R}^N)}^q+\|u\,\chi\|_{L^q(\mathbb{R}^N)}^q\right),
\end{equation}
where\footnote{The constant blows-up as $\beta\nearrow 1$, $\alpha\nearrow\beta$ and $h_0\searrow 0$.} $C=C(h_0,N,q,\beta,\alpha)$.
We now observe that
$$
\delta^2_h (u\,\chi)=\chi_{2h}\,\delta^2_h u+2\,\delta_h u\, \delta_h \chi_h+u\,\delta^2_h\chi,
$$
therefore for $0<|h|<h_0$
\begin{equation}
\label{stupid}
\begin{split}
\left\|\frac{\delta^2_h (u\,\chi)}{|h|^\beta}\right\|_{L^q(\mathbb{R}^N)}^q&\leq C\,\left(\left\|\frac{\chi_{2h}\,\delta^2_h u}{|h|^\beta}\right\|_{L^q(\mathbb{R}^N)}^q+\left\|\frac{\delta_h u\,\delta_h\chi}{|h|^\beta}\right\|_{L^q(\mathbb{R}^N)}^q+\left\|\frac{u\,\delta^2_h\chi}{|h|^\beta}\right\|_{L^q(\mathbb{R}^N)}^q\right) \\
&\leq C\,\left(\left\|\frac{\delta^2_h u}{|h|^\beta}\right\|_{L^q(B_r)}^q+h_0^{(1-\beta)\,q}\,\|\delta_h u\|_{L^q(B_{\varrho+2\,h_0/3})}^q+h_0^{(2-\beta)\,q}\,\|u\|_{L^q(B_r)}^q\right)\\
&\leq C\left(\left\|\frac{\delta^2_h u}{|h|^\beta}\right\|_{L^q(B_r)}^q+\|u\|_{L^q(B_r)}^q\right)\\
\end{split}
\end{equation}
where we used the estimates on $\nabla \chi$ and $D^2 \chi$ and the triangle inequality to estimate the $L^q$ norm of $\delta_h u$. The constant $C$ depends on $\alpha,\beta,q, N$ and $h_0$ as before. By combining \eqref{clearly}, \eqref{asprilla} and \eqref{stupid}, we obtain
$$
[u]_{W^{\alpha,q}(B_\varrho)}^q\leq C\,\left(\sup_{0<|h|<h_0}\left\|\frac{\delta^2_h u}{|h|^\beta}\right\|_{L^q(B_r)}^q+\|u\|_{L^q(B_r)}^q\right),
$$
thus concluding the proof of the first inequality.
\vskip.2cm\noindent
In order to prove the second inequality, we still use the cut-off function $\chi$ and the Leibniz rule \eqref{leibniz} for $\delta_h(u\,\chi)$. We may write for $0<|h|<h_0$
\[
\begin{split}
\left\|\frac{\delta_h u}{|h|^\beta}\right\|_{L^q(B_\varrho)}^q&\leq \left\|\frac{\delta_h u\,\chi}{|h|^\beta}\right\|_{L^q(\mathbb{R}^N)}^q\\
&\leq C\,\left(\left\|\frac{\delta_h (u\, \chi) }{|h|^\beta}\right\|_{L^q(\mathbb{R}^N)}^q+\left\|\frac{u_h\,\delta_h\chi }{|h|^{\beta}}\right\|_{L^q(\mathbb{R}^N)}^{q}\right)\\
&\leq C\,\left(\left\|\frac{\delta^2_h (u\, \chi) }{|h|^\beta}\right\|_{L^q(\mathbb{R}^N)}^q+\|u_h\|_{L^q(B_{R+2\,h_0/3})}^q+\|u\|_{L^q(B_{R+h_0})}^q \right)\\ 
&\leq C\,\left(\left\|\frac{\delta^2_h (u\,\chi) }{|h|^\beta}\right\|_{L^q(\mathbb{R}^N)}^q+\|u\|_{L^q(B_{R+h_0})}^q\right),
\end{split}
\]
where we have used the bound on $\nabla \chi$ and $D^2 \chi$ and \cite[Lemma 2.3]{Brolin}. Finally we apply \eqref{stupid} to treat the term containing $\delta^2_h (u\,\chi)$.
\end{proof}
In what follows, we set
\[
p^*_s=\left\{\begin{array}{rcll}
\dfrac{N\,p}{N-s\,p},& \mbox{ if } s\,p<N,\\
&\\
+\infty,& \mbox{ if } s\,p>N,
\end{array}
\right.
\qquad\mbox{ and }\qquad (p^*_s)'=\left\{\begin{array}{rcll}
\dfrac{N\,p}{N\,p-N+s\,p},& \mbox{ if } s\,p<N,\\
&\\
1,& \mbox{ if } s\,p>N.
\end{array}
\right..
\]
\begin{prop}[Poincar\'e \& Sobolev]
\label{prop:sobpoin1}
Suppose $1<p<\infty$ and $0<s<1$. Let $\Omega\subset\mathbb{R}^N$ be an open and bounded set. For every $u\in W^{s,p}(\mathbb{R}^N)$ such that $u=0$ almost everywhere in $\mathbb{R}^N\setminus\Omega$, we have
\[
\|u\|^p_{L^{p^*_s}(\Omega)}\le C_1\,[u]^p_{W^{s,p}(\mathbb{R}^N)}, \qquad \mbox{ if }  s\,p<N,
\]
\[
\|u\|^p_{L^\infty(\Omega)}\le C_2\,|\Omega|^{\frac{s\,p}{N}-1}\,[u]^p_{W^{s,p}(\mathbb{R}^N)}, \qquad \mbox{ if }  s\,p>N,
\]
\[
\|u\|^p_{L^{q}(\Omega)}\le C_3\,|\Omega|^{\frac{p}{q}+\frac{s\,p}{N}-1}\,[u]^p_{W^{s,p}(\mathbb{R}^N)}, \qquad \mbox{ for every } 1\le q<\infty, \mbox{ if }  s\,p=N
\]
for constants $C_1=C_1(N,p,s)>0$, $C_2=C_2(N,p,s)>$ and $C_3=C_3(N,p,s,q)>0$. In particular, we also have
\begin{equation}
\label{poin}
\|u\|^p_{L^p(\Omega)}\le C\,|\Omega|^\frac{s\,p}{N}\,[u]^p_{W^{s,p}(\mathbb{R}^N)},
\end{equation}
for some $C=C(N,p,s)>0$.
\end{prop}

\begin{comment}
\begin{prop}[Poincar\'e \& Sobolev]
\label{prop:sobpoin}
Suppose $1<p<\infty$ and $0<s<1$. Let $\Omega\Subset \Omega'\subset\mathbb{R}^N$ be two open and bounded sets. For every $u\in W^{s,p}(\Omega')$ such that $u=0$ almost everywhere in $\Omega'\setminus\Omega$, we have
\begin{equation}
\label{poin}
\|u\|^p_{L^p(\Omega)}\le \frac{\mathrm{diam}(\Omega')^{N+s\,p}}{|\Omega'\setminus\Omega|}\,[u]^p_{W^{s,p}(\Omega')},
\end{equation}
and
\[
\|u\|^p_{L^{p^*_s}(\Omega)}\le C_1\,[u]^p_{W^{s,p}(\Omega')}, \qquad \mbox{ if }  s\,p\not=N,
\]
\[
\|u\|^p_{L^{q}(\Omega)}\le C_2\,[u]^p_{W^{s,p}(\Omega')}, \qquad \mbox{ for every } 1\le q<\infty, \mbox{ if }  s\,p=N
\]
for constants $C_1=C_1(N,p,s,\Omega,\Omega')>0$ and $C_2=C_2(N,p,s,q,\Omega,\Omega')>0$, which blow-up as $|\Omega'\setminus \Omega|$ goes to $0$.
\end{prop}
\begin{proof}
The first inequality \eqref{poin} can be proved for example by mimicking the proof of \cite[Lemma 2.4]{bralinpar}, we leave the details to the reader.
\par 
As for the second one, we focus for simplicity on the case $s\,p\not =N$: we take an open set with Lipschitz boundary $E$ such that $\Omega\Subset E\Subset \Omega'$. Then by Sobolev embedding (see \cite[Theorem 7.57]{Ad}) we know that
\[
\|u\|^p_{L^{p^*_s}(\Omega)}=\|u\|^p_{L^{p^*_s}(E)}\le C\,\left([u]^p_{W^{s,p}(E)}+\|u\|^p_{L^p(E)}\right)\le C\,\left([u]^p_{W^{s,p}(\Omega')}+\|u\|^p_{L^p(\Omega)}\right).
\]
We can now use \eqref{poin} in order to control the last $L^p$ norm.
\end{proof}
\end{comment}

In what follows, we will use the Campanato space $\mathcal{L}^{q,\lambda}$, we refer to \cite[Chapter 2]{Gi} for the relevant definition.
\begin{teo}[Morrey-type embedding]
\label{thm:Nholder}
Let $1\le q<\infty$ and $0<\beta<1$ be such that $\beta\,q>N$. If $\psi\in \mathcal{N}^{\beta,q}_\infty(\mathbb{R}^N)$, then $\psi \in C^{0,\alpha}_{\rm loc}(\mathbb{R}^N)$, for every $0<\alpha<\beta-N/q$. More precisely, we have
\[
\sup_{x\not=y} \frac{|\psi(x)-\psi(y)|}{|x-y|^\alpha}\le C\,\left([\psi]_{\mathcal{N}^{\beta,q}_\infty(\mathbb{R}^N)}\right)^\frac{\alpha\,q+N}{\beta\,q}\,\left(\|\psi\|_{L^q(\mathbb{R}^N)}\right)^\frac{(\beta-\alpha)q-N}{\beta\,q},
\]
with $C=C(N,q,\alpha,\beta)>0$ which blows-up as $\alpha\nearrow \beta-N/q$.
\end{teo}
\begin{proof}
Let us fix a ball $B$ and for every $x_0\in B$ we consider the ball $B_{\delta}(x_0)$, with $\delta>0$. We are going to show that $\psi$ belongs to the Campanato space $\mathcal{L}^{q,\gamma\,q}(B)$, for every $N/q<\gamma<\beta$.
\par
By a slight abuse of notation we denote by
\[
\overline\psi_{x_0,\delta}=\fint_{B\cap B_\delta(x_0)} \psi\,dx.
\]
Then for $N/q<\gamma<\beta$ we have
\[
\begin{split}
\int_{B\cap B_\delta(x_0)} |\psi-\overline\psi_{x_0,\delta}|^q\,dx&=\int_{B\cap B_\delta(x_0)}\left|\fint_{B\cap B_\delta(x_0)} \big[\psi(x)-\psi(y)\big]\,dy\right|^q\,dx\\
&\le \frac{1}{|B\cap B_\delta(x_0)|}\, \iint_{(B\cap B_\delta(x_0))\times (B\cap B_\delta(x_0))}|\psi(x)-\psi(y)|^q\,dx\,dy\\
&\le C\,\frac{\delta^{N+\gamma\,q}}{|B\cap B_\delta(x_0)|}\,\iint_{(B\cap B_\delta(x_0))\times (B\cap B_\delta(x_0))}\frac{|\psi(x)-\psi(y)|^q}{|x-y|^{N+\gamma\,q}}\,dx\,dy\\
&\le C\,\frac{\delta^{N+\gamma\,q}}{|B\cap B_\delta(x_0)|}\,\iint_{\mathbb{R}^N\times\mathbb{R}^N}\frac{|\psi(x)-\psi(y)|^q}{|x-y|^{N+\gamma\,q}}\,dx\,dy.
\end{split}
\]
We now recall that from Proposition \ref{prop:lostinpassing}
\[
[\psi]^q_{W^{\gamma,q}(\mathbb{R}^N)}\le C\,\frac{\beta}{(\beta-\gamma)\,\gamma}\, \left([\psi]_{\mathcal{N}^{\beta,q}_\infty(\mathbb{R}^N)}^q\right)^\frac{\gamma}{\beta}\,\left(\|\psi\|^q_{L^q(\mathbb{R}^N)}\right)^\frac{\beta-\gamma}{\beta}.
\]
Finally, we observe that
\[
\frac{\omega_N\,\delta^N}{2}\le |B\cap B_\delta(x_0)|,\qquad \mbox{ for every }x_0\in B.
\]
Thus we obtain
\[
\sup_{x_0\in B,\, \delta>0} \delta^{-\gamma\,q}\,\int_{B\cap B_\delta(x_0)} |\psi-\overline\psi_{x_0,\delta}|^q\,dx\le C\, \frac{\beta}{(\beta-\gamma)\,\gamma}\, \left([\psi]_{\mathcal{N}^{\beta,q}_\infty(\mathbb{R}^N)}^q\right)^\frac{\gamma}{\beta}\,\left(\|\psi\|^q_{L^q(\mathbb{R}^N)}\right)^\frac{\beta-\gamma}{\beta}.
\]
This implies that $\psi\in \mathcal{L}^{q,\gamma\,q}(B)$, for every ball $B\subset\mathbb{R}^N$ and every $N/q<\gamma<\beta$. By using that $\mathcal{L}^{q,\gamma\,q}(B)$ is isomorphic to $C^{0,\alpha}(\overline B)$ (see \cite[Theorem 2.9]{Gi}) with $\alpha=\gamma-N/q$ and
\[
[\psi]_{C^{0,\alpha}(B)}\le C\, \left(\sup_{x_0\in B,\, \delta>0} \delta^{-\gamma\,q}\,\int_{B\cap B_\delta(x_0)} |\psi-\overline\psi_{x_0,\delta}|^q\,dx\right)^\frac{1}{q},
\]
we get the desired conclusion on the ball $B$. The constant $C>0$ above does not depend on the size of the ball $B$. Thus by the arbitrariness of $B$, we get the conclusion.
\end{proof}

\subsection{Existence}\label{sec:ex}
We start with the following
\begin{defi}[Local weak solution]
\label{defi:localweak}
Suppose $1<p<\infty$ and $0<s<1$. Let $\Omega\subset\mathbb{R}^N$ be an open and bounded set and $f\in L^q_{\rm loc}(\Omega)$, with
\[
q\ge (p^*_s)'\quad \mbox{ if } s\,p\not =N\qquad \mbox{ or }\qquad q>1\quad \mbox{ if }s\,p=N.
\]
We say that $u\in W^{s,p}_{\rm loc}(\Omega)\cap L^{p-1}_{sp}(\mathbb{R}^N)$ is a {\it local weak solution} of $(-\Delta_p)^s u=f$ in $\Omega$ if  
\begin{equation}
\label{equation}
\iint_{\mathbb{R}^N\times\mathbb{R}^N} \frac{J_p(u(x)-u(y))\,\Big(\varphi(x)-\varphi(y)\Big)}{|x-y|^{N+s\,p}}\,dx\,dy=\int_\Omega f\,\varphi\,dx,
\end{equation}
for every $\varphi\in W^{s,p}(\Omega)$ compactly supported in $\Omega$.
\end{defi}
We now want to detail the notion of weak solutions to a Dirichlet boundary value problem for $(-\Delta_p)^s$. With this aim, we introduce the following space: given $\Omega\Subset\Omega'\subset\mathbb{R}^N$ open and bounded sets and $g\in L^{p-1}_{s\,p}(\mathbb{R}^N)$, we define
\[
X^{s,p}_g(\Omega,\Omega'):=\{v\in W^{s,p}(\Omega')\cap L^{p-1}_{s\,p}(\mathbb{R}^N)\, :\, v=g\ \mbox{ a.\,e. in }\mathbb{R}^N\setminus \Omega \}.
\]
\begin{defi}[Dirichlet problem]
Suppose $1<p<\infty$ and $0<s<1$. Let $\Omega\Subset \Omega'\subset\mathbb{R}^N$ be two open and bounded sets, $f\in L^q(\Omega)$, with
\[
q\ge (p^*_s)'\quad \mbox{ if } s\,p\not =N\qquad \mbox{ or }\qquad q>1\quad \mbox{ if }s\,p=N,
\]
and $g\in L^{p-1}_{sp}(\mathbb{R}^N)$. We say that $u\in X^{s,p}_g(\Omega,\Omega')$ is a {\it weak solution} of the boundary value problem
\begin{equation}
\label{BVP}
\left\{\begin{array}{rcll}
(-\Delta_p)^s\,u&=&f,&\mbox{ in }\Omega,\\
u&=&g,& \mbox{ in }\mathbb{R}^N\setminus \Omega,
\end{array}
\right.
\end{equation}
if \eqref{equation} holds
for every $\varphi\in X^{s,p}_0(\Omega,\Omega')$.
\end{defi}
About the space $X^{s,p}_0(\Omega,\Omega')$, the following simple technical result will be useful.
\begin{lm}
\label{lm:cazzata}
Let $1<p<\infty$ and $0<s<1$. Let $\Omega\Subset \Omega'\subset\mathbb{R}^N$ be two open and bounded sets. Then for every $u\in X^{s,p}_0(\Omega,\Omega')$ we have
\[
[u]^p_{W^{s,p}(\mathbb{R}^N)}\le [u]^p_{W^{s,p}(\Omega')}+\frac{C}{s}\,\mathrm{dist}(\Omega,\mathbb{R}^N\setminus \Omega')^{-s\,p}\,\|u\|_{L^p(\Omega)}^p,
\]
for some $C=C(N,p)>0$. In particular, we have the continuous embedding $X^{s,p}_0(\Omega,\Omega')\hookrightarrow W^{s,p}(\mathbb{R}^N)$.
\end{lm}
\begin{proof}
We take $u\in X^{s,p}_0(\Omega,\Omega')$ and then write
\[
\begin{split}
[u]^p_{W^{s,p}(\mathbb{R}^N)}&=\iint_{\mathbb{R}^N\times\mathbb{R}^N} \frac{|u(x)-u(y)|^p}{|x-y|^{N+s\,p}}\,dx\,dy\\
&=\iint_{\Omega'\times\Omega'} \frac{|u(x)-u(y)|^p}{|x-y|^{N+s\,p}}\,dx\,dy+2\, \iint_{\Omega\times(\mathbb{R}^N\setminus\Omega')}\frac{|u(x)|^p}{|x-y|^{N+s\,p}}\,dx\,dy,
\end{split}
\]
where we used that $u\equiv 0$ outside $\Omega$. It is only left to observe that
\[
\begin{split}
\iint_{\Omega\times(\mathbb{R}^N\setminus\Omega')}\frac{|u(x)|^p}{|x-y|^{N+s\,p}}\,dx\,dy&= \int_\Omega |u|^p\,\left(\int_{\mathbb{R}^N\setminus\Omega'} \frac{dy}{|x-y|^{N+s\,p}}\right)\,dx\\
&\le  \int_\Omega |u|^p\,\left(\int_{\mathbb{R}^N\setminus B_{d(x)}(x)} \frac{dy}{|x-y|^{N+s\,p}}\right)\,dx,
\end{split}
\]
where $d(x)=\mathrm{dist}(x,\mathbb{R}^N\setminus\Omega')$. By computing the last integral and using that $d\ge \mathrm{dist}(\Omega,\mathbb{R}^N\setminus\Omega')$. we get the desired conclusion.
\end{proof}
By enforcing the assumptions on $g$, a solution to \eqref{BVP} does exist. This is the content of the next result.
\begin{prop}[Existence]
\label{prop:death!}
Suppose $1<p<\infty$ and $0<s<1$. Let $\Omega\Subset \Omega'\subset\mathbb{R}^N$ be two open and bounded sets, $f\in L^q(\Omega)$, with
\[
q\ge (p^*_s)'\quad \mbox{ if } s\,p\not =N\qquad \mbox{ or }\qquad q>1\quad \mbox{ if }s\,p=N,
\]
and $g\in W^{s,p}(\Omega')\cap L^{p-1}_{sp}(\mathbb{R}^N)$. Then problem \eqref{BVP} admits a unique weak solution $u\in X^{s,p}_g(\Omega,\Omega')$. %Moreover, we have the energy estimate
%\[
%[u]_{W^{s,p}(\Omega')}\le C\,\left([g]_{W^{s,p}(\Omega')}+\|g\|_{L^{p}(B_R)}+\left(\int_{\mathbb{R}^N\setminus B_R} \frac{|g(y)|^{p-1}}{|y|^{N+s\,p}}\,dy\right)^\frac{1}{p-1}+\|f\|_{L^q(\Omega)}^\frac{1}{p-1}\right),
%\]
%where $B_R$ is an open ball centered at the origin, such that $\Omega'\subset B_R$ and $C=C(N,s,p,q,\Omega,\Omega')>0$ is a constant.
\end{prop}
\begin{proof}
We can adapt the proof of \cite[Theorem 1 \& Remark 3]{KKP}, concerning the case $f=0$. In what follows, whenever $X$ is a normed vector space, we denote by $X^*$ its topological dual. We endow the vector space
\[
X^{s,p}_0(\Omega,\Omega')=\{v\in W^{s,p}(\Omega')\cap L^{p-1}_{s\,p}(\mathbb{R}^N)\, :\, v=0 \ \mbox{ a.\,e. in }\mathbb{R}^N\setminus \Omega \},
\]
with the norm of $W^{s,p}(\Omega')$.
This is a separable reflexive Banach space. 
We now introduce the operator $\mathcal{A}:X^{s,p}_g(\Omega;\Omega')\to (W^{s,p}(\Omega'))^*$ defined by
\[
\begin{split}
\langle\mathcal{A}(v),\varphi\rangle&=\iint_{\Omega'\times\Omega'} \frac{J_p(v(x)-v(y))\,\Big(\varphi(x)-\varphi(y)\Big)}{|x-y|^{N+s\,p}}\,dx\,dy\\
&+2\,\iint_{\Omega\times (\mathbb{R}^N\setminus\Omega)'}\frac{J_p(v(x)-g(y))\,\varphi(x)}{|x-y|^{N+s\,p}}\,dx\,dy,\qquad v\in X^{s,p}_g(\Omega;\Omega'),\ \varphi\in W^{s,p}(\Omega'),
\end{split}
\]
where $\langle \cdot,\cdot\rangle$ denotes the relevant duality product. We know that $\mathcal{A}(v)\in (W^{s,p}(\Omega'))^*$ for every $v\in X^{s,p}_g(\Omega;\Omega')$ (see \cite[Remark 1]{KKP}). Moreover, we have that $\mathcal{A}$ has the following properties (see \cite[Lemma 3]{KKP}):
\begin{enumerate}
\item for every $v,u\in X^{s,p}_g(\Omega,\Omega')$, we have
\[
\langle \mathcal{A}(u)-\mathcal{A}(v),u-v\rangle\ge 0;
\]
\item if $\{u_n\}_{n\in\mathbb{N}}\subset X^{s,p}_g(\Omega,\Omega')$ converges in $W^{s,p}(\Omega')$ to $u\in X^{s,p}_g(\Omega,\Omega')$, then
\[
\lim_{n\to\infty} \langle \mathcal{A}(u_n)-\mathcal{A}(u),v\rangle=0\quad \text{for all }v\in W^{s,p}(\Omega');
\]
\item we have
\[
\lim_{\|u\|_{W^{s,p}(\Omega')}\to+\infty} \frac{\langle \mathcal{A}(u)-\mathcal{A}(g),u-g\rangle}{\|u-g\|_{W^{s,p}(\Omega')}}=+\infty.
\]
\end{enumerate} 
Finally, we introduce the modified functional
\[
\mathcal{A}_0(v):=\mathcal{A}(v+g),\qquad \mbox{ for every } v\in X^{s,p}_0(\Omega,\Omega').
\]
We observe that $\mathcal{A}_0:X^{s,p}_0(\Omega;\Omega')\to (W^{s,p}(\Omega'))^*$
and that $X^{s,p}_0(\Omega;\Omega')\subset W^{s,p}(\Omega')$ with continuous injection. This implies that there holds $(W^{s,p}(\Omega'))^*\subset (X^{s,p}_0(\Omega;\Omega'))^*$ as well, still with continuous injection. Thus $\mathcal{A}_0$ can be considered as an operator from $X^{s,p}_0(\Omega;\Omega')$ to its topological dual. Moreover, properties (1), (2) and (3) above imply that $\mathcal{A}_0$ is monotone, coercive and hemicontinuous (see \cite[Chapter II, Section 2]{Sh} for the relevant definitions). It is only left to observe that under the standing assumptions, the linear functional 
\[
T_f:v\mapsto \int_\Omega f\,v\,dx,\qquad v\in X^{s,p}_0(\Omega,\Omega'),
\] 
belongs to the topological dual of $X^{s,p}_0(\Omega,\Omega')$. Indeed, for every $v\in X^{s,p}_0(\Omega;\Omega')$ we have\footnote{We assume for simplicity that $s\,p\not =N$. The borderline case $s\,p=N$ can be treated in the same manner, we leave the details to the reader.}
\[
|T_f(v)|=\left|\int_\Omega f\, v\,dx\right|\le \|f\|_{L^q(\Omega)}\,\|v\|_{L^{q'}(\Omega)}\le |\Omega|^{\frac{1}{q'}-\frac{1}{p^*_s}}\,\|f\|_{L^q(\Omega)}\,\|v\|_{L^{p^*_s}(\Omega)},
\]
and the last term can be controlled by means of the Sobolev inequality in Proposition \ref{prop:sobpoin1}, thanks to the fact that $X^{s,p}_0(\Omega,\Omega')\subset W^{s,p}(\mathbb{R}^N)$ by Lemma \ref{lm:cazzata}. Then by \cite[Corollary 2.2]{Sh} we get existence of $v\in X^{s,p}_0(\Omega,\Omega')$ such that
\[
\langle \mathcal{A}_0(v),\varphi\rangle=\langle T_f,\varphi\rangle,\qquad \mbox{ for every }\varphi \in X^{s,p}_0(\Omega,\Omega').
\]
By definition, this is equivalent to 
\[
\langle \mathcal{A}(v+g),\varphi\rangle=\langle T_f,\varphi\rangle,\qquad \mbox{ for every }\varphi \in X^{s,p}_0(\Omega,\Omega'),
\]
i.e.
\[
\begin{split}
\iint_{\Omega'\times\Omega'} &\frac{J_p(v(x)+g(x)-v(y)-g(y))\,\Big(\varphi(x)-\varphi(y)\Big)}{|x-y|^{N+s\,p}}\,dx\,dy\\
&+2\,\iint_{\Omega\times (\mathbb{R}^N\setminus\Omega')}\frac{J_p(v(x)+g(x)-g(y))\,\varphi(x)}{|x-y|^{N+s\,p}}\,dx\,dy=\int_\Omega f\,\varphi\,dx.
\end{split}
\]
By observing that $v=0$ in $\mathbb{R}^N\setminus\Omega$ and that
\[
\begin{split}
2\,\iint_{\Omega\times (\mathbb{R}^N\setminus\Omega')}&\frac{J_p(v(x)+g(x)-g(y))\,\varphi(x)}{|x-y|^{N+s\,p}}\,dx\,dy\\
&=\iint_{\Omega\times (\mathbb{R}^N\setminus\Omega')}\frac{J_p(v(x)+g(x)-v(y)-g(y))\,\varphi(x)}{|x-y|^{N+s\,p}}\,dx\,dy\\
&-\iint_{(\mathbb{R}^N\setminus\Omega')\times \Omega}\frac{J_p(v(x)+g(x)-v(y)-g(y))\,\varphi(y)}{|x-y|^{N+s\,p}}\,dx\,dy,
\end{split}
\]
this is the same as \eqref{equation}.
Then $v+g$ is the desired solution. Uniqueness now follows from the {\it strict} monotonicity of the operator $\mathcal{A}_0$.
\end{proof}
\begin{oss}[Variational solutions]
We observe that under the slightly stronger assumption $g\in W^{s,p}(\Omega')\cap L^p_{s\,p}(\mathbb{R}^N)$, existence of the solution to \eqref{BVP} can be obtained by solving the following strictly convex variational problem
\[
\min\left\{\mathcal{F}(v)\, :\, v\in X^{s,p}_g(\Omega)\cap L^p_{s\,p}(\mathbb{R}^N)\right\},
\]
where the functional $\mathcal{F}$ is defined by
\[
\mathcal{F}(v)=\frac{1}{p}\,\iint_{\Omega'\times \Omega'} \frac{|v(x)-v(y)|^p}{|x-y|^{N+s\,p}}\,dx\,dy+\frac{2}{p}\,\iint_{\Omega\times (\mathbb{R}^N\setminus\Omega')}\frac{|v(x)-g(y)|^p}{|x-y|^{N+s\,p}}\,dx\,dy-\int_\Omega f\,v\,dx.
\]
Existence of a minimizer can be easily inferred by using the Direct Methods in the Calculus of Variations, see for example \cite{Gi}.
\end{oss}

\section{Basic regularity estimates}\label{sec:basic}

\subsection{Known results}
In this section we list some known regularity results for weak solutions. We start with the following important result, contained in \cite[Theorem 1.2]{DKP}. The proof in \cite{DKP} is performed under the stronger assumption that the boundary datum $g$ is in $W^{s,p}(\mathbb{R}^N)$. However, a closer inspection of the proof reveals that this is not needed, see also \cite[Remarks 1.1 \& 2.1]{KMS}. 
\par
We use the standard notation
\[
\mathrm{osc}_E \psi=\sup_E \psi-\inf_E \psi.
\]
\begin{teo}[H\"older continuity for $(s,p)-$harmonic functions]
\label{teo:DKP}
Let $1<p<\infty$, $0<s<1$ and $E\Subset E'\subset \mathbb{R}^N$ be open and bounded sets. Let $v\in X^{s,p}_g(E,E')$ be the solution to
\[
\left\{\begin{array}{rcll}
(-\Delta_p)^s\,v&=&0,&\mbox{ in }E,\\
v&=&g,& \mbox{ in }\mathbb{R}^N\setminus E,
\end{array}
\right.
\]
for some $g\in W^{s,p}(E')\cap L^{p-1}_{s\,p}(\mathbb{R}^N)$. Then there exists $\alpha=\alpha(N,s,p)>0$ such that for every $B_{2R}(x_0)\Subset E$ we have
\begin{equation}
\label{oscillation0}
\mathrm{osc}_{B_r(x_0)} v\le C\,\left(\frac{r}{R}\right)^\alpha\,\left[\left(\fint_{B_{2R}(x_0)} |v|^p\,dx\right)^\frac{1}{p}+\mathrm{Tail}_{p-1,s\,p}\left(v;x_0,\frac{R}{2}\right)\right],
\end{equation}
for every $0<r\le R$. 
In particular, 
\begin{equation}
\label{oscillation}
\int_{B_r(x_0)} \left|v-\overline v_{x_0,r}\right|^p\,dx\le C\,r^N\,\left(\frac{r}{R}\right)^{\alpha\,p}\,\left[\fint_{B_{2R}(x_0)} |v|^p\,dx+\mathrm{Tail}_{p-1,s\,p}\left(v;x_0,\frac{R}{2}\right)^p\right].
\end{equation}
Here $C=C(N,s,p)>0$.
\end{teo}
\begin{proof}
The estimate \eqref{oscillation0} is proved in \cite[Theorem 1.2]{DKP}. Here we just show how to get the excess decay estimate \eqref{oscillation} from \eqref{oscillation0}. We have
\[
\begin{split}
\int_{B_r(x_0)} \left|v-\overline v_{x_0,r}\right|^p&=\int_{B_r(x_0)} \left|\fint_{B_r(x_0)} \Big(v(x)-v(y)\Big)\,dy\right|^p\,dx\\
&\le \int_{B_r(x_0)} \fint_{B_r(x_0)} |v(x)-v(y)|^p\,dy\,dx\\
&\le \int_{B_r(x_0)} \left(\mathrm{osc}_{B_r(x_0)} v\right)^p\,dx=\omega_N\, r^N\,\left(\mathrm{osc}_{B_r(x_0)} v\right)^p.
\end{split}
\]
By using \eqref{oscillation0}, we get the desired conclusion.
\end{proof}
A local $L^\infty$-bound for solutions with a right-hand side will be needed. The result below is \cite[Theorem 3.8]{BP}. Just like for the previous result, the stronger assumption $g\in W^{s,p}(\mathbb{R}^N)$ stated in \cite{BP} is not needed.
\begin{teo}[Local boundedness]
\label{teo:loc_bound}
Let $1<p<\infty$ and $0<s<1$. Let $E\Subset E'\subset \mathbb{R}^N$ be open and bounded sets. For $f\in L^q(E)$ with 
\[
\left\{\begin{array}{lr}
q>\dfrac{N}{s\, p},& \mbox{ if } s\,p\le N,\\
&\\
q\ge 1,& \mbox{ if }s\,p>N,
\end{array}
\right.
\]
take $u\in X^{s,p}_g(E,E')$ to be the weak solution to
\[
\left\{\begin{array}{rcll}
(-\Delta_p)^s\,u&=&f,&\mbox{ in }E,\\
u&=&g,& \mbox{ in }\mathbb{R}^N\setminus E,
\end{array}
\right.
\]
for some $g\in W^{s,p}(E')\cap L^{p-1}_{s\,p}(\mathbb{R}^N)$. For every $R>0$ such that
$B_{R}(x_0)\Subset E$ and every $0<\sigma<1$, the following scaling invariant estimate holds
\[
\begin{split}
\|u\|_{L^\infty(B_{\sigma R}(x_0))}&\le C\,\left[\left(\fint_{B_{R}(x_0)} |u|^p\, dx\right)^\frac{1}{p}+\mathrm{Tail}_{p-1,s\,p}(u;x_0,\sigma R)+\left(R^{s\,p-\frac{N}{q}}\,\|f\|_{L^q(B_{R}(x_0))}\right)^\frac{1}{p-1}\right],
\end{split}
\]
where $C=C(N,s,p,q,\sigma)>0$.
\end{teo}
\begin{oss}
\label{oss:stupid}
From the previous result, we get in particular that if $f\in L^q_{\rm loc}(\Omega)$ with
\[
\left\{\begin{array}{lr}
q>\dfrac{N}{s\,p},& \mbox{ if } s\,p\le N,\\
&\\
q\ge 1,& \mbox{ if }s\,p>N,
\end{array}
\right.
\]
and $u\in W^{s,p}_{\rm loc}(\Omega)\cap L^{p-1}_{s\,p}(\mathbb{R}^N)$ is a local weak solution of $(-\Delta_p)^s u=f$ in $\Omega$, then $u\in L^\infty_{\rm loc}(\Omega)$. Indeed, for every $B_{2R}(x_0)\Subset \Omega$, we have that $u$ is the weak solution to
\[
\left\{\begin{array}{rcll}
(-\Delta_p)^s\,u&=&f,&\mbox{ in }B_{R}(x_0),\\
u&=&u,& \mbox{ in }\mathbb{R}^N\setminus B_{R}(x_0),
\end{array}
\right.
\]
and the boundary datum (which is $u$ itself) lies in $W^{s,p}(B_{2R}(x_0))\cap L^{p-1}_{s\,p}(\mathbb{R}^N)$. Then Theorem~\ref{teo:loc_bound} implies $u\in L^\infty(B_{\sigma R}(x_0))$, for every $0<\sigma<1$.
\end{oss}

\subsection{H\"older regularity for non-homogeneous equations}
Here we prove that local weak solutions of
\[
(-\Delta_p)^s u=f,
\]
are in $C^{0,\beta}$ for some $0<\beta<s$, provided that $f$ is integrable enough. We establish the result by transferring the excess decay estimate of Theorem~\ref{teo:DKP} from an $(s,p)-$harmonic function to the solution $u$.
\begin{lm}
\label{lm:1}
Let $p\ge 2$, $0<s<1$ and $\Omega\subset\mathbb{R}^N$ be an open and bounded set. For $f\in L^q_{\rm loc}(\Omega)$, with
\[
q\ge (p^*_s)'\quad \mbox{ if } s\,p\not =N\qquad \mbox{ or }\qquad q>1\quad \mbox{ if }s\,p=N,
\]
we consider a local weak solution $u\in W^{s,p}_{\rm loc}(\Omega)\cap L^{p-1}_{s\,p}(\mathbb{R}^N)$ of the equation
\[
(-\Delta_p)^s u=f,\qquad \mbox{ in }\Omega.
\]
We take a pair of concentric balls $B\Subset B'\Subset\Omega $, and define $v\in X^{s,p}_u(B,B')$ the unique weak solution of 
\[
\left\{\begin{array}{rcll}
(-\Delta_p)^s\,v&=&0,&\mbox{ in }B,\\
v&=&u,& \mbox{ in }\mathbb{R}^N\setminus B.
\end{array}
\right.
\]
Then for $s\,p\not =N$ we have
\begin{equation}
\label{prima}
[u-v]^p_{W^{s,p}(\mathbb{R}^N)}\le C\,|B|^{\frac{p'}{q'}-\frac{p}{p-1}\,\frac{N-s\,p}{N\,p}}\,\left(\int_B |f|^{q}\,dx\right)^\frac{p'}{q},
\end{equation}
and also
\begin{equation}
\label{seconda}
\fint_{B} |u-v|^p\,dx\le  C\,|B|^{\frac{p'}{q'}-\frac{p}{p-1}\,\frac{N-s\,p}{N\,p}+\frac{s\,p}{N}-1}\,\left(\int_B |f|^{q}\,dx\right)^\frac{p'}{q},
\end{equation}
for a constant $C=C(N,p,s)>0$.
\par
In the case $s\,p=N$, we have the same estimates with $N\,p/(N-s\,p)$ replaced by an arbitrary exponent $m<\infty$ and the constant $C$ depending on $m$ as well.
\end{lm}
\begin{proof}
We first observe that $v$ exists thanks to Proposition \ref{prop:death!}, since $u\in W^{s,p}(B')\cap L^{p-1}_{s\,p}(\mathbb{R}^N)$.
Subtracting the weak formulations of the equations solved by $u$ and $v$, we get
\[
\iint_{\mathbb{R}^N\times\mathbb{R}^N} \frac{\Big(J_p(u(x)-u(y))-J_p(v(x)-v(y))\Big)\,\big(\varphi(x)-\varphi(y)\big)}{|x-y|^{N+s\,p}}\,dx\,dy=\int f\,\varphi\,dx,
\]
for every $\varphi\in X^{s,p}_0(B,B')$. With $\varphi=u-v$ (which is admissible by definition of $X^{s,p}_u(B,B')$), we use inequality \eqref{xxx} to get\footnote{As usual, for simplicity we only treat the case $s\,p<N$. The case $s\,p\ge N$ can be handled similarly.}
\[
\begin{split}
[u-v]^p_{W^{s,p}(\mathbb{R}^N)}\le C\,\int_{B} f\,(u-v)\,dx&\le C\,\left(\int_{B} |f|^{q}\,dx\right)^\frac{1}{q}\,\left(\int_B |u-v|^{q'}\,dx\right)^\frac{1}{q'}\\
&\le C\,|B|^{\frac{1}{q'}-\frac{1}{p^*_s}}\,\left(\int_{B} |f|^{q}\,dx\right)^\frac{1}{q}\,\|u-v\|_{L^{p^*_s}(B)}\\
&\le C\,|B|^{\frac{1}{q'}-\frac{1}{p^*_s}}\,\left(\int_{B} |f|^q\,dx\right)^\frac{1}{q}\,[u-v]_{W^{s,p}(\mathbb{R}^N)},
\end{split}
\]
where we used the Sobolev inequality for the space $X^{s,p}_0(B,B')$, see Proposition \ref{prop:sobpoin1} and Lemma \ref{lm:cazzata}. By simplifying a factor $[u-v]_{W^{s,p}(\mathbb{R}^N)}$, we obtain
\[
[u-v]^{p-1}_{W^{s,p}(\mathbb{R}^N)}\le C\,|B|^{\frac{1}{q'}-\frac{1}{p^*_s}}\,\left(\int_{B} |f|^{q}\,dx\right)^\frac{1}{q},
\] 
which in turn gives \eqref{prima}. 
\par
Estimate \eqref{seconda} now follows by applying Poincar\'e's inequality in \eqref{prima}, see again Proposition \ref{prop:sobpoin1}.
\end{proof}
\begin{lm}[Decay transfer]
\label{lm:transfer}
Let $p\ge 2$, $0<s<1$ and $\Omega\subset\mathbb{R}^N$ be an open and bounded set. For $f\in L^{q}_{\rm loc}(\Omega)$
with
\[
q\ge (p^*_s)'\quad \mbox{ if } s\,p\not =N\qquad \mbox{ or }\qquad q>1\quad \mbox{ if }s\,p=N,
\]
we consider a local weak solution $u\in W^{s,p}_{\rm loc}(\Omega)\cap L^{p-1}_{s\,p}(\mathbb{R}^N)$ of the equation
\[
(-\Delta_p)^s u=f,\qquad \mbox{ in }\Omega.
\]
If $B_{4R}(x_0)\Subset\Omega$, then we have the excess decay estimate
\[
\begin{split}
\fint_{B_r(x_0)} |u-\overline u_{x_0,r}|^p\,dx&\le C\,\left(\frac{R}{r}\right)^N\,R^\gamma\,\|f\|_{L^{q}(B_{4R(x_0)})}^{p'}\\
&+C\,\left(\frac{r}{R}\right)^{\alpha\,p}\,\left[R^\gamma\,\|f\|_{L^{q}(B_{4R}(x_0))}^{p'}+\fint_{B_{4R}(x_0)} |u|^p\,dx+\mathrm{Tail}_{p-1,sp}\left(u;x_0,4R\right)^p\right],
\end{split}
\]
for every $0<r\le R$. Here
\begin{equation}
\label{gamma}
\gamma:=\left\{\begin{array}{cc}
s\,p\,p'+N\,\left(\dfrac{p'}{q'}-\dfrac{1}{p-1}-1\right),& \mbox{ if }s\,p\not=N,\\
&\\
N\,p'\,\left(\dfrac{1}{q'}-\dfrac{1}{m}\right),& \mbox{ for an arbitrary } q'< m<\infty, \mbox{ if } s\,p=N. 
\end{array}
\right.
\end{equation}
and $C=C(N,s,p,q,m)>0$.
The exponent $\alpha$ is the same as in Theorem~\ref{teo:DKP}.
\end{lm}
\begin{proof}
As always, for simplicity we work with the case $s\,p<N$.
Take a ball $B_{4R}(x_0)\Subset\Omega$, and define $v\in X^{s,p}_u(B_{3R}(x_0),B_{4R}(x_0))$ the unique solution to the Dirichlet boundary value problem
\[
\left\{\begin{array}{rcll}
(-\Delta_p)^s\,v&=&0,&\mbox{ in }B_{3R}(x_0),\\
v&=&u,& \mbox{ in }\mathbb{R}^N\setminus B_{3R}(x_0).
\end{array}
\right.
\]
We start by observing that
\[
|\overline u_{x_0,r}-\overline v_{x_0,r}|^p=\left|\fint_{B_r(x_0)} [u-v]\,dx\right|^p\le \fint_{B_r(x_0)} |u-v|^p\,dx,
\]
thus 
\[
\begin{split}
\fint_{B_r(x_0)} |u-\overline u_{x_0,r}|^p\,dx&\le C\, \fint_{B_r(x_0)} |u-v|^p\,dx+C\,\fint_{B_r(x_0)} |\overline u_{x_0,r}-\overline v_{x_0,r}|^p\,dx\\
&+C\,\fint_{B_r(x_0)} |v-\overline v_{x_0,r}|^p\,dx\\
&\le 2\,C\, \fint_{B_r(x_0)} |u-v|^p\,dx+C\,\fint_{B_r(x_0)} |v-\overline v_{x_0,r}|^p\,dx,
\end{split}
\]
where $C=C(p)>0$. By using Lemma \ref{lm:1} and Theorem~\ref{teo:DKP} for $v$ with $E=B_{3R}(x_0)$ and $E'=B_{4R}(x_0)$, we obtain for $0<r\le R$
\[
\begin{split}
\fint_{B_r(x_0)} |u-\overline u_{x_0,r}|^p\,dx&\le C\, \left(\frac{R}{r}\right)^N\,R^{s\,p+N\,\left(\frac{p'}{q'}-\frac{p'}{p^*_s}-1\right)}\,\left(\int_{B_{3R}(x_0)} |f|^{q}\,dx\right)^\frac{p'}{q}\\
&+C\,\left(\frac{r}{R}\right)^{\alpha\,p}\,\left[\fint_{B_{2R}(x_0)} |v|^p\,dx+\mathrm{Tail}_{p-1,sp}\left(v;x_0,\frac{R}{2}\right)^p\right],
\end{split}
\]
with $C=C(N,s,p)$. For the second term, we use Lemma \ref{lm:tails} with $x_1=x_0$ and deduce
\[
\begin{split}
\mathrm{Tail}_{p-1,s\,p}\left(v;x_0,\frac{R}{2}\right)^p&\le C\left(\mathrm{Tail}_{p-1,sp}(u;x_0,4R)^p+\fint_{B_{4R}(x_0)} |v|^p\,dx\right).
\end{split}
\]
Finally, we use Lemma \ref{lm:1} again and obtain
\[
\begin{split}
\fint_{B_{4R}(x_0)} |v|^p\,dx&\le C\,\fint_{B_{4R}(x_0)} |v-u|^p\,dx+C\,\fint_{B_{4R}(x_0)} |u|^p\,dx\\
&\le C\,R^{s\,p+N\,\left(\frac{p'}{q'}-\frac{p'}{p^*_s}-1\right)}\,\left(\int_{B_{4R}(x_0)} |f|^{q}\,dx\right)^\frac{p'}{q}+C\,\fint_{B_{4R}(x_0)} |u|^p\,dx.
\end{split}
\]
Here $C=C(N,s,p)$. This concludes the proof.
\end{proof}
We can now obtain H\"older regularity of solutions to the non-homogeneous equation.
\begin{teo}
\label{teo:holderf}
Let $p\ge 2$ and $0<s<1$. For $f\in L^q_{\rm loc}(\Omega)$ with 
\[
\left\{\begin{array}{lr}
q>\max\left\{\dfrac{N}{s\,p},\,1\right\},& \mbox{ if } s\,p\le N,\\
q\ge 1,& \mbox{ if }s\,p>N,
\end{array}
\right.
\]
we consider a local weak solution $u\in W^{s,p}_{\rm loc}(\Omega)\cap L^{p-1}_{s\,p}(\mathbb{R}^N)$ of the equation
\[
(-\Delta_p)^s u=f,\qquad \mbox{ in }\Omega.
\]
We set
\[
\beta=\frac{\gamma}{N+\alpha\,p+\gamma}\,\alpha,
\] 
where $\gamma$ is as in \eqref{gamma} and $\alpha$ as in Theorem \ref{teo:DKP}. Then $u\in C^{0,\beta}_{\rm loc}(\Omega)$. 
\par
More precisely, for every ball $B_{R_0}(z)\Subset\Omega$ we have the estimate
\[
\begin{split}
[u]_{C^{0,\beta}(B_{R_0}(z))}^p%&\leq C\sup_{x_0\in B_{R_0}(z),\, r>0} r^{-N-\beta\,p}\,\int_{B_r(x_0)\cap B_{R_0}(z)} |u-\overline u_{x_0,r}|^p\,dx\\
&\le C\,\left[\|f\|_{L^{q}(B_{R_1}(z))}^{p'}+\|u\|^p_{L^\infty(B_{R_1}(z))}+\mathrm{Tail}_{p-1,sp}\left(u;z,R_1\right)^{p}+1\right],
\end{split}
\]
where 
\[
R_1=R_0+\frac{\mathrm{dist}(B_{R_0}(z),\partial\Omega)}{2}.
\] 
The constant $C$ depends only on $N,p,s,q,R_0$ and $\mathrm{dist}(B_{R_0}(z),\partial\Omega)$.
\end{teo}
\begin{proof}
Take a ball $B_{R_0}(z)\Subset\Omega$ and set
\[
\mathrm{d}=\mathrm{dist}(B_{R_0}(z),\partial\Omega)>0\qquad \mbox{ and }\qquad R_1=\frac{\mathrm{d}}{2}+R_0.
\] 
Choose a point $x_0\in B_{R_0}(z)$ and consider the ball $B_{4R}(x_0)$ with $R<\min\{1,\mathrm{d}/8\}$. We observe that $q>(p^*_s)'$. If\footnote{For the conformal case $s\,p=N$, it is sufficient to reproduce the proof, by using Lemma \ref{lm:transfer} with an exponent $m>q'$.} $s\,p\not =N$, we may apply Lemma \ref{lm:transfer} and obtain
\[
\begin{split}
\fint_{B_r(x_0)} |u-\overline u_{x_0,r}|^p\,dx&\le C\,\left(\frac{R}{r}\right)^N\,R^{\gamma}\,\|f\|_{L^{q}(B_{4R(x_0)})}^{p'}\\
&+C\,\left(\frac{r}{R}\right)^{\alpha\,p}\,\left[R^\gamma\,\|f\|_{L^{q}(B_{4R}(x_0))}^{p'}+\fint_{B_{4R}(x_0)} |u|^p\,dx+\mathrm{Tail}_{p-1,s\,p}\left(u;x_0,4R\right)^p\right]\\
&\le C\,\left(\frac{R}{r}\right)^N\,R^{\gamma}\,\|f\|_{L^q(B_{R_1}(z))}^{p'}\\
&+C\,\left(\frac{r}{R}\right)^{\alpha\,p}\,\left[\mathrm{d}^\gamma\,\|f\|_{L^{q}(B_{R_1}(z))}^{p'}+\|u\|^p_{L^\infty(B_{R_1}(z))}+\mathrm{Tail}_{p-1,s\,p}\left(u;x_0,4R\right)^p\right],
\end{split}
\]
for every $0<r\le R<\min\{1,\mathrm{d}/8\}$. Here we used that $u\in L^\infty_{\rm loc}(\Omega)$, thanks to the hypothesis on $f$ and Theorem~\ref{teo:loc_bound}.
For the tail term, we can use \eqref{paura} of Lemma \ref{lm:tails} with $m=\infty$ with $B_{4R}(x_0)\subset B_{R_1}(z)$, then for $R<\min\{1,\delta/8\}$ we have
\[
\begin{split}
\mathrm{Tail}_{p-1,s\,p}\left(u;x_0,4R\right)^{p-1}&\le \left(\frac{4\,R}{R_1}\right)^{s\,p}\,\left(\frac{4\,R}{R_1-|x_0-z|}\right)^{N+s\,p}\,\mathrm{Tail}_{p-1,s\,p}\left(u;z,R_1\right)^{p-1}+C\,\|u\|^{p-1}_{L^\infty(B_{R_1}(z))}\\
&\le \mathrm{Tail}_{p-1,s\,p}\left(u;z,R_1\right)^{p-1}+C\,\|u\|_{L^\infty(B_{R_1}(z))}^{p-1}.
\end{split}
\]
In the second estimate, we used that
\[
\frac{4\,R}{R_1}<\frac{\dfrac{\mathrm{d}}{2}}{R_0+\dfrac{\mathrm{d}}{2}}<1\qquad \mbox{ and }\qquad \frac{4\,R}{R_1-|x_0-z|}\le \frac{4\,R}{R_1-R_0}<1.
\]
In conclusion, 
\[
\begin{split}
\fint_{B_r(x_0)} |u-\overline u_{x_0,r}|^p\,dx &\le C\,\left(\frac{R}{r}\right)^N\,R^{\gamma}\,\|f\|_{L^q(B_{R_1}(z))}^{p'}\\
&+C\,\left(\frac{r}{R}\right)^{\alpha\,p}\,\left[\mathrm{d}^\gamma\,\|f\|_{L^{q}(B_{R_1}(z))}^{p'}+\|u\|^p_{L^\infty(B_{R_1}(z))}+\mathrm{Tail}_{p-1,s\,p}\left(u;z,R_1\right)^{p}\right],
\end{split}
\]
possibly with a different constant $C=C(N,s,p,q)>0$.
Observe that thanks to the hypothesis on $q$, we have 
\[
\gamma:=s\,p+N\,\left(\dfrac{p'}{q'}-\dfrac{p'}{p^*_s}-1\right)>0.
\]
Then we make the choice $r=R^\theta$, with 
\[
\theta=1+\frac{\gamma}{N+\alpha\,p}.
\]
For every $0<r\le \min\{1,(\mathrm{d}/8)^\theta\}$ and every $x_0\in B_{R_0}(z)$, we obtain
\[
r^{-\beta\,p}\,\fint_{B_r(x_0)\cap B_{R_0}(z)} |u-\overline u_{x_0,r}|^p\,dx\le C\,\left[(\mathrm{d}^\gamma+1)\,\|f\|_{L^{q}(B_{R_1}(z))}^{p'}+\|u\|^p_{L^\infty(B_{R_1}(z))}+\mathrm{Tail}_{p-1,s\,p}\left(u;z,R_1\right)^p\right],
\]
where
\[
\beta=\frac{\gamma\,\alpha}{N+\alpha\,p+\gamma}>0.
\]
This shows that $u$ belongs to the Campanato space $\mathcal{L}^{p,N+\beta\,p}(B_{R_0}(z))$, which is isomorphic to $C^{0,\beta}(\overline{B_{R_0}(z)})$. This implies the desired conclusion.
\end{proof}

\section{Almost \texorpdfstring{$C^s$}{Cs}-regularity: homogeneous case}
\label{sec:almost}
In what follows, we use the notation $B_R$ for the $N-$dimensional open ball of radius $R$, centered at the origin. The cornerstone of our main result is the following integrability gain for the $s-$derivative
\[
\frac{\delta^2_h u}{|h|^s},
\]
of an $(s,p)-$harmonic function.
\begin{prop}
\label{prop:improve} 
Assume $p\ge 2$ and $0<s<1$. Let  $u\in W^{s,p}_{\rm loc}(B_2)\cap L^{p-1}_{s\,p}(\mathbb{R}^N)$ be a local weak solution of $(-\Delta_p)^s u=0$ in $B_2$. Suppose that 
\begin{equation}
\label{bounds}
\|u\|_{L^\infty(B_1)}\leq 1\qquad \mbox{ and }\qquad \mathrm{Tail}_{p-1,s\,p}(u;0,1)^{p-1}=\int_{\mathbb{R}^N\setminus B_1} \frac{|u(y)|^{p-1}}{|y|^{N+s\,p}}\,  dy\leq 1,
\end{equation} 
and that for some $q\ge p$ and $0<h_0<1/10$ we have
\[
\sup_{0<|h|< h_0}\left\|\frac{\delta^2_h u }{|h|^s}\right\|_{L^q(B_1)}^q<+\infty.
\]
Then for every radius $4\,h_0<R\le 1-5\,h_0$
we have
$$
\sup_{0<|h|< h_0}\left\|\frac{\delta^2_h u}{|h|^{s}}\right\|_{L^{q+1}(B_{R-4\,h_0})}^{q+1}\leq C\,\left(\sup_{0<|h|< h_0}\left\|\frac{\delta^2_h u }{|h|^s}\right\|_{L^q(B_{R+4\,h_0})}^q+1\right).
$$
Here $C=C(N,s,p,q,h_0)>0$ and $C\nearrow +\infty$ as $h_0\searrow 0$.
\end{prop}
\begin{proof}
We first observe that $u\in L^\infty_{\rm loc}(B_2)$, thanks to Theorem~\ref{teo:loc_bound} and Remark \ref{oss:stupid}. Thus the first hypothesis does make sense. We divide the proof into five steps.
\vskip.2cm\noindent
{\noindent}{\bf Step 1: Discrete differentiation of the equation.}
For notational simplicity, set 
\[
r=R-4\,h_0\qquad \mbox{ and }\qquad d\mu(x,y)=\frac{dx\,dy}{|x-y|^{N+s\,p}}.
\]
Take a test function $\varphi\in W^{s,p}(B_R)$, vanishing outside $B_{(R+r)/2}$.
By testing \eqref{equation} with $\varphi_{-h}$ for $h\in\mathbb{R}^N\setminus\{0\}$ with $|h|<h_0$ and then changing variables, we get
\begin{equation}
\label{equationh}
\iint_{\mathbb{R}^N\times\mathbb{R}^N} \Big(J_p(u_h(x)-u_h(y))\Big)\,\Big(\varphi(x)-\varphi(y)\Big)\,d\mu(x,y)=0.
\end{equation}
We subtract \eqref{equation} from \eqref{equationh} and divide by $|h|>0$ to get
\begin{equation}
\label{differentiated}
\begin{split}
\int_{\mathbb{R}^N\times\mathbb{R}^N} \frac{J_p(u_h(x)-u_h(y))-J_p(u(x)-u(y))}{|h|}\,\Big(\varphi(x)-\varphi(y)\Big)\,d\mu=0,
\end{split}
\end{equation}
for every $\varphi\in W^{s,p}(B_R)$, vanishing outside $B_{(R+r)/2}$. Let $\beta\ge 1$ and $\vartheta>0$, and use \eqref{differentiated} with the test function
\[
\varphi=J_{\beta+1}\left(\frac{u_h-u}{|h|^\vartheta}\right)\,\eta^p,\qquad 0<|h|< h_0,
\]
where $\eta$ is a non-negative standard Lipschitz cut-off function such that
\[
\eta\equiv 1 \quad \mbox{ on } B_r,\qquad \eta\equiv 0 \quad \mbox{ on } \mathbb{R}^N\setminus B_{(R+r)/2}.\qquad |\nabla \eta|\le \frac{C}{R-r}=\frac{C}{4\,h_0}.
\] 
Note that these assumptions implies
$$
\left|\frac{\delta_h \eta}{|h|}\right|\leq \frac{C}{h_0}.
$$
We get
\[
\begin{split}
\iint_{\mathbb{R}^N\times\mathbb{R}^N}& \frac{\Big(J_p(u_h(x)-u_h(y))-J_p(u(x)-u(y))\Big)}{|h|^{1+\vartheta\,\beta}}\\
&\times\Big(J_{\beta+1}(u_h(x)-u(x))\,\eta(x)^p-J_{\beta+1}(u_h(y)-u(y))\,\eta(y)^p\Big)\,d\mu=0.
\end{split}
\]
The double integral is now divided into three pieces:
\[
\begin{split}
\mathcal{I}_1:=\iint_{B_R\times B_R}& \frac{\Big(J_p(u_h(x)-u_h(y))-J_p(u(x)-u(y))\Big)}{|h|^{1+\vartheta\,\beta}}\\
&\times\Big(J_{\beta+1}(u_h(x)-u(x))\,\eta(x)^p-J_{\beta+1}(u_h(y)-u(y))\,\eta(y)^p\Big)d\mu,
\end{split}
\]
\[
\begin{split}
\mathcal{I}_2:=\iint_{B_\frac{R+r}{2}\times (\mathbb{R}^N\setminus B_R)}& \frac{\Big(J_p(u_h(x)-u_h(y))-J_p(u(x)-u(y))\Big)}{|h|^{1+\vartheta\,\beta}}\,J_{\beta+1}(u_h(x)-u(x))\,\eta(x)^p\,d\mu,
\end{split}
\]
and
\[
\begin{split}
\mathcal{I}_3:=-\iint_{(\mathbb{R}^N\setminus B_R)\times B_\frac{R+r}{2}}& \frac{\Big(J_p(u_h(x)-u_h(y))-J_p(u(x)-u(y))\Big)}{|h|^{1+\vartheta\,\beta}}\,J_{\beta+1}(u_h(y)-u(y))\,\eta(y)^p\,d\mu,
\end{split}
\]
where we used that $\eta$ vanishes identically outside $B_{(R+r)/2}$.
\par
We will estimate $\mathcal{I}_1$ in what follows. We start by observing that
\[
\begin{split}
J_{\beta+1}(u_h(x)-u(x))\,\eta(x)^p&-J_{\beta+1}(u_h(y)-u(y))\,\eta(y)^p\\
&=\frac{\Big(J_{\beta+1}(u_h(x)-u(x))-J_{\beta+1}(u_h(y)-u(y))\Big)}{2}\,\Big(\eta(x)^p+\eta(y)^p\Big)\\
&+\frac{\Big(J_{\beta+1}(u_h(x)-u(x))+J_{\beta+1}(u_h(y)-u(y))\Big)}{2}\,(\eta(x)^p-\eta(y)^p).
\end{split}
\]
Thus
\[
\begin{split}
\Big(J_p(u_h(x)-u_h(y))&-J_p(u(x)-u(y))\Big)\,\Big(J_\beta(u_h(x)-u(x))\,\eta(x)^p-J_\beta(u_h(y)-u(y))\,\eta(y)^p\Big)\\
&\ge \Big(J_p(u_h(x)-u_h(y))-J_p(u(x)-u(y))\Big)\\
&\times \Big(J_{\beta+1}(u_h(x)-u(x))-J_{\beta+1}(u_h(y)-u(y))\Big)\,\left(\frac{\eta(x)^p+\eta(y)^p}{2}\right)\\
&-\Big|J_p(u_h(x)-u_h(y))-J_p(u(x)-u(y))\Big|\\
&\times \Big(|u_h(x)-u(x)|^\beta+|u_h(y)-u(y)|^\beta\Big)\,\left|\frac{\eta(x)^p-\eta(y)^p}{2}\right|.
\end{split}
\]
The first term in the last expression has a positive sign. For the negative term, we proceed like this: we use \eqref{lipschitz}, Young's inequality and then \eqref{erik2}
\[
\begin{split}
\Big|J_p(u_h(x)-u_h(y))&-J_p(u(x)-u(y))\Big|\Big(|u_h(x)-u(x)|^\beta+|u_h(y)-u(y)|^\beta\Big)\,\left|\frac{\eta(x)^p-\eta(y)^p}{2}\right|\\
&\le 2\,\frac{p-1}{p}\,\left(|u_h(x)-u_h(y)|^\frac{p-2}{2}+|u(x)-u(y)|^\frac{p-2}{2}\right)\\
&\times\left||u_h(x)-u_h(y)|^\frac{p-2}{2}\,(u_h(x)-u_h(y))-|u(x)-u(y)|^\frac{p-2}{2}\,(u(x)-u(y))\right|\\
&\times \Big(|u_h(x)-u(x)|^\beta+|u_h(y)-u(y)|^\beta\Big)\,\frac{\eta(x)^\frac{p}{2}+\eta(y)^\frac{p}{2}}{2}\,\left|\eta(x)^\frac{p}{2}-\eta(y)^\frac{p}{2}\right|\\
&\le \frac{C}{\varepsilon}\,\left(|u_h(x)-u_h(y)|^\frac{p-2}{2}+|u(x)-u(y)|^\frac{p-2}{2}\right)^2\\
&\times (|u_h(x)-u(x)|^{\beta+1}+|u_h(y)-u(y)|^{\beta+1})\,\left|\eta(x)^\frac{p}{2}-\eta(y)^\frac{p}{2}\right|^2\\
&+C\,\varepsilon\,\left||u_h(x)-u_h(y)|^\frac{p-2}{2}\,(u_h(x)-u_h(y))-|u(x)-u(y)|^\frac{p-2}{2}\,(u(x)-u(y))\right|^2\\
&\times \Big(|u_h(x)-u(x)|^{\beta-1}+|u_h(y)-u(y)|^{\beta-1}\Big)\,\Big(\eta(x)^p+\eta(y)^p\Big)\\
&\le \frac{C}{\varepsilon}\,\left(|u_h(x)-u_h(y)|^\frac{p-2}{2}+|u(x)-u(y)|^\frac{p-2}{2}\right)^2\\
&\times (|u_h(x)-u(x)|^{\beta+1}+|u_h(y)-u(y)|^{\beta+1})\,\left|\eta(x)^\frac{p}{2}-\eta(y)^\frac{p}{2}\right|^2\\
&+C\,\varepsilon\,\Big(J_p(u_h(x)-u_h(y))-J_p(u(x)-u(y))\Big)\\
&\times\Big(J_{\beta+1}(u_h(x)-u(x))-J_{\beta+1}(u_h(y)-u(y))\Big)\,\Big(\eta(x)^p+\eta(y)^p\Big),
\end{split}
\]
where $C=C(p)>0$ and $\varepsilon>0$ is arbitrary. By putting all the estimates together and choosing $\varepsilon$ sufficiently small, we then get for a different $C=C(p)>0$
\[
\begin{split}
\mathcal{I}_1&\ge \frac{1}{C}\,\iint_{B_R\times B_R}  \frac{\Big(J_p(u_h(x)-u_h(y))-J_p(u(x)-u(y))\Big)}{|h|^{1+\vartheta\,\beta}}\\
&\times\Big(J_{\beta+1}(u_h(x)-u(x))-J_{\beta+1}(u_h(y)-u(y))\Big)\,(\eta(x)^p+\eta(y)^p)\,d\mu\\
&-C\,\iint_{B_R\times B_R} \left(|u_h(x)-u_h(y)|^\frac{p-2}{2}+|u(x)-u(y)|^\frac{p-2}{2}\right)^2\,\left|\eta(x)^\frac{p}{2}-\eta(y)^\frac{p}{2}\right|^2\\
&\times \frac{|u_h(x)-u(x)|^{\beta+1}+|u_h(y)-u(y)|^{\beta+1}}{|h|^{1+\vartheta\,\beta}}\,d\mu.
\end{split}
\]
We can further estimate the positive term by using \eqref{erik}. This leads us to
\[
\begin{split}
\mathcal{I}_1&\ge c\iint_{B_R\times B_R} \left|\frac{|\delta_h u(x)|^\frac{\beta-1}{p}\,\delta_h u(x)}{|h|^\frac{1+\vartheta\,\beta}{p}}-\frac{|\delta_h u(y)|^\frac{\beta-1}{p}\,\delta_h u(y)}{|h|^\frac{1+\vartheta\,\beta}{p}}\right|^p\,(\eta(x)^p+\eta(y)^p)\,d\mu\\
&-C\,\iint_{B_R\times B_R} \left(|u_h(x)-u_h(y)|^\frac{p-2}{2}+|u(x)-u(y)|^\frac{p-2}{2}\right)^2\,\left|\eta(x)^\frac{p}{2}-\eta(y)^\frac{p}{2}\right|^2\,\frac{|\delta_h u(x)|^{\beta+1}+|\delta_h u(y)|^{\beta+1}}{|h|^{1+\vartheta\,\beta}}\,d\mu,
\end{split}
\]
where $c=c(p,\beta)>0$ and $C=C(p)>0$. We now observe that if we set for simplicity
\[
A=\frac{|\delta_h u(x)|^\frac{\beta-1}{p}\,\delta_h u(x)}{|h|^\frac{1+\vartheta\,\beta}{p}}\qquad \mbox{ and }\qquad B=\frac{|\delta_h u(y)|^\frac{\beta-1}{p}\,\delta_h u(y)}{|h|^\frac{1+\vartheta\,\beta}{p}}
\]
and then use the convexity of $\tau\mapsto \tau^p$, we have
\[
\begin{split}
\left|A\,\eta(x)-B\,\eta(y)\right|^p&=\left|(A-B)\,\frac{\eta(x)+\eta(y)}{2}+(A+B)\,\frac{\eta(x)-\eta(y)}{2}\right|^p\\
&\le \frac{1}{2}\,|A-B|^p\,\left|\eta(x)+\eta(y)\right|^p\\
&+\frac{1}{2}\, |A+B|^p\,\left|\eta(x)-\eta(y)\right|^p\\
&\le 2^{p-2}\,|A-B|^p\, (\eta(x)^p+\eta(y)^p)\\
&+2^{p-2}\, (|A|^p+|B|^p)\, |\eta(x)-\eta(y)|^p.
\end{split}
\]
We thus get the following lower bound for $\mathcal{I}_1$
\[
\begin{split}
\mathcal{I}_1\ge c& \left[\frac{|\delta_h u|^\frac{\beta-1}{p}\,\delta_h u}{|h|^\frac{1+\vartheta\,\beta}{p}}\,\eta\right]^p_{W^{s,p}(B_R)}\\
&-C\,\iint_{B_R\times B_R} \left(|u_h(x)-u_h(y)|^\frac{p-2}{2}+|u(x)-u(y)|^\frac{p-2}{2}\right)^2\,\left|\eta(x)^\frac{p}{2}-\eta(y)^\frac{p}{2}\right|^2\\
&\times \frac{|u_h(x)-u(x)|^{\beta+1}+|u_h(y)-u(y)|^{\beta+1}}{|h|^{1+\vartheta\,\beta}}\,d\mu\\
&-C\,\iint_{B_R\times B_R}\, \left(\frac{|\delta_h u(x)|^{\beta-1+p}}{|h|^{1+\vartheta\,\beta}}+\frac{|\delta_h u(y)|^{\beta-1+p}}{|h|^{1+\vartheta\,\beta}}\right)\, |\eta(x)-\eta(y)|^p\,d\mu,
\end{split}
\]
where $c=c(p,\beta)>0$ and $C=C(p,\beta)>0$. By recalling that $\mathcal{I}_1+\mathcal{I}_2+\mathcal{I}_3=0$ and using the estimate for $\mathcal{I}_1$, we arrive at
\begin{equation}\label{Ieq}
	\left[\frac{|\delta_h u|^\frac{\beta-1}{p}\,\delta_h u}{|h|^\frac{1+\vartheta\,\beta}{p}}\,\eta\right]^p_{W^{s,p}(B_R)}\leq C\,\Big(\mathcal{I}_{11}+\mathcal{I}_{12}+|\mathcal{I}_2|+|\mathcal{I}_3|\Big),\quad \mbox{ for }C=C(p,\beta)>0, 
\end{equation}
where we set 
\begin{equation}
\label{I11}
\begin{split}
\mathcal{I}_{11}&:=\,\iint_{B_R\times B_R} \left(|u_h(x)-u_h(y)|^\frac{p-2}{2}+|u(x)-u(y)|^\frac{p-2}{2}\right)^2\,\left|\eta(x)^\frac{p}{2}-\eta(y)^\frac{p}{2}\right|^2\, \frac{|\delta_h u(x)|^{\beta+1}+|\delta_h u(y)|^{\beta+1}}{|h|^{1+\vartheta\,\beta}}\,d\mu,
\end{split}
\end{equation}
and
\[
\begin{split}
\mathcal{I}_{12}&:=\,\iint_{B_R\times B_R}\, \left(\frac{|\delta_h u(x)|^{\beta-1+p}}{|h|^{1+\vartheta\,\beta}}+\frac{|\delta_h u(y)|^{\beta-1+p}}{|h|^{1+\vartheta\,\beta}}\right)\, |\eta(x)-\eta(y)|^p\,d\mu.
\end{split}
\]
\vskip.2cm
{\noindent}{\bf Step 2: Estimates of the local terms $\mathcal{I}_{11}$ and $\mathcal{I}_{12}$.}
We first treat $\mathcal{I}_{11}$, by estimating
$$
\iint_{B_R\times B_R} |u(x)-u(y)|^{p-2}\,\left|\eta(x)^\frac{p}{2}-\eta(y)^\frac{p}{2}\right|^2
\frac{|\delta_h u(x)|^{\beta+1}}{|h|^{1+\vartheta\,\beta}}\,d\mu.
$$
The other terms of $\mathcal{I}_{11}$ can be dealt with similarly. By using that $\eta$ is Lipschitz and that $p\ge 2$, we get
\[
\begin{split}
\iint_{B_R\times B_R} |u(x)-u(y)|^{p-2}\,&\left|\eta(x)^\frac{p}{2}-\eta(y)^\frac{p}{2}\right|^2
\frac{|\delta_h u(x)|^{\beta+1}}{|h|^{1+\vartheta\,\beta}}\,d\mu\\
&\le \frac{C}{h_0^2}\, \iint_{B_R\times B_R} \frac{|u(x)-u(y)|^{p-2}}{|x-y|^{N+s\,p-2}}\,
\frac{|\delta_h u(x)|^{\beta+1}}{|h|^{1+\vartheta\,\beta}}\,dx\,dy.
\end{split}
\]
For $p=2$, the last term reduces to 
\[
\begin{split}
\iint_{B_R\times B_R} \frac{1}{|x-y|^{N+2\,s-2}}\,
\frac{|\delta_h u(x)|^{\beta+1}}{|h|^{1+\vartheta\,\beta}}\,dx\,dy&=\int_{B_R}\left(\int_{B_R} \frac{dy}{|x-y|^{N+2\,s-2}}\right)\,
\frac{|\delta_h u(x)|^{\beta+1}}{|h|^{1+\vartheta\,\beta}}\,dx\\
&\le C\,\int_{B_R}\frac{|\delta_h u(x)|^{\beta+1}}{|h|^{1+\vartheta\,\beta}}\,dx\\
&\le C\,\|u\|_{L^\infty(B_{R+h_0})}\,\int_{B_R}\frac{|\delta_h u(x)|^{\beta}}{|h|^{1+\vartheta\,\beta}}\,dx\le C\,\int_{B_R}\frac{|\delta_h u(x)|^{\beta}}{|h|^{1+\vartheta\,\beta}}\,dx
\end{split}
\]
for some $C=C(N,s)>0$.
\par
For $p>2$, we take instead 
\begin{equation}
\label{epsilon}
\varepsilon=\min\left\{\frac{p-2}{2},\, \frac{1}{s}-1\right\}>0,
\end{equation}
then by Young's inequality with exponents $q/(p-2)$ and $q/(q-p+2)$ we have\footnote{We realize that the second term of $\mathcal{I}_{11}$ can be estimated by exactly the same quantity except that $B_R$ is replaced by $B_{R+h_0}$ in the first term. This is the reason why we changed $B_R$ to $B_{R+h_0}$ above, so that both terms in $\mathcal{I}_{11}$ enjoy the same estimate.}
\[
\begin{split}
\iint_{B_R\times B_R}& \frac{|u(x)-u(y)|^{p-2}}{|x-y|^{N+s\,p-2}}\, \frac{|\delta_h u(x)|^{\beta+1}}{|h|^{1+\vartheta\,\beta}}\, dx\,dy\\
&\leq C\,[u]_{W^{\frac{s\,(p-2-\varepsilon)}{p-2},q}(B_R)}^q+\left(\frac{C}{h_0}\right)^\frac{2\,q}{q-p+2}\,\int_{B_R}\left(\int_{B_R}|x-y|^{\frac{q\,(2-2\,s-\varepsilon\, s)}{q-p+2}-N}\, dy\right)\frac{|\delta_h u(x)|^{(\beta+1)\frac{q}{q-p+2}}}{|h|^{(1+\vartheta\,\beta)\frac{q}{q-p+2}}} dx\\
&\leq C\,[u]_{W^{\frac{s\,(p-2-\varepsilon)}{p-2},q}(B_{R+h_0})}^q+\left(\frac{C}{h_0}\right)^\frac{2\,q}{q-p+2}\,\frac{q-p+2}{q\,(2-2\,s-\varepsilon\,s)}\,R^{\frac{q\,(2-2\,s-\varepsilon\,s)}{q-p+2}}\,\int_{B_R}\frac{|\delta_h u_h(x)|^{\frac{(\beta+1)\, q}{q-p+2}}}{|h|^{(1+\vartheta\,\beta)\frac{q}{q-p+2}}} dx\\
&\le C\,[u]_{W^{\frac{s\,(p-2-\varepsilon)}{p-2},q}(B_{R+h_0})}^q+C\,\|u\|_{L^\infty(B_{R+h_0})}^\frac{q}{q-p+2}\,\int_{B_R}\frac{|\delta_h u_h(x)|^{\frac{\beta\, q}{q-p+2}}}{|h|^{(1+\vartheta\,\beta)\frac{q}{q-p+2}}} dx\
\end{split}
\]
where $C=C(N,h_0,p,s,q)>0$.  We also used that $q\geq p$. Observe that the choice \eqref{epsilon} of $\varepsilon$ 
assures 
that
\[
\frac{q\,(2-2s-\varepsilon\, s)}{q-p+2}>0.
\]
In order to estimate the term 
\[
[u]_{W^{\frac{s\,(p-2-\varepsilon)}{p-2},q}(B_{R+h_0})}^q,
\] 
we use the Sobolev embedding, namely Lemma \ref{lm:erik2} with
\[
\alpha=s\,\frac{(p-2-\varepsilon)}{p-2},\quad \beta=s,\quad r=R+4\,h_0,\qquad \varrho=R+h_0.
\]
By further using that $u$ is locally bounded, more precisely by \eqref{bounds}, and since $R + 4h_0 \leq 1$, we get
$$
[u]_{W^{\frac{s\,(p-2-\varepsilon)}{p-2},q}(B_{R+h_0})}^q\leq C\,\left(\sup_{0<|h|< h_0}\left\|\frac{\delta^2_h u}{|h|^s}\right\|_{L^q(B_{R+4\,h_0})}^q+1\right).
$$
Therefore, for $p\ge 2$ we get
\[
|\mathcal{I}_{11}|\leq C\,\left(\int_{B_R}\frac{|\delta_h u(x)|^{\frac{\beta\, q}{q-p+2}}}{|h|^{(1+\vartheta\,\beta)\frac{q}{q-p+2}}}\, dx +\sup_{0<|h|< h_0}\left\|\frac{\delta^2_h u}{|h|^s}\right\|_{L^q(B_{R+4h_0})}^q+1\right),
\]
where $C=C(N,h_0,p,s,q)>0$.
\par
As for $\mathcal{I}_{12}$, we have
\begin{align}
\iint_{B_R\times B_R} \frac{|\delta_h u(x)|^{\beta-1+p}}{|h|^{1+\vartheta\,\beta}}\, \frac{|\eta(x)-\eta(y)|^p}{|x-y|^{N+s\,p}}\,dx\,dy&\le C\int_{B_R} \frac{|\delta_h u(x)|^{\beta-1+p}}{|h|^{1+\vartheta\,\beta}}\,dx\label{I2estusedagain}\\
&\le C\left( \int_{B_R} \frac{|\delta_h u(x)|^{\beta\,\frac{q}{q-p+2}}}{|h|^{(1+\vartheta\,\beta)\,\frac{q}{q-p+2}}}\,dx+\int_{B_R} |\delta_h u(x)|^\frac{q\,(p-1)}{p-2}\,dx\right)\nonumber\\
&\le C\left( \int_{B_R} \frac{|\delta_h u|^{\beta\,\frac{q}{q-p+2}}}{|h|^{(1+\vartheta\,\beta)\,\frac{q}{q-p+2}}}\,dx+1\right)\nonumber,
\end{align}
where we used the Lipschitz character of $\eta$, treated the integral in $y$ as we did above, used the fact that $\beta-1+p\geq \beta$, Young's inequality with exponents\footnote{Only in the case $p>2$, for $p=2$ it is not needed, we simply estimate
\[
|\delta_h u(x)|^{\beta+1}\le 2\,\|u\|_{L^\infty(B_{R+h_0})}\,|\delta_h u(x)|^{\beta}.
\]} $q/(p-2)$ and $q/(q-p+2)$ and that $u$ is bounded. Here $C=C(N,h_0,p,s,q)>0$. Thus, we obtain
\[
|\mathcal{I}_{12}|\leq C\left( \int_{B_R} \frac{|\delta_h u|^{\beta\,\frac{q}{q-p+2}}}{|h|^{(1+\vartheta\,\beta)\,\frac{q}{q-p+2}}}\,dx+1\right),\qquad\mbox{ with } C=C(N,h_0,p,s,q)>0.
\]
If we now use these estimates in \eqref{Ieq}, we get
\begin{equation}
\label{Ieq2}
\begin{split}
\left[\frac{|\delta_h u|^\frac{\beta-1}{p}\,\delta_h u}{|h|^\frac{1+\vartheta\,\beta}{p}}\,\eta\right]^p_{W^{s,p}(B_R)}&\leq C\,\left(\int_{B_R}\left|\frac{\delta_h u}{|h|^\frac{1+\vartheta\,\beta}{\beta}}\right|^\frac{\beta\,q}{q-p+2}\, dx +\sup_{0<|h|< h_0}\left\|\frac{\delta^2_h u}{|h|^s}\right\|_{L^q(B_{R+4h_0})}^q+1\right)\\
&+C\,\Big(|\mathcal{I}_2|+|\mathcal{I}_3|\Big),\qquad \mbox{ with }C=C(h_0,N,p,s,q)>0.
\end{split}
\end{equation}
{\noindent}{\bf Step 3: Estimates of the nonlocal terms $\mathcal{I}_2$ and $\mathcal{I}_3$.}
Both nonlocal terms $\mathcal{I}_2$ and $\mathcal{I}_3$ can be treated in the same way. We only estimate $\mathcal{I}_2$ for simplicity. Since $|u|\le 1$ in $B_1$, we have 
\[
\begin{split}
\Big|(J_p(u_h(x)-u_h(y))-J_p(u(x)-u(y))\, J_{\beta+1}(\delta_h u(x))\Big|&\leq C\left(1+|u_h(y)|^{p-1}+|u(y)|^{p-1}\right)\,|\delta_h u(x)|^{\beta},
\end{split}
\]
where $C=C(p)>0$. For $x\in B_{(R+r)/2}$ we have $B_{(R-r)/2}(x)\subset B_{R}$ and thus
$$
\int_{\mathbb{R}^N\setminus B_{R}}\frac{1}{|x-y|^{N+s\,p}}\,  dy\leq \int_{\mathbb{R}^N\setminus B_\frac{R-r}{2}(x)} \frac{1}{|x-y|^{N+s\,p}}\, dy\leq C(N,h_0,p,s),
$$
by recalling that $R-r=4\,h_0$.
By using Lemma \ref{lm:allarga} we get for $x \in  B_{(R+r)/2}$ and then Lemma \ref{lm:tails}
\[
\begin{split}
\int_{\mathbb{R}^N\setminus B_{R}} \frac{|u(y)|^{p-1}}{|x-y|^{N+s\,p}}\, dy&\le \left(\frac{2\,R}{R-r}\right)^{N+s\,p}\,\int_{\mathbb{R}^N\setminus B_R} \frac{|u(y)|^{p-1}}{|y|^{N+s\,p}} \, dy
\\
&\le \left(\frac{2\,R}{R-r}\right)^{N+s\,p}\,\int_{\mathbb{R}^N\setminus B_1} \frac{|u(y)|^{p-1}}{|y|^{N+s\,p}} \, dy+\left(\frac{2\,R}{R-r}\right)^{N+s\,p}\,R^{-N}\,\int_{B_1} |u|^{p-1}\,dy\\
&\leq C(N,h_0,p,s).
\end{split}
\]
In the last estimate we have used the bounds assumed on $u$ and $4\,h_0< R \leq 1$. The term involving $u_h$ can be estimated similarly.
Recall also that $\eta=0$ outside $B_{(R+r)/2}$. Hence, we have
\begin{align}
|\mathcal{I}_2|+|\mathcal{I}_3|&\leq C\,\int_{B_{\frac{R+r}{2}}}\frac{|\delta_h u|^{\beta}}{|h|^{1+\vartheta\beta}} dx\leq C\,\left(1+\int_{B_{R}}\left|\frac{\delta_h u}{|h|^{\frac{1+\vartheta\beta}{\beta}}}\right|^{\frac{\beta\, q}{q-p+2}}\,dx\right) \label{I3I4est}, 
\end{align}
by Young's inequality. Here $C=C(N,h_0,s,p)>0$ as always. By inserting this estimate in \eqref{Ieq2}, we obtain
\begin{equation}
\label{I2I3}
\begin{split}
\left[\frac{|\delta_h u|^\frac{\beta-1}{p}\,\delta_h u}{|h|^\frac{1+\vartheta\,\beta}{p}}\,\eta\right]^p_{W^{s,p}(B_R)}&\leq C\,\left(\int_{B_{R}}\left|\frac{\delta_h u}{|h|^{\frac{1+\vartheta\beta}{\beta}}}\right|^{\frac{\beta\, q}{q-p+2}}\,dx+\sup_{0<|h|< h_0}\left\|\frac{\delta^2_h u}{|h|^s}\right\|_{L^q(B_{R+4h_0})}^q+1\right)
\end{split}
\end{equation}
{\bf Step 4: Going back to the equation.}
For $\xi,h\in\mathbb{R}^N\setminus\{0\}$ such that $|h|,|\xi|<h_0$, we use inequality \eqref{holder} with the choices
\[
A=u(x+h+\xi)-u(x+\xi),\qquad B=u(x+h)-u(x),\qquad \gamma=\frac{\beta-1+p}{p}
\]
to arrive at 
\begin{equation}
\label{est3}
\left\|\frac{\delta_\xi \delta_h u}{|\xi|^\frac{s\,p}{\beta-1+p}\,|h|^\frac{1+\vartheta\,\beta}{\beta-1+p}}\right\|^{\beta-1+p}_{L^{\beta-1+p}(B_r)}
\leq C \left\|\frac{\delta_\xi \left(|\delta_h u|^\frac{\beta-1}{p}\,\delta_h u\right)}{|\xi|^s\,|h|^\frac{1+\vartheta\,\beta}{p}}\right\|^p_{L^p(B_r)}
\leq C\,\left\|\eta\,\frac{\delta_\xi}{|\xi|^s}\left(\frac{|\delta_h u|^\frac{\beta-1}{p}\,(\delta_h u)}{|h|^\frac{1+\vartheta\,\beta}{p}}\right)\right\|^p_{L^p(\mathbb{R}^N)},
\end{equation}
where $C=C(\beta)>0$. Here we used that $\eta \equiv 1$ on $B_r$.
Now by the discrete Leibniz rule \eqref{leibniz}, we write
$$
\eta\,\delta_\xi\left(|\delta_h u|^\frac{\beta-1}{p}\,(\delta_h u)\right) = \delta_\xi\left(\eta\,|\delta_h u|^\frac{\beta-1}{p}\,(\delta_h u)\right)-\left(|\delta_h u|^\frac{\beta-1}{p}\,(\delta_h u)\right)_\xi\,\delta_\xi\eta,
$$
and thus find
\begin{equation}\label{extraest}
\left\|\frac{\delta_\xi \delta_h u}{|\xi|^\frac{s\,p}{\beta-1+p}\,|h|^\frac{1+\vartheta\,\beta}{\beta-1+p}}\right\|^{\beta-1+p}_{L^{\beta-1+p}(B_r)}
\leq C\, \left\|\frac{\delta_\xi}{|\xi|^s}\left(\frac{|\delta_h u|^\frac{\beta-1}{p}\,(\delta_h u)\,\eta}{|h|^\frac{1+\vartheta\,\beta}{p}}  \right)\right\|^p_{L^p(\mathbb{R}^N)} 
+ C\, \left\|\frac{\delta_\xi\eta}{|\xi|^s}\frac{\left(|\delta_h u|^\frac{\beta-1}{p}\,(\delta_h u)\right)_\xi}{|h|^\frac{1+\vartheta\,\beta}{p}}\right\|^p_{L^p(\mathbb{R}^N)},
\end{equation}
where $C=C(p,\beta)>0$. For the first term in \eqref{extraest}, we apply \cite[Proposition 2.6]{Brolin} with the choice
\[
\psi=\frac{|\delta_h u|^\frac{\beta-1}{p}\,(\delta_h u)\,\eta}{|h|^\frac{1+\vartheta\,\beta}{p}},
\]
and get
\begin{equation}
\label{est2}
\sup_{|\xi|>0}\left\|\frac{\delta_\xi}{|\xi|^s}\frac{|\delta_h u|^\frac{\beta-1}{p}\,(\delta_h u)\,\eta}{|h|^\frac{1+\vartheta\,\beta}{p}}\right\|^p_{L^p(\mathbb{R}^N)}\leq C\,(1-s)\left[\frac{|\delta_h u|^\frac{\beta-1}{p}\,(\delta_h u)\,\eta}{|h|^\frac{1+\vartheta\,\beta}{p}}\right]^p_{W^{s,p}(B_R)}, 
\end{equation}
where $C=C(N,h_0,p)>0$. Here we also used that $\frac{R+r}{2} + 2h_0 = R$.
\par
As for the second term in \eqref{extraest}, we observe that for every $0<|\xi|<h_0$
\begin{equation}
\label{arrivotre}
\begin{split}
\left\|\frac{\delta_\xi\eta}{|\xi|^s}\frac{\left(|\delta_h u|^\frac{\beta-1}{p}\,(\delta_h u)\right)_\xi}{|h|^\frac{1+\vartheta\,\beta}{p}}\right\|^p_{L^p(\mathbb{R}^N)}&\leq
 C\,\left\|\frac{\left(|\delta_h u|^\frac{\beta-1}{p}\,(\delta_h u)\right)_\xi}{|h|^\frac{1+\vartheta\,\beta}{p}}\right\|^p_{L^p(B_{\frac{R+r}{2}+h_0})}\\
&\leq C\,\int_{B_{\frac{R+r}{2}+2h_0}}\frac{|\delta_h u|^{\beta-1+p}}{|h|^{1+\vartheta\,\beta}} dx\\
&\leq C\,\int_{B_{R}}\frac{|\delta_h u|^{\beta}}{|h|^{1+\vartheta\,\beta}} dx\\
&\leq C\,\left(\int_{B_R}\left|\frac{\delta_h u}{|h|^\frac{1+\vartheta\beta}{\beta}}\right|^\frac{q\,\beta}{q-p+2} dx+1\right), 
\end{split}
\end{equation}
where $C=C(N,h_0,p,s)>0$. Here we have used the estimates of $\nabla \eta$, the fact that $u$ is bounded and Young's inequality with exponents $q/(p-2)$ and $q/(q-p+2)$.
\par
From \eqref{extraest}, \eqref{est2}, and \eqref{arrivotre} we get for any $0 < |\xi| < h_0$
\begin{equation}
\label{est4}
\left\|\frac{\delta_\xi \delta_h u}{|\xi|^\frac{s\,p}{\beta-1+p}|h|^\frac{1+\vartheta\,\beta}{\beta-1+p}}\right\|^{\beta-1+p}_{L^{\beta-1+p}(B_r)}\leq C\,\left[\frac{|\delta_h u|^\frac{\beta-1}{p}\,(\delta_h u)}{|h|^\frac{1+\vartheta\,\beta}{p}}\eta\right]^p_{W^{s,p}(B(R))}+C\,\left(\int_{B_{R}}\left|\frac{\delta_h u}{|h|^\frac{1+\vartheta\beta}{\beta}}\right|^\frac{q\,\beta}{q-p+2} dx+1\right),
\end{equation}
with $C=C(N,h_0,s,p,\beta)>0$.
We then choose $\xi=h$ and take the supremum over $h$ for $0<|h|< h_0$. Then \eqref{est4} together with \eqref{I2I3} imply
\begin{equation}\label{almostfinal}
\sup_{0<|h|< h_0}\int_{B_r}\left|\frac{\delta^2_h u}{|h|^\frac{1+s\,p+\vartheta\,\beta}{\beta-1+p}}\right|^{\beta-1+p}\,dx\leq C\,\left(\sup_{0<|h|< h_0}\int_{B_R}\left|\frac{\delta_h u}{|h|^\frac{1+\vartheta\,\beta}{\beta}}\right|^\frac{q\,\beta}{q-p+2}\,dx +\sup_{0<|h|< h_0}\left\|\frac{\delta^2_h u}{|h|^s}\right\|_{L^q(B_{R+4h_0})}^q+ 1\right),
\end{equation}
where $C=C(N,h_0,p,q,s,\beta)>0$. By observing that $(1+\vartheta\,\beta)/\beta=s< 1$, we can use the second estimate of Lemma \ref{lm:erik2} and replace the first order difference quotients in the right-hand side of \eqref{almostfinal} with second order ones. This gives
\begin{equation}
\label{finalbeforeexp}
\sup_{0<|h|< h_0}\int_{B_r}\left|\frac{\delta^2_h u}{|h|^\frac{1+s\,p+\vartheta\,\beta}{\beta-1+p}}\right|^{\beta-1+p}\,dx\leq C\,\left(\sup_{0<|h|< h_0}\left\|\frac{\delta^2_h u }{|h|^\frac{1+\vartheta\,\beta}{\beta}}\right\|_{L^\frac{q\,\beta}{q-p+2}(B_{R+4h_0})}^\frac{q\,\beta}{q-p+2}+\sup_{0<|h|< h_0}\left\|\frac{\delta^2_h u}{|h|^s}\right\|_{L^q(B_{R+4\,h_0})}^q+ 1\right),
\end{equation}
for some constant $C=C(N,h_0,p,q,s,\beta)>0$.
\vskip.2cm\noindent
{\bf Step 5: Conclusion.}
We now specify our choices for $\beta$ and $\vartheta$. We set 
\begin{equation}
\label{betatheta}
\beta=q-p+2\qquad \mbox{ and }\qquad \vartheta=\frac{(q-p+2)\,s-1}{q-p+2}.
\end{equation}
In particular we obtain
\[
\frac{1+s\,p+\vartheta\,\beta}{\beta-1+p}=\frac{s}{q+1}+s,
\]
\[
\beta-1+p=q+1,\qquad \frac{q\,\beta}{q-p+2}=q,
\]
and
\[
\frac{1+\vartheta\,\beta}{\beta}=s.
\]
Then \eqref{finalbeforeexp} becomes
$$
\sup_{0<|h|< h_0}\left\|\frac{\delta^2_h u}{|h|^{\frac{s}{q+1}+s}}\right\|_{L^{q+1}(B_r)}^{q+1}\leq C\,\left(\sup_{0<|h|< h_0}\left\|\frac{\delta^2_h u }{|h|^s}\right\|_{L^q(B_{R+4h_0})}^q+1\right), 
$$
where $C=C(N,h_0,p,q,s)>0$. In particular, up to modifying the constant $C$, we obtain
$$
\sup_{0<|h|< h_0}\left\|\frac{\delta^2_h u}{|h|^{s}}\right\|_{L^{q+1}(B_{R-4\,h_0})}^{q+1}\leq C\,\left(\sup_{0<|h|< h_0}\left\|\frac{\delta^2_h u }{|h|^s}\right\|_{L^q(B_{R+4h_0})}^q+1\right),
$$
as desired, where we recall that $r=R-4\,h_0$.
\end{proof}
By iterating the previous result, we can obtain the following regularity result for $(s,p)-$harmonic functions. This is the main outcome of this section.
\begin{teo}[Almost $C^s$ regularity for $(s,p)-$harmonic functions.]
\label{teo:1}
Let $\Omega\subset\mathbb{R}^N$ be a bounded and open set, $p\ge 2$ and $0<s<1$. Suppose $u\in W^{s,p}_{\rm loc}(\Omega)\cap L^{p-1}_{s\,p}(\mathbb{R}^N)$ is a local weak solution of 
\[
(-\Delta_p)^s u=0\qquad \mbox{ in }\Omega.
\] 
Then $u\in C^\delta_{\rm loc}(\Omega)$ for every $0<\delta<s$. 
\par
More precisely, for every $0<\delta<s$ and every ball $B_{2R}(x_0)\Subset\Omega$, there exists a constant $C=C(N,s,p,\delta)>0$ such that we have the scaling invariant estimate
\begin{equation}
\label{apriori}
[u]_{C^\delta(B_{R/2}(x_0))}\leq \frac{C}{R^\delta}\,\left(\|u\|_{L^\infty(B_{R}(x_0))}+R^{s-\frac{N}{p}}\,[u]_{W^{s,p}(B_{R}(x_0))}+\mathrm{Tail}_{p-1,s\,p}(u;x_0,R)\right).
\end{equation}
\end{teo}
\begin{proof}
We first observe that $u\in L^\infty_{\rm loc}(\Omega)$, by Theorem~\ref{teo:loc_bound} and Remark \ref{oss:stupid}.
  We assume for simplicity that $x_0=0$, then we set 
\[
\mathcal{M}_R=\|u\|_{L^\infty(B_{R})}+R^{s-\frac{N}{p}}\,[u]_{W^{s,p}(B_{R})}+\mathrm{Tail}_{p-1,s\,p}(u;0,R)>0.
\]
We point out that it is sufficient to prove that the rescaled function
\[
u_R(x):=\frac{1}{\mathcal{M}_R}\,u(R\,x),\qquad \mbox{ for }x\in B_2,
\]
satifies the estimate
\[
[u_R]_{C^{\delta}(B_{1/2})}\leq C.
\]
By scaling back, we would get the desired estimate. Observe that by definition, the function $u_R$ 
is a local weak solution of $(-\Delta_p)^s u=0$ in $B_2$ and satisfies
\begin{equation}
\label{assumption}
\|u_R\|_{L^\infty(B_1)}\leq 1,\qquad \int_{\mathbb{R}^N\setminus B_1}\frac{|u_R(y)|^{p-1}}{|y|^{N+s\,p}}\,  dy\leq 1,\qquad [u_R]_{W^{s,p}(B_1)}\leq 1.
\end{equation}
In what follows, we forget the subscript $R$ and simply write $u$ in place of $u_R$, in order not to overburden the presentation.
\vskip.2cm\noindent
We fix $0<\delta<s$ and choose $i_\infty\in\mathbb{N}\setminus\{0\}$ such that
\[
s-\delta> \frac{N}{p+i_\infty}.
\]
Then we define the sequence of exponents  
\[
q_i=p+i,\qquad i=0,\dots,i_\infty.
\]
We define also 
$$
h_0=\frac{1}{64\,i_\infty},\qquad R_i=\frac{7}{8}-4\,(2i+1)\,h_0=\frac{7}{8}-\frac{2i+1}{16\,i_\infty},\qquad \mbox{ for } i=0,\dots,i_\infty.
$$
We note that 
\[
R_0+4\,h_0=\frac{7}{8}\qquad \mbox{ and }\qquad R_{i_\infty-1}-4\,h_0=\frac{3}{4}.
\] 
By applying Proposition \ref{prop:improve} with\footnote{We observe that by construction we have
\[
4\,h_0<R_i\le 1-5\,h_0,\qquad \mbox{ for } i=0,\dots,i_\infty-1.
\]
Thus these choices are admissible in Proposition \ref{prop:improve}.} 
\[
R=R_i\qquad \mbox{ and }\qquad q=q_i=p+i,\qquad \mbox{ for } i=0,\ldots,i_\infty-1,
\] 
and observing that $R_i-4\,h_0=R_{i+1}+4\,h_0$,
we obtain the iterative scheme of inequalities
\[
\left\{\begin{array}{rcll}
\sup\limits_{0<|h|< h_0}\left\|\dfrac{\delta^2_h u}{|h|^{s}}\right\|_{L^{q_1}(B_{R_1+4h_0})}&\leq& C\,\sup\limits_{0<|h|< h_0}\left(\left\|\dfrac{\delta^2_h u }{|h|^s}\right\|_{L^p(B_{7/8})}+1\right)\\
&&&\\
\sup\limits_{|h|\leq h_0}\left\|\dfrac{\delta^2_h u}{|h|^{s}}\right\|_{L^{q_i+1}(B_{R_{i+1}+4h_0})}&\leq& C\,\sup\limits_{0<|h|< h_0}\left(\left\|\dfrac{\delta^2_h u }{|h|^s}\right\|_{L^{q_i}(B_{R_i+4h_0})}+1\right),&\mbox{ for } i=1,\ldots,i_\infty-2,
\end{array}
\right.
\]
and finally
$$
\sup_{0<|h|< {h_0}}\left\|\frac{\delta^2_h u}{|h|^{s}}\right\|_{L^{q_{i_\infty}}(B_{3/4})}=\sup_{0<|h|< h_0}\left\|\frac{\delta^2_h u}{|h|^{s}}\right\|_{L^{p+i_\infty}(B_{R_{i_\infty-1}-4h_0})}\leq C\sup_{0<|h|< h_0}\left(\left\|\frac{\delta^2_h u }{|h|^s}\right\|_{L^{p+i_\infty-1}(B_{R_{i_\infty-1}+4h_0})}+1\right).
$$
Here $C=C(N,\delta,p,s)>0$ as always. We note that by \cite[Proposition 2.6]{Brolin} together with the relation 
$$
\delta_h u =\frac12\left(\delta_{2h} u-\delta^2_{h}u\right),
$$ 
we have
\begin{align}\label{eq:1sttofrac}
\sup_{0<|h|< h_0}\left\|\frac{\delta^2_h u }{|h|^s}\right\|_{L^{p}(B_{7/8})}&\leq C\,\sup_{0<|h|<2\,h_0}\left\|\frac{\delta_h u }{|h|^s}\right\|_{L^{p}(B_{7/8})}\nonumber \\
&\leq C\left([u]_{W^{s,p}(B_{7/8+2\,h_0})}+\|u\|_{L^\infty(B_{7/8+2\,h_0})}\right)\\
&\leq C\left([u]_{W^{s,p}(B_1)}+\|u\|_{L^\infty(B_1)}\right)\leq C(N,\delta,s,p),\nonumber 
\end{align}
where we also have used the assumptions \eqref{assumption} on $u$. Hence, the iterative scheme of inequalities leads us to
\begin{equation}
\sup_{0<|h|< {h_0}}\left\|\frac{\delta^2_h u}{|h|^{s}}\right\|_{L^{q_{i_\infty}}(B_{3/4})}\leq C(N,\delta,p,s).
\label{otherest}
\end{equation}
Take now $\chi\in C_0^\infty(B_{5/8})$ such that 
$$
0\leq \chi\leq 1, \qquad \chi =1 \text{ in $B_{1/2}$},\qquad |\nabla \chi|\leq C,\qquad |D^2 \chi|\leq C.
$$
In particular we have for all $|h| > 0$
$$
\frac{|\delta_h\chi|}{|h|^s}\leq C,\qquad \frac{|\delta^2_h\chi|}{|h|^s}\leq C.
$$
We also recall that
$$
\delta^2_h (u\,\chi)=\chi_{2h}\,\delta^2_h u+2\,\delta_h u\, \delta_h \chi_h+u\,\delta^2_h\chi.
$$
Hence, for $0<|h|< h_0$
\begin{align}\label{Neq}
[u\,\chi]_{\mathcal{B}^{s,q_{i_\infty}}_\infty(\mathbb{R}^N)}=\left\|\frac{\delta^2_h (u\,\chi)}{|h|^s}\right\|_{L^{q_{i_\infty}}(\mathbb{R}^N)}&\leq C\,\left(\left\|\frac{\chi_{2h}\,\delta^2_h u}{|h|^s}\right\|_{L^{q_{i_\infty}}(\mathbb{R}^N)}+\left\|\frac{\delta_h u\,\delta_h\chi}{|h|^s}\right\|_{L^{q_{i_\infty}}(\mathbb{R}^N)}+\left\|\frac{u\,\delta^2_h\chi}{|h|^s}\right\|_{L^{q_{i_\infty}}(\mathbb{R}^N)}\right) \nonumber\\
&\leq C\,\left(\left\|\frac{\delta^2_h u}{|h|^s}\right\|_{L^{q_{i_\infty}}(B_{5/8+2\,h_0})}+\|\delta_h u\|_{L^{q_{i_\infty}}(B_{5/8+h_0})}+\|u\|_{L^{q_{i_\infty}}(B_{5/8+2h_0})}\right) \\
&\leq C\,\left(\left\|\frac{\delta^2_h u}{|h|^s}\right\|_{L^{q_{i_\infty}}(B_{3/4})}+\|u\|_{L^{q_{i_\infty}}(B_{3/4})}\right)\leq C(N,\delta,p,s), \nonumber
\end{align}
by \eqref{otherest}. By Lemma \ref{lm:fromBtoN}
\begin{equation}\label{Neq2}
[u\,\chi]_{\mathcal{N}_\infty^{s,q_{i_\infty}}(\mathbb{R}^N)}\leq C(N,\delta,s)\,[u\,\chi]_{\mathcal{B}_\infty^{s,q_{i_\infty}}(\mathbb{R}^N)}\leq C(N,\delta,p,s).
\end{equation}
Finally, by noting that thanks to the choice of $i_\infty$ we have
\[
s\,q_{i_\infty}>N\qquad \mbox{and }\qquad \delta<s-\frac{N}{q_{i_\infty}},
\] 
we may apply Theorem~\ref{thm:Nholder} with  $\beta=s$, $\alpha=\delta$ and $q=q_{i_\infty}$ to obtain
$$
[u]_{C^\delta(B_{1/2})}= [u\,\chi]_{C^\delta(B_{1/2})}\leq C\left([u\,\chi]_{\mathcal{N}_\infty^{s,q_{i_\infty}}(\mathbb{R}^N)}\right)^{\frac{\delta\,q_{i_\infty}+N}{s\,q_{i_\infty}}}\,\left(\|u\,\chi\|_{L^q(\mathbb{R}^N)}\right)^\frac{(s-\delta)\,q_{i_\infty}-N}{s\,q_{i_\infty}}\leq C(N,\delta,p,s).
$$
This concludes the proof.
\end{proof}
\begin{oss}
\label{oss:flexibility}
Under the assumptions of the previous theorem, a covering argument combined with \eqref{apriori} implies the more flexible estimate: for every $0<\sigma<1$
\[
[u]_{C^\delta(B_{\sigma R}(x_0))}\leq \frac{C}{R^\delta}\,\left(\|u\|_{L^\infty(B_{R}(x_0))}+R^{s-\frac{N}{p}}\,[u]_{W^{s,p}(B_{R}(x_0))}+\mathrm{Tail}_{p-1,s\,p}(u;x_0,R)\right),
\]
with $C$ now depending on $\sigma$ as well (and blowing-up as $\sigma\nearrow 1$). Indeed, if $\sigma\le 1/2$ then this is immediate. If $1/2<\sigma<1$, then we can cover $\overline{B_{\sigma\,R}(x_0)}$ with a finite number of balls $B_{r/2}(x_1),\dots,B_{r/2}(x_k)$, where
\[
x_i\in B_{R/2}(x_0)\qquad \mbox{ and }\qquad r=\frac{R}{2}.
\]
By using \eqref{apriori} on every ball $B_{2\,r}(x_i)\Subset B_{2R}(x_0)\Subset\Omega$, we get
\[
[u]_{C^\delta(B_{r/2}(x_i))}\leq \frac{C}{r^\delta}\,\left(\|u\|_{L^\infty(B_{r}(x_i))}+r^{s-\frac{N}{p}}\,[u]_{W^{s,p}(B_{r}(x_i))}+\mathrm{Tail}_{p-1,s\,p}(u;x_i,r)\right).
\]
By observing that $B_{r}(x_i)\subset B_R(x_0)$, summing over $i$ and using Lemma \ref{lm:tails} for the tail term, we get the desired conclusion.
\end{oss}

\section{Improved H\"older regularity: the homogeneous case}
\label{sec:higher}
Now that we know that a solution to the homogeneous equation is locally H\"older $\delta-$continuous for any $0<\delta<s$, we can obtain the following improvement of Proposition \ref{prop:improve}, which provided an increased integrability estimate. In contrast, the following proposition also increases the differentiability.
\begin{prop}\label{prop:improve2}
Assume $p\ge 2$ and $0<s<1$. Let  $u\in W^{s,p}_{\rm loc}(B_2)\cap L^{p-1}_{s\,p}(\mathbb{R}^N)$ be a local weak solution of $(-\Delta_p)^s u=0$ in $B_2$. Suppose that 
\[
\|u\|_{L^\infty(B_1)}\leq 1\qquad \mbox{ and }\qquad  \int_{\mathbb{R}^N\setminus B_1} \frac{|u(y)|^{p-1}}{|y|^{N+s\,p}}\,  dy\leq 1.
\]
Assume further that for some $0<h_0<1$, $0<\vartheta<1$ and $\beta>1$ such that $(1+\vartheta\, \beta)/\beta<1$, we have
\[
\sup_{0<|h|\leq h_0}\left\|\frac{\delta^2_h u }{|h|^\frac{1+\vartheta\, \beta}{\beta}}\right\|_{L^\beta(B_{1})}^\beta<+\infty.
\]
Then for every radius $4\,h_0<R\le 1-5\,h_0$
we have
$$
\sup_{0<|h|< h_0}\left\|\frac{\delta^2_h u}{|h|^{\frac{1+s\,p+\vartheta\,\beta}{\beta-1+p}}}\right\|_{L^{\beta-1+p}(B_{R-4\,h_0})}^{\beta-1+p}\leq C\,\sup_{0<|h|< h_0}\left(\left\|\frac{\delta^2_h u }{|h|^\frac{1+\vartheta\, \beta}{\beta}}\right\|_{L^\beta(B_{R+4\,h_0})}^\beta+1\right).
$$
Here $C$ depends on the $N$, $h_0$, $s$, $p$ and $\beta$.

\end{prop}
\begin{proof} We will go back to the estimates in the proof of Proposition \ref{prop:improve}. The only difference is that we estimate the term $\mathcal{I}_{11}$ defined in \eqref{I11} differently. From Theorem~\ref{teo:1} and Remark \ref{oss:flexibility}, we can choose 
\[
0<\varepsilon<\min\left\{s,\,2\,\frac{1-s}{p-2}\right\},
\] 
such that 
$$
[u]_{C^{s-\varepsilon}(B_R)}\leq C(N,h_0,p,s).
$$
Using this together with the assumed regularity of $\eta$, we have
$$
\frac{|u(x)-u(y)|^{p-2}\,\left|\eta(x)^\frac{p}{2}-\eta(y)^\frac{p}{2}\right|^2}{|x-y|^{N+s\,p}}\leq C\,|x-y|^{-N+2\,(1-s)-\varepsilon\,(p-2)}.
$$
Thanks to the choice of $\varepsilon$, the last exponent is strictly less than $N$ and we may conclude
$$
\int_{B_R}\frac{|u(x)-u(y)|^{p-2}\,\left|\eta(x)^\frac{p}{2}-\eta(y)^\frac{p}{2}\right|^2}{|x-y|^{N+s\,p}} dy\leq C(N,h_0,p,s),
$$
for any $x\in B_R$. Therefore we have the estimate
\[
\begin{split}
|\mathcal{I}_{11}|&\leq C\,\int_{B_R}\frac{|\delta_h u(x)|^{\beta+1}}{|h|^{1+\vartheta\,\beta}} dx\\
&\leq C\,\|u\|_{L^\infty(B_{R+h_0})}\,\int_{B_R}\frac{|\delta_h u(x)|^{\beta}}{|h|^{1+\vartheta\,\beta}} dx\le C\,\int_{B_R}\frac{|\delta_h u(x)|^{\beta}}{|h|^{1+\vartheta\,\beta}} dx,\qquad \mbox{�for some } C=C(N,h_0,p,s)>0.
\end{split}
\]
Going back to \eqref{I2estusedagain} we can also extract
$$
|\mathcal{I}_{12}|\leq C\int_{B_R} \frac{|\delta_h u(x)|^{\beta-1+p}}{|h|^{1+\vartheta\,\beta}}\,dx\leq C\int_{B_R}\frac{|\delta_h u(x)|^{\beta}}{|h|^{1+\vartheta\,\beta}} dx,\qquad \mbox{ for some }C=C(N,h_0,p,s)>0,
$$
where we used the local $L^\infty$ bound on $u$.
In addition, from the first inequality in \eqref{I3I4est} we have
$$
|\mathcal{I}_2|+|\mathcal{I}_3|\leq C\int_{B_R}\frac{|\delta_h u(x)|^{\beta}}{|h|^{1+\vartheta\,\beta}} dx, \qquad C=C(N,h_0,p,s)>0.
$$
Combining these new estimates with \eqref{est2}, \eqref{est3}, \eqref{arrivotre} and \eqref{Ieq}, we arrive at
$$
\sup_{0<|h|< h_0}\int_{B_r}\left|\frac{\delta^2_h u(x)}{|h|^\frac{1+s\,p+\vartheta\,\beta}{\beta-1+p}}\right|^{\beta-1+p}\,dx\leq C\sup_{0<|h|< h_0}\int_{B_R}\left|\frac{\delta_h u(x)}{|h|^\frac{1+\vartheta\,\beta}{\beta}}\right|^\beta\,dx,\quad  C=C(N,h_0,p,s,\beta)>0.
$$
By appealing again to the second estimate in Lemma \ref{lm:erik2} and using that
\[
\frac{1+\vartheta\,\beta}{\beta}<1,
\] 
we may replace the first order differential quotients in the right-hand side by second order ones, so that we arrive at
$$
\sup_{0<|h|< h_0}\int_{B_r}\left|\frac{\delta^2_h u(x)}{|h|^\frac{1+s\,p+\vartheta\,\beta}{\beta-1+p}}\right|^{\beta-1+p}\,dx\leq C\left(\sup_{0<|h|< h_0}\int_{B_R+4h_0}\left|\frac{\delta^2_h u(x)}{|h|^\frac{1+\vartheta\,\beta}{\beta}}\right|^\beta\,dx+1\right),\quad  C=C(N,h_0,p,s,\beta).
$$
Recalling again that $r=R-4\,h_0$ this concludes the proof.
\end{proof}

We are now ready to prove Theorem~\ref{mainthm} for the homogeneous equation, i.e. when $f\equiv 0$.
\begin{teo}
\label{teo:1higher}
Let $\Omega\subset\mathbb{R}^N$ be an open and bounded set, $p\ge 2$ and $0<s<1$.  Suppose $u\in W^{s,p}_{\rm loc}(\Omega)\cap L^{p-1}_{s\,p}(\mathbb{R}^N)$ is a local weak solution of 
\[
(-\Delta_p)^s u=0\qquad \mbox{ in }\Omega.
\] 
We define the exponent
\[
\Theta(s,p):=\min\left\{\frac{s\,p}{p-1},\,1\right\}.
\] 
Then $u\in C^\delta_{\rm loc}(\Omega)$ for every $0<\delta<\Theta(s,p)$. 
\par
More precisely, for every $0<\delta<\Theta(s,p)$ and every ball $B_{2R}(x_0)\Subset\Omega$, there exists a constant $C=C(N,s,p,\delta)>0$ such that we have the scaling invariant estimate
\[
[u]_{C^\delta(B_{R/2}(x_0))}\leq \frac{C}{R^\delta}\,\left(\|u\|_{L^\infty(B_{R}(x_0))}+R^{s-\frac{N}{p}}\,[u]_{W^{s,p}(B_{R}(x_0))}+\mathrm{Tail}_{p-1,s\,p}(u;x_0,R)\right).
\]
\end{teo}
\begin{proof}By the same scaling argument as in the proof of Theorem~\ref{teo:1}, it is enough to prove that
$$
[u]_{C^\delta(B_{1/2})}\leq C(N,p,s,\delta),
$$
under the assumption that $u$ satisfies \eqref{assumption}. For $i\in\mathbb{N}$, we define the sequences of exponents
$$
\beta_0=p,\qquad\qquad \beta_{i+1}=\beta_i+p-1=p+i\,(p-1),
$$
and
$$
\vartheta_0=s-\frac{1}{p},\qquad \vartheta_{i+1}=\frac{\vartheta_i\,\beta_i+s\,p}{\beta_{i+1}}=\vartheta_i\,\frac{p+i\,(p-1)}{p+(i+1)(p-1)}+\frac{s\,p}{p+(i+1)(p-1)}.
$$
By induction, we see that $\{\vartheta_i\}_{i\in\mathbb{N}}$ is explictely given by the increasing sequence
$$
 \vartheta_i =\left(s-\frac{1}{p}\right)\,\frac{p}{p+i\,(p-1)}+\frac{s\,p\,i}{p+i\,(p-1)},
$$
and thus
$$
\lim_{i\to\infty} \vartheta_i = \frac{s\,p}{p-1}.
$$
The proof is now split into two different cases.
\vskip.2cm\noindent
{\bf  Case 1: $s\,p\leq (p-1)$.} 
Fix $0<\delta<s\,p/(p-1)$ and choose $i_\infty\in\mathbb{N}\setminus\{0\}$ such that
\[
\delta<\frac{1}{\beta_{i_\infty}}+\vartheta_{i_\infty}-\frac{N}{\beta_{i_\infty}}.
\]
This is clearly possible since 
\[
\lim_{i\to\infty} \beta_i=\infty,\qquad \lim_{i\to\infty}\vartheta_i=\frac{s\,p}{p-1}\qquad \mbox{ and } \qquad \delta< \frac{s\,p}{p-1}.
\]
Define also as in the proof of Theorem \ref{teo:1}
$$
h_0=\frac{1}{64\,i_\infty},\qquad R_i=\frac{7}{8}-4\,(2\,i+1)\,h_0=\frac{7}{8}-\frac{2\,i+1}{16\,i_\infty},\qquad \mbox{ for } i=0,\dots,i_\infty.
$$
We note that 
\[
R_0+4\,h_0=\frac{7}{8}\qquad \mbox{ and }\qquad R_{i_\infty-1}-4\,h_0=\frac{3}{4}.
\] 
By applying Proposition\footnote{Note that in this case we will always have $1+\vartheta_i\beta_i<\beta_i$, so that the proposition applies.} \ref{prop:improve2} with
\[
R=R_i, \qquad \vartheta=\vartheta_i \qquad \mbox{ and }\qquad \beta=\beta_i,\quad \mbox{ for } i=0,\ldots,i_\infty-1,
\] 
and observing that $R_i-4\,h_0=R_{i+1}+4\,h_0$ and that by construction
$$
\frac{1+s\,p+\vartheta_i\,\beta_i}{\beta_i+(p-1)}=\frac{1+\vartheta_{i+1}\,\beta_{i+1}}{\beta_{i+1}},
$$
we obtain the iterative scheme of estimate
\[
\left\{\begin{array}{rcll}
\sup\limits_{0<|h|< h_0}\left\|\dfrac{\delta^2_h u}{|h|^{\frac{1+\vartheta_1\beta_1}{\beta_1}}}\right\|_{L^{\beta_1}(B_{R_1+4h_0})}&\leq& C\,\sup\limits_{0<|h|< h_0}\left(\left\|\dfrac{\delta^2_h u }{|h|^s}\right\|_{L^{p}(B_{7/8})}+1\right)\\
&&&\\
\sup\limits_{0<|h|< h_0}\left\|\dfrac{\delta^2_h u}{|h|^{\frac{1+\vartheta_{i+1}\beta_{i+1}}{\beta_{i+1}}}}\right\|_{L^{\beta_{i+1}}(B_{R_{i+1}+4h_0})}&\leq& C\,\sup\limits_{0<|h|< h_0}\left(\left\|\dfrac{\delta^2_h u }{|h|^{\frac{1+\vartheta_{i}\beta_{i}}{\beta_{i}}}}\right\|_{L^{\beta_i}(B_{R_i+4\,h_0})}+1\right),&\mbox{ for } i=1,\ldots,i_\infty-2,
\end{array}
\right.
\]
and finally
$$
\sup_{0<|h|< {h_0}}\left\|\frac{\delta^2_h u}{|h|^{\frac{1}{\beta_{i_\infty}}+\vartheta_{i_\infty}}}\right\|_{L^{\beta_{i_\infty}}(B_{3/4})}\leq C\sup_{0<|h|< h_0}\left(\left\|\frac{\delta^2_h u }{|h|^{\frac{1+\vartheta_{i_\infty-1}\beta_{i_\infty-1}}{\beta_{i_\infty-1}}}}\right\|_{L^{\beta_{i_\infty-1}}(B_{R_{i_\infty-1}+4\,h_0})}+1\right).
$$
Here $C=C(N,p,s,\delta)>0$ as always. As in \eqref{eq:1sttofrac} we have
$$
\sup_{0<|h|< h_0}\left\|\frac{\delta^2_h u }{|h|^s}\right\|_{L^{p}(B_{7/8})}\leq C(N,\delta,s,p).\nonumber 
$$
Hence, the previous iterative scheme of inequalities implies
\begin{equation}
\label{differentiability}
\sup_{0<|h|< {h_0}}\left\|\frac{\delta^2_h u}{|h|^{\frac{1}{\beta_{i_\infty}}+\vartheta_{i_\infty}}}\right\|_{L^{\beta_{i_\infty}}(B_{3/4})}\leq C(N,\delta,p,s).
\end{equation}
From here we may repeat the arguments at the end of the proof of Theorem~\ref{teo:1} (see \eqref{Neq} and \eqref{Neq2}) and use Theorem~\ref{thm:Nholder} with $\beta = 1/\beta_{i_\infty}+\vartheta_{i_\infty}$, $q=\beta_{i_\infty}$ and $\alpha =\delta$ to obtain 
$$
[u]_{C^\delta(B_{1/2})}\leq C(N,\delta,p,s).
$$
This concludes the proof in this case.
\vskip.2cm\noindent
{\bf Case 2: $s\,p> (p-1)$.} Fix $0<\delta<1$. Let $i_\infty\in\mathbb{N}\setminus\{0\}$ be such that 
$$
\frac{1+\vartheta_{i_\infty-1}\,\beta_{i_\infty-1}}{\beta_{i_\infty-1}}< 1\qquad \mbox{ and }\qquad \frac{1+\vartheta_{i_\infty}\,\beta_{i_\infty}}{\beta_{i_\infty}}\ge 1.
$$
Observe that such a choice is feasible, since 
\[
\lim_{i\to\infty} \frac{1+\vartheta_i\,\beta_i}{\beta_i}=\frac{s\,p}{p-1}>1.
\] 
Now choose $j_\infty$ so that 
$$
\delta<1-\frac{N}{i_\infty+j_\infty},
$$
and let 
\[
\gamma=1-\varepsilon,\qquad \mbox{ for some } 0<\varepsilon<1 \mbox{ such that } \delta<1-\varepsilon-\frac{N}{i_\infty+j_\infty}.
\] 
Define also 
$$
h_0=\frac{1}{64\,(i_\infty+j_\infty)},\qquad R_i=\frac{7}{8}-4\,(2\,i+1)\,h_0=\frac{7}{8}-\frac{2\,i+1}{16\,(i_\infty+j_\infty)},\qquad \mbox{ for } i=0,\dots,i_\infty+j_\infty.
$$
We note that 
\[
R_0+4\,h_0=\frac{7}{8}\qquad \mbox{ and }\qquad R_{(i_\infty+j_\infty)-1}-4\,h_0=\frac{3}{4}.
\] 
By applying Proposition\footnote{Note that for $i\leq i_\infty$ we have $1+\vartheta_i\,\beta_i<\beta_i$, so that the proposition applies.} \ref{prop:improve2} with
\[
R=R_i, \qquad \vartheta=\vartheta_i \qquad \mbox{ and }\qquad \beta=\beta_i,\quad \mbox{ for } i=0,\ldots,i_\infty-1,
\] 
and observing that $R_i-4\,h_0=R_{i+1}+4\,h_0$ and that
$$
\frac{1+s\,p+\vartheta_i\,\beta_i}{\beta_i+(p-1)}=\frac{1+\vartheta_{i+1}\,\beta_{i+1}}{\beta_{i+1}},
$$
we arrive as in {\bf Case 1} at
$$
\sup_{0<|h|< {h_0}}\left\|\frac{\delta^2_h u}{|h|^{\gamma}}\right\|_{L^{\beta_{i_\infty}}(B_{R_{i_\infty}+4h_0})}\leq \sup_{0<|h|< {h_0}}\left\|\frac{\delta^2_h u}{|h|^{\frac{1}{\beta_{i_\infty}}+\vartheta_{i_\infty}}}\right\|_{L^{\beta_{i_\infty}}(B_{R_{i_\infty}+4h_0})}\leq C(N,\delta,p,s),
$$
where we used that $\gamma<1\le 1/\beta_{i_\infty}+\vartheta_{i_\infty}$. We now apply Proposition \ref{prop:improve2} with
\[
R=R_i, \qquad \beta=\beta_i\qquad \mbox{ and }\qquad \vartheta=\widetilde \vartheta_i=\gamma-\frac{1}{\beta_i} \qquad  \mbox{ for } i=i_\infty,\ldots,i_\infty+j_\infty-1.
\] 
Observe that by construction we have
$$
\frac{1+\widetilde \vartheta_i\,\beta_i}{\beta_i}=\gamma,\qquad \mbox{ for } i=i_\infty,\ldots,i_\infty+j_\infty-1,
$$
and using that $s\,p>(p-1)$
$$
\frac{1+s\,p+\widetilde \vartheta_i\, \beta_i}{\beta_i+p-1}>\frac{p+\widetilde \vartheta_i\, \beta_i}{\beta_i+p-1}=1+\frac{\beta_i\,(\gamma-1)}{\beta_i+p-1}>\gamma,\quad \mbox{ for } i=i_\infty,\ldots,i_\infty+j_\infty-1.
$$
We obtain the following chain of estimate
$$
\sup\limits_{|h|\leq h_0}\left\|\dfrac{\delta^2_h u}{|h|^{\gamma}}\right\|_{L^{\beta_{i+1}}(B_{R_{i+1}+4h_0})}\leq  C\,\sup\limits_{0<|h|< h_0}\left(\left\|\dfrac{\delta^2_h u }{|h|^{\gamma}}\right\|_{L^{\beta_i}(B_{R_i+4h_0})}+1\right),\quad \mbox{ for } i=i_\infty,\ldots,i_\infty+j_\infty-2,
$$
and finally
$$
\sup_{0<|h|< {h_0}}\left\|\frac{\delta^2_h u}{|h|^{\gamma}}\right\|_{L^{\beta_{i_\infty+j_\infty}}(B_{3/4})}\leq C\sup_{0<|h|< h_0}\left(\left\|\frac{\delta^2_h u }{|h|^{\gamma}}\right\|_{L^{\beta_{i_\infty+j_\infty-1}}(B_{R_{i_\infty+j_\infty-1}+4h_0})}+1\right).
$$
Hence, recalling that $\gamma=1-\varepsilon$, we conclude
$$
\sup_{0<|h|< {h_0}}\left\|\frac{\delta^2_h u}{|h|^{1-\varepsilon}}\right\|_{L^{\beta_{i_\infty+j_\infty}}(B_{3/4})}\leq C(N,\delta,p,s).
$$
We are again in the position to repeat the arguments at the end of the proof of Theorem \ref{teo:1} (see \eqref{Neq} and \eqref{Neq2}) and use Theorem~\ref{thm:Nholder} with $\beta = 1-\varepsilon$, $q=\beta_{i_\infty+j_\infty}$ and $\alpha =\delta$ to obtain 
$$
[u]_{C^\delta(B_{1/2})}\leq C(N,\delta,p,s).
$$
This concludes the proof.
\end{proof}
\begin{oss}
\label{oss:flexibility2}
Under the assumptions of the previous theorem, by the same arguments as in Remark \ref{oss:flexibility} we may deduce the more flexible estimate that for every $0<\sigma<1$
\[
[u]_{C^\delta(B_{\sigma R}(x_0))}\leq \frac{C}{R^\delta}\,\left(\|u\|_{L^\infty(B_{R}(x_0))}+R^{s-\frac{N}{p}}\,[u]_{W^{s,p}(B_{R}(x_0))}+\mathrm{Tail}_{p-1,s\,p}(u;x_0,R)\right),
\]
with $C$ now depending on $\sigma$ as well (and blowing-up as $\sigma\nearrow 1$)
\end{oss}
\section{Improved H\"older regularity: the non-homogeneous case}\label{sec:last}
In this section we finally prove Theorem~\ref{mainthm} in its full generality. We already know from Theorem~\ref{teo:holderf} that $u$ is locally in $C^\beta$ {\it for some} $\beta$. We will show that the H\"older exponent can be enhanced up to the desired one, by means of a blow-up argument inspired by \cite{CS}.
\vskip.2cm
We start with the following
\begin{lm}[Stability in $L^\infty$]
\label{lm:stab}
Assume $p\ge 2$ and $0<s<1$. Let $\Omega\subset\mathbb{R}^N$ be an open and bounded set. For $f\in L^q_{\rm loc}(\Omega)$ with 
\[
\left\{\begin{array}{lr}
q>\dfrac{N}{s\,p},& \mbox{ if } s\,p\le N,\\
&\\
q\ge 1,& \mbox{ if }s\,p>N,
\end{array}
\right.
\]
we consider a local weak solution $u\in W^{s,p}_{\rm loc}(\Omega)\cap L^{p-1}_{s\,p}(\mathbb{R}^N)$ of the equation
\[
(-\Delta_p)^s u=f,\qquad \mbox{ in }\Omega.
\]
Let $B_{2R}(x_0)\Subset\Omega$ and assume that
$$
\|u\|_{L^\infty(B_R(x_0))}+\int_{\mathbb{R}^N\setminus B_R(x_0)} \frac{|u(x)|^{p-1}}{|x-x_0|^{N+s\,p}}\, dx\leq M\qquad \mbox{ and }\qquad \|f\|_{L^q(B_R(x_0))}\leq \eta.
$$
Suppose that $h\in X^{s,p}_u(B_R(x_0),B_{2R}(x_0))$ weakly solves
$$
\left\{\begin{array}{rcll}
(-\Delta_p)^s h &=& 0,& \mbox{ in }B_R(x_0),\\
h&=&u, &\mbox{ in }\mathbb{R}^N\setminus B_R(x_0).
\end{array}
\right.
$$
Then for any $0<\sigma <1$, there is $\tau_{M,R,\sigma}(\eta)$ such that 
\begin{equation}
\label{stable}
\|u-h\|_{L^\infty(B_{\sigma\,R}(x_0))}\leq \tau_{M,R,\sigma}(\eta)
\end{equation}
and $\tau_{M,R,\sigma}(\eta)$ converges to $0$ as $\eta$ goes to $0$.
\end{lm}
\begin{proof}
The existence of a bound of the form \eqref{stable} easily follows from the triangle inequality and the local $L^\infty$ estimate of Theorem~\ref{teo:loc_bound}. Let us prove that $\eta\mapsto\tau_{M,R,\sigma}(\eta)$ is infinitesimal.
\par
For simplicity, we assume $R=1$ and $x_0=0$. We argue by contradiction and assume that there exist two sequences $\{f_n\}_{n\in\mathbb{N}}\subset L^{q}(B_1)$ and $\{u_n\}_{n\in\mathbb{N}}$ such that
\[
\|u_n\|_{L^\infty(B_1)}+\int_{\mathbb{R}^N\setminus B_1} \frac{|u_n|^{p-1}}{|x|^{N+s\,p}}\, dx\leq M,\qquad \|f_n\|_{L^q(B_1)}\to 0,
\]
but
\[
\liminf_{n\to\infty} \|u_n-h_n\|_{L^\infty(B_{\sigma})}>0.
\]
We observe that by \eqref{prima} we have
\begin{equation}
\label{azzero}
\lim_{n\to\infty}[u_n-h_n]^p_{W^{s,p}(\mathbb{R}^N)}\le C\,\lim_{n\to\infty}\left(\int_{B_1} |f_n|^{q}\,dx\right)^\frac{p'}{q}=0.
\end{equation}
Moreover, by Theorem~\ref{teo:DKP} and Theorem~\ref{teo:holderf} we have that both sequences $\{u_n\}_{n\in\mathbb{N}}$ and $\{h_n\}_{n\in\mathbb{N}}$ have uniformly bounded $C^{0,\beta}$ seminorms in $B_{\sigma}$. They are also uniformly bounded in $L^\infty(B_{\sigma})$, thanks to Theorem~\ref{teo:loc_bound}. By using the Ascoli-Arzel\`a theorem, we may conclude that $u_n-h_n$ converges uniformly in $\overline{B_{\sigma}}$, up to a subsequence. By \eqref{azzero} we get that
\[
\lim_{n\to\infty} \|u_n-h_n\|_{L^\infty(B_{\sigma})}=0,
\]
which gives the desired contradiction.
\end{proof}
The next result is the crucial gateway to improving the H\"older exponent found in Theorem~\ref{teo:holderf}, provided $f$ is sufficiently summable. For simplicity, we state it ``at scale $1$'', then we will show how to use it to get the general result.
\begin{prop} 
\label{prop:caffsilv}
Let $p\ge 2$, $0<s<1$ and $q$ be such that
\[
\left\{\begin{array}{lr}
q>\dfrac{N}{s\,p},& \mbox{ if } s\,p\le N,\\
&\\
q\ge 1,& \mbox{ if }s\,p>N,
\end{array}
\right.
\]
We consider $\Theta=\Theta(N,s,p,q)$ the exponent defined in \eqref{theta}. 
For every $0<\varepsilon<\Theta$ there exists $\eta>0$ such that if $f\in L^q_{\rm loc}(B_4)$ and
\[
\|f\|_{L^q(B_1)}\leq \eta,
\] 
then every local weak solution $u\in W^{s,p}_{\rm loc}(B_4)\cap L^{p-1}_{s\,p}(\mathbb{R}^N)$ of the equation
\[
(-\Delta_p)^s u=f,\qquad \mbox{ in }B_4,
\]
that satisifes
\begin{equation}
\label{startup}
\|u\|_{L^\infty(B_1)}\leq 1,\qquad \int_{\mathbb{R}^N\setminus B_1}\frac{|u|^{p-1}}{|x|^{N+s\,p}}\, dx\leq 1.
\end{equation}
belongs to $C^{\Theta-\varepsilon}(\overline{B_{1/8}})$.
\end{prop}
\begin{proof} 
We divide the proof in two parts, for ease of readability.
\vskip.2cm\noindent
{\bf Part 1: Regularity at the origin}. Here we prove that for every $0<\varepsilon<\Theta$ and every $0<r<1/2$, there exists $\eta$ and a constant $C=C(N,s,p,\varepsilon)>0$ such that if $f$ and $u$ are as above, then we have
\[
\sup_{x\in B_r} |u(x)-u(0)|\leq C\,r^{\Theta-\varepsilon}.
\]
We may assume $u(0)=0$ without loss of generality. We fix $0<\varepsilon<\Theta$ and observe that it is sufficient to prove that there exist $\lambda<1/2$ and $\eta>0$ such that if $f$ and $u$ are as above, then
\begin{equation}
\label{eq:keq}
\sup_{B_{\lambda^k}}|u|\leq \lambda^{k\,(\Theta-\varepsilon)},\qquad \int_{\mathbb{R}^N\setminus B_1}\left|\frac{u(\lambda^k\, x)}{\lambda^{k\,(\Theta-\varepsilon)}}\right|^{p-1}\,|x|^{-N-s\,p}\, dx\leq 1,\qquad \mbox{ for every }k\in\mathbb{N}.
\end{equation}
Indeed, assume this were true. Then for every $0<r<1/2$, there would exist $k\in \mathbb{N}$ such that $\lambda^{k+1}< r\le \lambda^k$. From the first property in \eqref{eq:keq}, we would get
\[
\sup_{B_r} |u|\le \sup_{B_{\lambda^k}} |u|\le \lambda^{k\,(\Theta-\varepsilon)}=\frac{1}{\lambda^{\Theta-\varepsilon}}\,\lambda^{(k+1)\,(\Theta-\varepsilon)}\le C\,r^{\Theta-\varepsilon},
\]
as desired.
\par
In order to prove \eqref{eq:keq} we proceed by induction.
For $k=0$ this holds true by the assumptions in \eqref{startup}. 
Suppose \eqref{eq:keq} holds up to $k$, we now show that it also holds for $k+1$, provided that 
\[
\|f\|_{L^q(B_1)}\le \eta,
\]
with $\eta$ small enough, but independent of $k$.
Define
$$
w_k=\frac{u(\lambda^k x)}{\lambda^{k\,(\Theta-\varepsilon)}}.
$$
By the hypotheses
$$
\|w_k\|_{L^\infty(B_1)}\leq 1 \qquad \mbox{ and } \qquad \int_{\mathbb{R}^N\setminus B_1}\frac{|w_k|^{p-1}}{|x|^{N+s\,p}}\,dx\leq 1,
$$
Moreover
$$
(-\Delta_p)^s w_k(x) = \lambda^{k\,[s\,p\,-(\Theta-\varepsilon)(p-1)]}\,f(\lambda^k\, x)=:f_k(x), 
$$
so that  
$$\
\|f_k\|_{L^{q}(B_1)}=\lambda^{k(s\,p\,-(\Theta-\varepsilon)(p-1))}\lambda^{-\frac{N}{q}\,k}\,\left(\int_{B_{\lambda^k}}|f|^{q}\,dx\right)^{\frac{1}{q}}\leq \|f\|_{L^{q}(B_1)}\le \eta.
$$
Here we used the hypotheses on $f$ and the definition of $\Theta$, and again the fact that $\lambda<1/2$. 
Take $h_k\in X^{s,p}_{w_k}(B_1,B_2)$ to be the weak solution of 
$$
\left\{\begin{array}{rcll}
(-\Delta_p)^s h &=& 0,& \mbox{ in } B_1,\\
h&=&w_k,& \mbox{ in } \mathbb{R}^N\setminus B_1.
\end{array}
\right.
$$
By the stability estimate of Lemma \ref{lm:stab}, we have 
\[
\|w_k-h_k\|_{L^\infty(B_{3/4})}<\tau(\eta),
\]
with $\tau$ independent of $k$.
By using this, we then have the following estimate
\begin{equation}
\label{estimata}
\begin{split}
|w_k(x)|&\leq |w_k(x)-h_k(x)|+|h_k(x)-h_k(0)|+|h_k(0)-w_k(0)|\\
&\leq 2\,\tau(\eta)+[h_k]_{C^{\Theta-\varepsilon/2}(B_{1/2})}\,|x|^{\Theta-\varepsilon/2},\qquad\qquad \mbox{ for }x\in B_{1/2}.\\
\end{split}
\end{equation}
We also used that $h_k$ is $C^{\Theta-\varepsilon/2}(\overline{B_{1/2}})$ thanks to Theorem~\ref{teo:1higher}, which comes with the estimate
\[
[h_k]_{C^{\Theta-\varepsilon/2}(B_{1/2})}\le C\, \left(\|h_k\|_{L^\infty(B_{2/3})}+[h_k]_{W^{s,p}(B_{2/3})}+\mathrm{Tail}_{p-1,s\,p}(h_k;0,{2/3})\right),
\]
see Remark \ref{oss:flexibility2}.
We observe that the norms on the right-hand side are uniformly bounded, independently of $k$. Indeed, by the triangle inequality, Lemma \ref{lm:stab} and the hypothesis on $w_k$ we have
\[
\|h_k\|_{L^\infty(B_{2/3})}\le \|h_k\|_{L^\infty(B_{3/4})} \le \|h_k-w_k\|_{L^\infty(B_{3/4})}+\|w_k\|_{L^\infty(B_{3/4})}\le \tau(\eta)+1.
\]
As for the tail term, by using Lemma \ref{lm:tails}, the triangle inequality, the hypothesis on $w_k$ and \eqref{seconda} 
\[
\begin{split}
\mathrm{Tail}_{p-1,s\,p}(h_k;0,3/4)
%&\le C\,\mathrm{Tail}_{p-1,s\,p}(h_k;0,1)+C\,\|h_k\|_{L^{p-1}(B_1)}\\
&\le C\,\mathrm{Tail}_{p-1,s\,p}(w_k;0,3/4)+C\,\left(\int_{\mathbb{R}^N\setminus B_{3/4}} \frac{|h_k(y)-w_k(y)|^{p-1}}{|y|^{N+s\,p}}\,dy\right)^\frac{1}{p-1}\\
&\le C\,\mathrm{Tail}_{p-1,s\,p}(w_k;0,1)+C\,\|w_k\|_{L^{p-1}(B_1)}+C\,\|h_k-w_k\|_{L^{p-1}(B_1)}\\
%&+C\,\|h_k-w_k\|_{L^{p-1}(B_1)}+C\,\|w_k\|_{L^{p-1}(B_1)}\\
%&\le C+C\,\left(\int_{B_2\setminus B_1} \frac{|h_k(y)-w_k(y)|^{p-1}}{|y|^{N+s\,p}}\,dy\right)^\frac{1}{p-1}+C\,\|h_k-w_k\|_{L^{p-1}(B_1)}\\
%&\le C+C\,\|h_k-w_k\|_{L^{p-1}(B_2)}\\
&\le 2C+C\,\|f_k\|_{L^{q}(B_1)}\le C\,(2+\eta).
\end{split}
\]
We also used that $h_k=w_k$ outside $B_1$, by construction. Of course, the constant $C$ is universal, i.e. it depends on $N,s$ and $p$ only.
\par
Finally, by using a rescaled version of \cite[Lemma 4]{KKP2}, we can bound the Sobolev seminorm
\[
[h_k]^p_{W^{s,p}(B_{2/3})}\le C\,\left(\|h_k\|^p_{L^\infty(B_{3/4})}+\fint_{B_{3/4}} |h_k|^p\,dx+\mathrm{Tail}_{p-1,s\,p}(h_k,0,3/4)^p\right)\le C.
\]
Thus, we can infer that 
\[
C_1:=\sup_{k\in\mathbb{N}}\, [h_k]_{C^{\Theta-\varepsilon/2}(B_{1/2})}
\] 
is finite and then the estimate of \eqref{estimata} is uniform in $k$.
\par
We now consider the rescaled function
$$
w_{k+1}(x)=\frac{u(\lambda^{k+1}\, x)}{\lambda^{(k+1)\,(\Theta-\varepsilon)}}=\frac{w_k(\lambda\, x)}{\lambda^{\Theta-\varepsilon}}.
$$
By choosing $\eta$ so that $2\,\tau(\eta)<\lambda^\Theta$ and $\lambda$ small enough, we can transfer estimate \eqref{estimata} to $w_{k+1}$. Indeed, we have 
\[
\begin{split}
|w_{k+1}(x)|\leq 2\,\tau(\eta)\,\lambda^{\varepsilon-\Theta}+C_1\,\lambda^{\varepsilon/2}|x|^{\Theta-\varepsilon/2}\leq (1+C_1\,|x|^{\Theta-\varepsilon/2})\,\lambda^{\varepsilon/2},\qquad  x\in B_\frac{1}{2\lambda}.
\end{split}
\]
The previous estimate implies in particular that $\|w_{k+1}\|_{L^\infty(B_1)}\leq 1$ for $\lambda$ satisfying
\begin{equation}
\label{uno}
\lambda<\min\left\{\frac{1}{2},(1+C_1)^{-\frac{2}{\varepsilon}}\right\}.
\end{equation}
This information, rescaled back to $u$, is exactly the first part of \eqref{eq:keq} for $k+1$. As for the second part of \eqref{eq:keq}, the bound above of $|w_{k+1}|$ implies
\begin{equation}
\label{nonlocal1}
\begin{split}
\int_{B_{\frac{1}{2\lambda}}\setminus B_{1}} \frac{|w_{k+1}|^{p-1}}{|x|^{N+s\,p\,}}\,dx&\leq \lambda^{\varepsilon\, (p-1)/2}\int_{B_{\frac{1}{2\lambda}}\setminus B_{1}} \frac{(1+C_1\,|x|^{\Theta-\varepsilon/2})^{p-1}}{|x|^{N+s\,p}}\,dx\\
&\leq (1+C_1)^{p-1}\,\lambda^{\varepsilon\, (p-1)/2}\int_{B_{\frac{1}{2\lambda}}\setminus B_{1}} \frac{1}{|x|^{N+sp+(\varepsilon/2-\Theta)\,(p-1)}}\,dx\\
&\leq \frac{C_2}{s\,p-(\Theta-\varepsilon/2)\,(p-1)}\,\lambda^{\varepsilon\,(p-1)/2}.
\end{split}
\end{equation}
By a change of variables and using that $|w_k|\leq 1$ in $B_1$, we see that also
\begin{equation}
\label{nonlocal2}
\int_{B_{\frac{1}{\lambda}}\setminus B_{\frac{1}{2\lambda}}} \frac{|w_{k+1}|^{p-1}}{|x|^{N+s\,p\,}}\,dx=\lambda^{(\varepsilon-\Theta)\,(p-1)+s\,p}\,\int_{B_1\setminus B_\frac12} \frac{|w_k(x)|^{p-1}}{|x|^{N+s\,p}}\,dx\leq {C_3\,\lambda^{\varepsilon\,(p-1)/2}}.
\end{equation}
In addition, by the integral bound on $w_k$
\begin{equation}
\label{nonlocal3}
\int_{\mathbb{R}^N\setminus B_{\frac{1}{\lambda}}} \frac{|w_{k+1}(x)|^{p-1}}{|x|^{N+s\,p}}\,dx= \lambda^{(\varepsilon-\Theta)\,(p-1)+s\,p}\,\int_{\mathbb{R}^N\setminus B_1}\frac{|w_k(x)|^{p-1}}{|x|^{N+s\,p}}\,dx\leq \lambda^{\varepsilon\, (p-1)/2}.
\end{equation}
In both estimates, we have also used that $\lambda<1/2$ and the fact that
\[
(\varepsilon-\Theta)\,(p-1)+s\,p\ge \varepsilon\,\frac{p-1}{2}.
\]
We observe that the constants $C_2$ and $C_3$ depend on $N,s$ and $p$ only.
From \eqref{nonlocal1}, \eqref{nonlocal2} and \eqref{nonlocal3},
we get that the second part of \eqref{eq:keq} holds, provided that
$$
\left(\frac{C_2}{\varepsilon\,(p-1)}+C_3+1\right)\,\lambda^{\varepsilon\,(p-1)/2}\leq 1.
$$
By taking into account \eqref{uno}, we finally obtain that \eqref{eq:keq} holds true at step $k+1$ as well, provided that $\lambda$ and $\eta$ are chosen so that
\[
\lambda<\min\left\{\frac{1}{2},(1+C_1)^{-\frac{2}{\varepsilon}}, \left(\frac{C_2}{\varepsilon\,(p-1)}+C_3+1\right)^\frac{2}{\varepsilon\,(p-1)}\right\}\qquad \mbox{ and }\qquad \tau(\eta)<\frac{\lambda^\Theta}{2}.
\]
Then the iteration process is complete.
\vskip.2cm\noindent
{\bf Step 2: Regularity in a ball.} We now show how to obtain the desired regularity in the whole ball $B_{1/8}$. We choose $0<\varepsilon<\Theta$ and take the corresponding $\eta$, obtained in {\bf Step 1}. Let $z_0\in B_{1/2}$, we set $L=2^{N+1}\,(1+|B_1|)$ and define
$$
v(x):=L^{-\frac{1}{p-1}}\,u\left(\frac{x}{2}+z_0\right),\qquad x\in \mathbb{R}^N.
$$
We observe that $v\in W^{s,p}_{\rm loc}(B_4)\cap L^{p-1}_{s\,p}(\mathbb{R}^N)$ and it is
a local weak solution in $B_4$ of 
\[
(-\Delta_p)^s v=\frac{2^{-s\,p}}{L}\,f\left(\frac{x}{2}+z_0\right)=:\widetilde f,
\]
with
\[
\left\|\widetilde f\right\|_{L^{q}(B_1)}=\frac{2^{N/q-s\,p}}{L}\,\|f\|_{L^{q}(B_{1})}\le \frac{2^{N/q-s\,p}}{L}\,\eta<\eta.
\]
By construction, we also have
\[
\|v\|_{L^\infty(B_1)}\leq 1,
\] 
and since $B_{1/2}(z_0)\subset B_1$
\[
\begin{split}
\int_{\mathbb{R}^N\setminus B_1}\frac{|v(x)|^{p-1}}{|x|^{N+s\,p}}\, dx&=\frac{2^{-s\,p}}{L}\,\int_{\mathbb{R}^N\setminus B_{1/2}(z_0)}\frac{|u(y)|^{p-1}}{|y-z_0|^{N+s\,p}}\,dy\\
&\le \frac{1}{L}\,\left(\frac{1}{2}\right)^{s\,p}\,\left(\frac{1}{1-|z_0|}\right)^{N+s\,p}\,\int_{\mathbb{R}^N\setminus B_1}\frac{|u(y)|^{p-1}}{|y|^{N+s\,p}}\,dy+\frac{2^{N}}{L}\,\|u\|^{p-1}_{L^{p-1}(B_1)}\\
&\leq \frac{2^{N}}{L}\,\int_{\mathbb{R}^N\setminus B_1}\frac{|u(y)|^{p-1}}{|y|^{N+s\,p}}dy+\frac{2^N\,|B_1|}{L}\,\|u\|^{p-1}_{L^\infty(B_1)}\leq 1,
\end{split}
\]
thanks to the definition of $L$ and assumptions \eqref{startup}.
Here we have used Lemma \ref{lm:tails} with the balls $B_{1/2}(z_0)\subset B_1$. We may therefore apply {\bf Step 1} to $v$ and obtain
$$
\sup_{x\in B_r}|v(x)-v(0)|\leq C\,r^{\Theta-\varepsilon},\quad 0<r<\frac{1}{2}.
$$
In terms of $u$ this is the same as
\begin{equation}
\label{supest}
\sup_{x\in B_r(z_0)}|u(x)-u(z_0)|\leq C\,L^\frac{1}{p-1}\,r^{\Theta-\varepsilon},\qquad 0<r<\frac{1}{4}.
\end{equation}
We note that this holds for any $z_0\in B_{1/2}$. Now take any pair $x,y\in B_{1/8}$ and set $|x-y|= r$. We observe that $r<1/4$ and we set $z=(x+y)/2$. Then we apply \eqref{supest} with $z_0=z$ and obtain
\[
\begin{split}
|u(x)-u(y)|\leq |u(x)-u(z)|+|u(y)-u(z)|&\leq 2\sup_{w\in B_r(z)}|u(w)-u(z)|\\
&\leq 2\,C\,L^\frac{1}{p-1}\,r^{\Theta-\varepsilon}=2\,C\,L^\frac{1}{p-1}\,|x-y|^{\Theta-\varepsilon},
\end{split}
\]
which is the desired result.
\end{proof}
\begin{comment}
\begin{coro}\label{homocor} Assume the hypotheses of Proposition \ref{prop:caffsilv}. Assume in addition that 
$$
2^{N}\,|B_1|\,\|u\|_{L^\infty(B_1)}\leq \frac12\qquad \mbox{ and }\qquad 2^{N}\,\int_{\mathbb{R}^N\setminus B_1}\frac{|u|^{p-1}}{|x|^{N+s\,p}}\, dx\leq \frac12. 
$$
Then $u\in C^{\alpha-\varepsilon}(\overline{B_{1/2}})$ and in particular
$$
\sup_{x,y\in B_\frac12}\frac{|u(x)-u(y)|}{|x-y|^{\alpha-\varepsilon}}\leq C.
$$
\end{coro}
\begin{proof} 
\end{proof}
\end{comment}
We can finally prove Theorem \ref{mainthm}. 
\begin{comment}

\begin{teo}
Let $p>2$, $0<s<1$ and $\alpha=\min((sp-N/q)/(p-1),1)$.  Let $u\in W^{s,p}_{\rm loc}(\Omega)\cap L^{p-1}_{s\,p}(\mathbb{R}^N)$ be a local weak solution of 
$$
(-\Delta_p)^s u = f \qquad \mbox{ in }\Omega,
$$
where $f\in L_\text{loc}^{q}(\Omega)$. Then $u\in C^\delta_{\rm loc}(\Omega)$ for every $0<\delta<\alpha$. 
\par
More precisely, for every $0<\delta<\alpha$ and every ball $B_{4R}(x_0)\Subset\Omega$, there exists a constant $C=C(N,s,p,\delta)>0$ such that 
$$
[u]_{C^{\delta}(B_{R/8}(x_0))}\leq \frac{C}{R^{\delta}}\,\left(\|u\|_{L^\infty(B_{R}(x_0))}+\mathrm{Tail}_{p-1,s\,p}(u;x_0,R)+R^{s}\,\|f\|^\frac{1}{p-1}_{L^{N/s}(B_{R})}\right).
$$
\end{teo}
\end{comment}
\begin{proof}[~Proof of Theorem~\ref{mainthm}] We may assume $x_0=0$ without loss of generality. We modify $u$ so that it fits into the setting of Proposition \ref{prop:caffsilv}. We choose $0<\delta<\Theta$, take $\eta$ as in Proposition \ref{prop:caffsilv} with the choice $\varepsilon=\Theta-\delta$ and set
\[
\mathcal{A}_R=\|u\|_{L^\infty(B_{R})}+\left(R^{s\,p}\,\int_{\mathbb{R}^N\setminus B_{R}}\frac{|u(y)|^{p-1}}{|y|^{N+s\,p}}\,  dy\right)^\frac{1}{p-1}+\left(\frac{R^{s\,p-N/q}\|f\|_{L^{q}(B_{R})}}{\eta}\right)^\frac{1}{p-1}.
\] 
As in the proof of Theorem~\ref{teo:1}, we observe that it is sufficient to prove that the rescaled function
\[
u_R(x):=\frac{1}{\mathcal{A}_R}\,u(R\,x),\qquad \mbox{ for }x\in B_4,
\]
satifies the estimate
\[
[u_R]_{C^{\delta}(B_{1/8})}\leq C.
\]
It is easily seen that the choice of $\mathcal{A}_R$ implies
$$
\|u_R\|_{L^\infty(B_1)}\leq 1,\qquad \int_{\mathbb{R}^N\setminus B_1}\frac{|u_R|^{p-1}}{|x|^{N+s\,p}}\, dx\leq 1. 
$$
In addition, $u_R$ is a local weak solution of
$$
(-\Delta_p)^s u_R\, (x) = \frac{R^{s\,p}}{\mathcal{A}_R^{p-1}}\,f(R\,x):= f_R(x),\qquad x\in B_4,
$$
with $\|f_R\|_{L^{q}(B_{1})}\le \eta$.
We may apply Proposition \ref{prop:caffsilv} with $\varepsilon=\Theta-\delta$ to $u_R$ and obtain
\[
[u_R]_{C^{\delta}(B_{1/8})}\leq C.
\]
By scaling back, this concludes the proof.
\end{proof}

\appendix
\section{Pointwise inequalities}\label{sec:app}

In this section, we list the pointwise inequalities used throughout the whole paper.

\begin{lm}
Let $p\ge 2$ and $q>1$. Then for every $A,B\in\mathbb{R}$ we have
\begin{equation}
\label{monotone}
|A-B|^{q-2}\,\Big(J_p(A)-J_p(B)\Big)\,(A-B)\ge (p-1)\,\left(\frac{q}{p-2+q}\right)^q\, \left||A|^\frac{p-2}{q}\, A-|B|^\frac{p-2}{q}\, B\right|^q.
\end{equation}
\end{lm}
\begin{proof}
Since $J_p(A)-J_p(B)$ and $A-B$ share the same sign, we can assume without loss of generality that $A\ge B$.
If $A=B$ there is nothing to prove.
Let us assume that $A>B$, then we have
\[
\begin{split}
(A-B)^{q-2}\,\Big(J_p(A)-J_p(B)\Big)\,(A-B)&=(p-1)\,\left(\int_B^A |t|^{p-2}\,dt\right)\,(A-B)^{q-1}\\
&=(p-1)\, \left(\int_B^A |t|^{\frac{p-2}{q}\,q}\,dt\right)\,(A-B)^{q-1}\\
&\ge (p-1)\, \left(\int_B^A |t|^\frac{p-2}{q}\,dt\right)^q\\
&=(p-1)\,\left(\frac{q}{p-2+q}\right)^q\,\left||A|^\frac{p-2}{q}\, A-|B|^\frac{p-2}{q}\, B\right|^q, 
\end{split}
\]
which concludes the proof.
\end{proof}
\begin{lm}
Let $p\ge 2$. Then for every $A,B\in\mathbb{R}$ we have
\begin{equation}
\label{lipschitz}
\Big|J_p(A)-J_p(B)\Big|\le \frac{2\,(p-1)}{p}\, \left(|A|^\frac{p-2}{2}+|B|^\frac{p-2}{2}\right)\,\left||A|^\frac{p-2}{2}\,A-|B|^\frac{p-2}{2}\,B\right|.
\end{equation}
\end{lm}
\begin{proof}
For $A=B$ there is nothing to prove. Let us consider the case $A\not=B$. Without loss of generality, we can suppose that $A>B$. 
We first observe that if we set 
\[
\gamma=1+\frac{2\,(p-1)}{p},
\]
we can rewrite
\[
J_p(A)=J_\gamma\left(|A|^\frac{p-2}{2}\,A\right)\qquad \mbox{ and }\qquad J_p(B)=J_\gamma\left(|B|^\frac{p-2}{2}\,B\right).
\]
By basic calculus we thus have
\[
\begin{split}
J_p(A)-J_p(B)&=J_\gamma\left(|A|^\frac{p-2}{2}\,A\right)-J_\gamma\left(|B|^\frac{p-2}{2}\,B\right)\\
&=J'_\gamma(\xi)\,\left(|A|^\frac{p-2}{2}\,A-|B|^\frac{p-2}{2}\,B\right)\\
&\le \max\left\{J'_\gamma\left(|A|^\frac{p-2}{2}\,A\right),\, J'_\gamma\left(|B|^\frac{p-2}{2}\,B\right)\right\}\,\Big(|A|^\frac{p-2}{2}\,A-|B|^\frac{p-2}{2}\,B\Big).\\
\end{split}
\]
This concludes the proof.
\end{proof}
\begin{lm}
Let $\gamma\ge 1$. Then for every $A,B\in\mathbb{R}$ we have
\begin{equation}
\label{holder}
\left||A|^{\gamma-1}\,A-|B|^{\gamma-1}\,B\right|\ge \frac{1}{C}\, |A-B|^\gamma,
\end{equation}
for some constant $C=C(\gamma)>0$.
\end{lm}
\begin{proof}
When $\gamma=1$, there is nothing to prove.
We take $\gamma>1$ and observe that if $A=0$ or $B=0$, the result trivially holds. Thus let us suppose that $A\,B\not=0$ and observe that the function $F:\mathbb{R}\to\mathbb{R}$ defined by
\[
F(t)=|t|^\frac{1-\gamma}{\gamma}\,t,
\]
is $(1/\gamma)-$H\"older continuous. More precisely, we have
\[
|F(t)-F(s)|\le C\,|t-s|^\frac{1}{\gamma},\qquad t,s\in\mathbb{R}.
\]
By applying the previous with
\[
t=|A|^{\gamma-1}\, A\qquad \mbox{ and }\qquad s=|B|^{\gamma-1}\,B,
\]
we get the conclusion.
\end{proof}
\begin{oss}\label{usual}
By combining \eqref{monotone} with $q=2$ and \eqref{holder} with $\gamma=p/2$, we also get
\begin{equation}
\label{xxx}
\Big(J_p(A)-J_p(B)\Big)\,(A-B)\ge \frac{1}{C}\, |A-B|^p.
\end{equation}
\end{oss}
\begin{lm}
Let $p\ge 2$, $\gamma\ge 1$ and $a,b,c,d\in\mathbb{R}$. Then we have
\begin{equation}
\label{erik}
\begin{split}
\Big(J_p(a-c)-J_p(b-d)\Big)&\Big(J_{\gamma+1}(a-b)-J_{\gamma+1}(c-d)\Big)\\
&\ge \frac{1}{C}\, \Big||a-b|^\frac{\gamma-1}{p}\,(a-b)-|c-d|^\frac{\gamma-1}{p}\,(c-d)\Big|^p,
\end{split}
\end{equation}
for some $C=C(p,\gamma)>0$.
\end{lm}
\begin{proof}
Of course, we can assume that $a-c>b-d$ without loss of generality.
We first observe that we can write
\begin{equation}
\label{erik1}
\begin{split}
\Big(J_p(a-c)-J_p(b-d)\Big)&\Big(J_{\gamma+1}(a-b)-J_{\gamma+1}(c-d)\Big)\\
&=\Big(J_p(a-c)-J_p(b-d)\Big)\\
&\times\frac{\Big(J_{\gamma+1}(a-b)-J_{\gamma+1}(c-d)\Big)}{\Big((a-b)-(c-d)\Big)^{p-1}}\,\Big((a-b)-(c-d)\Big)^{p-1}.
\end{split}
\end{equation}
As for the first term on the right-hand side, by \eqref{holder} for $\gamma=p-1$ we have
\[
\Big(J_p(a-c)-J_p(b-d)\Big)\ge \frac{1}{C}\, \Big|a-c-b+d\Big|^{p-1},
\]
while by using \eqref{monotone} with $p=\gamma+1$ and $q=p$, we have 
\[
\begin{split}
\frac{\Big(J_{\gamma+1}(a-b)-J_{\gamma+1}(c-d)\Big)}{\Big((a-b)-(c-d)\Big)^{p-1}}&\Big((a-b)-(c-d)\Big)^{p-1}\\
&\ge \gamma\,\left(\frac{p}{\gamma-1+p}\right)^2\,\left||a-b|^\frac{\gamma-1}{p}\,(a-b)-|c-d|^\frac{\gamma-1}{p}\,(c-d)\right|^p\\
&\times \Big|a-b-c+d\Big|^{1-p}.
\end{split}
\]
By using these estimates in \eqref{erik1}, we finally get \eqref{erik}.
\end{proof}
\begin{lm}
Let $p\ge 2$, $\gamma\ge 1$ and $a,b,c,d\in\mathbb{R}$. Then we have
\begin{equation}
\label{erik2}
\begin{split}
\Big(J_p(a-b)-J_p(c-d)\Big)&\,\Big(J_{\gamma+1}(a-c)-J_{\gamma+1}(b-d)\Big)\\
&\ge \frac{2\,(p-1)}{p^2}\, \left||a-b|^\frac{p-2}{2}\,(a-b)-|c-d|^\frac{p-2}{2}\,(c-d)\right|^2\\
&\times\left(|a-c|^{\gamma-1}+|b-d|^{\gamma-1}\right).
\end{split}
\end{equation}
\end{lm}
\begin{proof} 
We first observe that if $a-b=c-d$, then there is nothing to prove. Thus we assume $a-b\not =c-d$.
By following \cite[Formula I, page 71]{Petersnotes}, for every $x,y\in\mathbb{R}$ we have
$$
\left(J_{\gamma+1}(y)-J_{\gamma+1}(x)\right)(y-x)=\frac12\left(|y|^{\gamma-1}+|x|^{\gamma-1}\right)(y-x)^2+\frac{|y|^{\gamma-1}-|x|^{\gamma-1}}{2}\,(y^2-x^2).
$$
By observing that the last term is nonnegative and choosing $y=a-c$ and $x=b-d$, we get
\[
\left(J_{\gamma+1}(a-c)-J_{\gamma+1}(b-d)\right)\,(a-c-(b-d))\ge \frac{1}{2}\,\left(|a-c|^{\gamma-1}+|b-d|^{\gamma-1}\right)\,(a-c-(b-d))^2.
\]
By using \eqref{monotone} with 
\[
A=a-b,\qquad B=c-d\qquad \mbox{ and }\qquad q=2,
\]
we have
\[
\left(J_p(a-b)-J_p(c-d)\right)\left(a-b-(c-d)\right)\geq (p-1)\,\frac{4}{p^2}\,\left||a-b|^\frac{p-2}{2}\,(a-b)-|c-d|^\frac{p-2}{2}\,(c-d)\right|^2.
\]
Multiplication of the two above inequalities yields
\begin{align*}
&\left(J_{\gamma+1}(a-c)-J_{\gamma+1}(b-d)\right)\left(J_p(a-b)-J_p(c-d)\right)\left(a-b-(c-d)\right)^2\\
&\geq (p-1)\,\frac{2}{p^2}\,\left(|a-c|^{\gamma-1}+|b-d|^{\gamma-1}\right)\,(a-c-(b-d))^2\Big||a-b|^\frac{p-2}{2}(a-b)-|c-d|^\frac{p-2}{2}(c-d)\Big|^2.
\end{align*}
By simplifying the factor $\left(a-b-(c-d)\right)^2$ in both sides, we get the desired inequality.
\end{proof}

\end{document}